\documentclass{article}
\usepackage{graphicx} % Required for inserting images
\usepackage{amsmath, amsthm, amssymb}
\usepackage{amsfonts}
\usepackage{bm}
\usepackage[margin=1in]{geometry}
\usepackage{hyperref}
\usepackage{enumitem}
\usepackage{subcaption}
\usepackage{csquotes}

\newtheorem{theorem}{Theorem}[section]
\newtheorem{corollary}{Corollary}[theorem]
\newtheorem{lemma}[theorem]{Lemma}

\newtheorem{definition}{Definition}
\newtheorem{conjecture}{Conjecture}
\newtheorem{remark}{Remark}

\usepackage{algorithm}
\usepackage{algpseudocode}

\usepackage[round]{natbib}
\makeatletter
\def\th@plain{%
  \thm@notefont{}% same as heading font
  \itshape % body font
}
\def\th@definition{%
  \thm@notefont{}% same as heading font
  \normalfont % body font
}

\theoremstyle{definition}

\makeatother

\renewcommand{\Im}{\operatorname{Im}}
\DeclareMathOperator*{\argmin}{argmin} 

\title{On the Effect of Misspecifying the Embedding Dimension in Low-rank Network Models} 
\author{Roddy Taing and Keith Levin \\
{\normalsize Department of Statistics, University of Wisconsin--Madison}}
\date{\vspace{-5ex}}

\begin{document}
\newcommand{\R}{\mathbb{R}}
\newcommand{\C}{\mathbb{C}}
\newcommand{\bbN}{\mathbb{N}}
\newcommand{\bbO}{\mathbb{O}}

\newcommand{\ts}{\textsuperscript}

\newcommand{\sm}{\setminus}

\newcommand{\submatrix}[2]{\left( #1 \right)^{\setminus #2}}

\newcommand{\norm}[1]{\left\lVert#1\right\rVert}
\newcommand{\average}[1]{\langle #1 \rangle}
\newcommand{\RDPG}{\operatorname{RDPG}}
\newcommand{\GRDPG}{\operatorname{GRDPG}}
\newcommand{\as}{~\text{ a.s.}}

\newcommand{\twoinfnorm}[1]{\left\lVert#1\right\rVert_{2,\infty}}
\newcommand{\mf}[1]{\mathbf{#1}}
\newcommand{\hmf}[1]{\hat{\mathbf{#1}}}
\newcommand{\E}[1]{\mathbb{E}\left[ #1 \right]}
\newcommand{\Prob}[1]{\mathbb{P}\left( #1 \right)}
\newcommand{\abs}[1]{\left| #1 \right|}

\newcommand{\Op}{O_{\mathbb{P}}}

\newcommand{\uhat}{\hat u}
\newcommand{\uhatt}{\hat{u}^{(\sm \T)}}
\newcommand{\vhatt}{\hat{v}^{(\sm \T)}}
\newcommand{\lambdahat}{\hat\lambda}
\newcommand{\lambdahatt}{\hat\lambda^{(\sm \T)}}
\newcommand{\muhat}{\hat\mu}
\newcommand{\muhatt}{\hat\mu^{(\sm \T)}}

\newcommand{\vhat}{\hat v}
\newcommand{\calX}{\mathcal{X} }

\newcommand{\latent}{\mathbf{X}}

\newcommand{\psitilde}{\tilde{\psi}}
\newcommand{\iu}{\mathrm{i}\mkern1mu}

\newcommand{\msc}{m_{\operatorname{sc}}}
\newcommand{\Tr}{\operatorname{Tr}}

\newcommand{\T}{\mathbb{T}}
\newcommand{\EMt}[1]{u_{i#1}(\bm{u}_{#1})^{\setminus\T}}
\newcommand{\EMtj}[1]{u_{j#1}(\bm{u}_{#1})^{\setminus\T}}

\newcommand{\EM}[1]{u_{i#1}(\bm{u}_{#1})^{\setminus i}}
\newcommand{\EMij}[2]{u_{#2#1}(\bm{u}_{#1})^{\setminus ij}}

\newcommand{\VMt}[1]{v_{i#1}(\bm{v}_{#1})^{\setminus\T}}
\newcommand{\VMtj}[1]{v_{j#1}(\bm{v}_{#1})^{\setminus\T}}
\newcommand{\VM}[1]{v_{i#1}(\bm{v}_{#1})^{\setminus i}}
\newcommand{\VMij}[2]{v_{#2#1}(\bm{v}_{#1})^{\setminus ij}}

\newcommand{\Gtilde}{\tilde{G}}
\newcommand{\Gtildematrix}{\tilde{\mathbf{G}}}
\newcommand{\Hi}{(\mathbf{H_{.i}})^{\setminus i}}
\newcommand{\Ai}{(\mathbf{A_{.i}})^{\setminus i}}
\newcommand{\Bi}{(\mathbf{B_{.i}})^{\setminus i}}
\newcommand{\Ht}{(\mathbf{H_{.i}})^{\setminus\T}}
\newcommand{\Hij}{(\mathbf{H})^{\setminus ij}}
\newcommand{\Gtildei}{\tilde{G}^{(\setminus i)}}
\newcommand{\Gtildeimatrix}{\tilde{\mathbf{G}}^{(\setminus i)}}
\newcommand{\Gtildeij}{\tilde{G}^{(\setminus ij)}}
\newcommand{\Gtildeijmatrix}{\tilde{\mathbf{G}}^{(\setminus ij)}}
\newcommand{\Gtildet}{\tilde{G}^{(\setminus \T)}}
\newcommand{\Gtildetmatrix}{\tilde{\mathbf{G}}^{(\setminus \T)}}
\newcommand{\Ot}{\tilde{\Omega}}
\newcommand{\Ld}{\tilde{\Lambda}_d}
\newcommand{\Lo}{\tilde{\Lambda}_o}
\newcommand{\Lam}{\tilde{\Lambda}}

\newcommand{\calAtilde}{\tilde{\mathcal{A}}}
\newcommand{\calB}{\mathcal{B}}
\newcommand{\calS}{\mathcal{S}}

\newcommand{\et}[1]{(\bm{e}_{#1})^{\setminus \T}}

\newcommand{\vt}[1]{(\bm{v}_{#1})^{\setminus \T}}
\newcommand{\be}{\bm{e}}
\newcommand{\buhat}{\hat{\bm{u}}}
\newcommand{\buhatt}{\hat{\bm{u}}^{(\sm \T)}}
\newcommand{\bu}{\bm{u}}

\newcommand{\bvhat}{\hat{\bm{v}}}
\newcommand{\bvhatt}{\hat{\bm{v}}^{(\sm \T)}}
\newcommand{\bv}{\bm{v}}

\newcommand{\ba}{\bm{a}}

\newcommand{\Ytilde}{\tilde{Y}}
\newcommand{\Ztilde}{\tilde{Z}}
\newcommand{\mtilde}{\tilde{m}}
\newcommand{\mtildet}{\tilde{m}^{(\sm \T)}}
\newcommand{\ntilde}{\tilde{n}}
\newcommand{\rhotilde}{\tilde{\rho}}

\newcommand{\Lambdatilde}{\tilde{\Lambda}}
\newcommand{\mBtilde}{\widetilde{\mathbf{B}}}
\newcommand{\mLtilde}{\widetilde{\mathbf{L}}}
\newcommand{\mWtilde}{\widetilde{\mathbf{W}}}
\newcommand{\mXtilde}{\widetilde{\mathbf{X}}}

\newcommand{\lambdamin}{\lambda_{\min}}
\newcommand{\lambdamax}{\lambda_{\max}}

\newcommand{\Vperp}{\mathbf V_{\perp}}
\newcommand{\Utildeperp}{\widetilde{\mathbf U}_{\perp}}
\newcommand{\Uhatperp}{\widehat{\mathbf U}_{\perp}}
\newcommand{\mUhattildeperp}{{\widetilde{\widehat{\mathbf U}}}_{\perp}}
\newcommand{\Mperp}{\mathbf M_{\perp}}
\newcommand{\mMtildeperp}{\widetilde{\mathbf{M}}_{\perp}}
\newcommand{\Mhatperp}{\widehat{\mathbf M}_{\perp}}

\newcommand{\mMhattildeperp}{{\widetilde{\widehat{\mathbf M}}}_{\perp}}
\newcommand{\valpha}{\bm{v}_{\alpha}}

\newcommand{\deloc}{(\log N)^{\gamma}}
\newcommand{\delocsq}{(\log N)^{2\gamma}}
\newcommand{\logN}{(\log N)}

\newcommand{\xinu}{(\xi, \nu)}

\newcommand{\tti}{2,\infty}

\newcommand{\mA}{\mathbf{A}}
\newcommand{\mB}{\mathbf{B}}
\newcommand{\mD}{\mathbf{D}}
\newcommand{\mE}{\mathbf{E}}
\newcommand{\mI}{\mathbf{I}}
\newcommand{\mP}{\mathbf{P}}
\newcommand{\mR}{\mathbf{R}}
\newcommand{\mS}{\mathbf{S}}
\newcommand{\mU}{\mathbf{U}}
\newcommand{\mW}{\mathbf{W}}
\newcommand{\mH}{\mathbf{H}}
\newcommand{\mQ}{\mathbf{Q}}
\newcommand{\mZ}{\mathbf{Z}}
\newcommand{\mV}{\mathbf{V}}
\newcommand{\mX}{\mathbf{X}}
\newcommand{\mG}{\mathbf{G}}
\newcommand{\mGtilde}{\tilde{\mathbf{G}}}
\newcommand{\mF}{\mathbf{F}}
\newcommand{\boldm}{\mathbf{m}}

\newcommand{\mSigma}{\mathbf{\Sigma}}

\newcommand{\bmA}{\bm{A}}
\newcommand{\bmP}{\bm{P}}
\newcommand{\bmE}{\bm{E}}

\newcommand{\shat}{\hat{s}}
\newcommand{\spop}{s}

\newcommand{\mShat}{\mathbf{\hat S}}
\newcommand{\mUhat}{\mathbf{\hat U}}
\newcommand{\mXhat}{\mathbf{\hat X}}

\newcommand{\mL}{\mathbf{L}}

\newcommand{\rhon}{\rho_N}
\newcommand{\iid}{\text{i.i.d}}

\newcommand{\Dir}{\operatorname{Dir}}
\newcommand{\Exp}{\operatorname{Exp}}
\newcommand{\rank}{\operatorname{rank}}
\newcommand{\Laplace}{\operatorname{Laplace}}
\newcommand{\Pois}{\operatorname{Pois}}
\newcommand{\Bernoulli}{\operatorname{Bernoulli}}

\newcommand{\Deltatilde}{\Delta^{\circ}}

\newcommand{\mutilde}{\tilde{\mu}}
\newcommand{\etan}{\phi_n}
\newcommand{\supp}{supp}

\maketitle
%% Check Bibliography has the correct capitalization (e.g., U.S. House of Representatives)

%When feeling tired: \hyperlink{https://www.youtube.com/watch?v=Jrxl3Op4rsI}{Meeseeks}

\begin{abstract} 
As network data has become ubiquitous in the sciences, there has been growing interest in network models whose structure is driven by latent node-level variables in a (typically low-dimensional) latent geometric space. 
These ``latent positions'' are often estimated via embeddings, whereby the nodes of a network are mapped to points in Euclidean space so that ``similar'' nodes are mapped to nearby points.
Under certain model assumptions, these embeddings are consistent estimates of the latent positions, but most such results require that the embedding dimension be chosen correctly, typically equal to the dimension of the latent space. 
Methods for estimating this correct embedding dimension have been studied extensive in recent years, but there has been little work to date characterizing the behavior of embeddings when this embedding dimension is misspecified. 
In this work, we provide theoretical descriptions of the effects of misspecifying the embedding dimension of the adjacency spectral embedding under the random dot product graph, a class of latent space network models that includes a number of widely-used network models as special cases, including the stochastic blockmodel.
We consider both the case in which the dimension is chosen too small, where we prove estimation error lower-bounds, and the case where the dimension is chosen too large, where we show that consistency still holds, albeit at a slower rate than when the embedding dimension is chosen correctly.
A range of synthetic data experiments support our theoretical results.
Our main technical result, which may be of independent interest, is a generalization of earlier work in random matrix theory, showing that all non-signal eigenvectors of a low-rank matrix subject to additive noise are delocalized.
\end{abstract}

\section{Introduction} \label{sec:intro}
Networks, which describe collections of interacting entities, have emerged as central to a range of applications in the past twenty years.
As mathematical objects, networks are graphs, in which vertices (also called nodes; we will use the terms interchangeably) correspond to the entities under study, and edges describe interactions among these entities.
Networks arise in genomics
\citep{gene_expression,gene_regulatory_networks,transcription_regulatory_network,parasitology}, where they describe which pairs of genes are co-expressed or are involved in the same biological processes.
In neuroscience, networks describe connectomes \citep{bullmore_brain,sporns_book}, in which vertices correspond to brain regions and edges encode which pairs of brain regions interact closely.
Networks are used in ecology to describe food webs and species co-occurrence \citep{soil_networks}.
In the social sciences, networks arise in  disciplines political science \citep{politics} to sociology \citep{Lazega2001} to economics \citep{Chaney2014,Acemoglu2015,EllGolJac2014}.

%The prevalence of networks across scientific disciplines has been met with a surge of interest from the statistics community.

With the growing prevalence of network data, there has been a proliferation of statistical models for networks, beginning with the stochastic blockmodel \citep[SBM;][]{SBM_original}.
Under this model, each vertex belongs to one of $r$ communities, and the probability of edge formation between any pair of vertices is determined by their community memberships.
The SBM has been extended in a number of ways to allow for degree heterogeneity \citep{dcSBM_original} and mixed community memberships \citep{MMSBM}.
These aforementioned models make the assumption that network structure is driven by latent node-level variables (e.g., community memberships in the SBM).
Subsequent models have extended this idea to allow for geometric latent structure, in which each node has an associated {\em latent position} in $r$-dimensional Euclidean space.
These models are often called latent space network models \cite[though this term is sometimes reserved for the family of models initiated in][]{HofRafHan2002}.
The most prominent examples include the random geometric graph \citep{random_geo_graph_original, random_geo_graph_Lie_group, garcia_trillos_error_2020}, wherein two nodes form an edge if the distance between their latent positions is below a set threshold; and the random dot product graph \citep[RDPG;][]{original_RDPG_Young, RDPG_survey,GRDPG_original}, in which edge probabilities are determined by inner products between latent positions.
These aforementioned models are all special cases of the graphon \citep{graphon_original, graphon_conference, Lei2021graphon, graphon_book}.

A typical statistical task under the latent space models described above is to estimate the latent positions of the nodes from an observed network.
A common approach is to use \textit{node embeddings}, which map the nodes of a network to a low-dimensional (typically Euclidean) space, in such a way that ``similar'' nodes are mapped to nearby points.
Examples of node embeddings, corresponding to different notions of node similarity, include \cite{ASE_spectral_clustering_original,Belkin,Rohe_clustering,Tang-Laplacian,node2vec}, to name but a few.
Node embeddings are often used to obtain Euclidean representations of vertices that can then be used in downstream tasks such as clustering \citep{Rohe_clustering, ASE_spectral_clustering_original}, inference \citep{Tang_inference, Zheng_inference}, and visualization \citep{visualization}, though our focus in the present work will be embeddings as estimators of latent positions.

The first step in constructing node embeddings is to choose the embedding dimension.
The vast majority of results showing consistent estimation of latent positions assume that this embedding dimension is chosen correctly \citep[see, e.g.,][]{RDPG_survey,Lei_community_consistency,cape_signal-plus-noise_2019,Rohe_clustering,UASE,levin_central_2019,levin_recovering_nodate,ASE_spectral_clustering_original,singular_subspace,consistent_dot_product_SBM, GRDPG_original}, though see \cite{Lei2021graphon,TanCap2025}, which instead rely on spectral properties of the latent structure.
This latent dimension is typically, though not always \citep[see, e.g.,][]{patrick_neurips, AthTanParPri2021, TanCapPri2022, Sansford2025}, equal to the rank of the expected adjacency matrix conditional on the latent positions.
Of course, in practice, the embedding dimension must be chosen based on data.
This model selection problem is widely recognized by the statistical network analysis community as important and challenging, and it has received attention commensurate with this importance.
The earliest methods relied on finding elbows the scree plot of the adjacency matrix \citep{scree_plot_rank}.
More recent approaches include cross-validation and maximum likelihood approaches \cite{li_network_2020, cross-valid_jing_lei, cross_validated_rank, Yang_dimension_selection, Jiang_rank_estimation, Bickel_model_selection_SBM}, parallel analysis \citep{hong_rank, hong_parallel_analysis}, hypothesis testing \citep{bickel_hypothesis_2015, GOF_SBM, Han_rank}, and thresholding approaches \citep{USVT}.

Although the embedding dimension can be estimated using the methods listed above, there is always the possibility that the embedding dimension is chosen incorrectly.
Thus, there is a need to better understand how node embeddings behave under model misspecification.
This question has received comparably little attention in the literature \citep[though see][for an early result in the setting of the SBM]{community_robustness}.
%investigated the robustness of community detection when the true number of communities is unknown. They showed that estimating the number of communities in the SBM setting is consistent when clustering adjacency spectral embeddings, even when the dimension is chosen larger than the rank of the block connectivity matrix.
In the present work, we consider this question in the context of latent position estimation under the RDPG.
In particular, we examine the behavior of the adjacency spectral embedding \citep[ASE;][]{ASE_spectral_clustering_original} when the embedding dimension is chosen incorrectly.
We show that when the embedding dimension is too large, the embeddings produced by the ASE are still consistent, though possibly at a slower rate than when the dimension is chosen correctly, settling an open question in \cite{RDPG_survey} (see Section 7 of that work).
%This result is concordant with that in \cite{community_robustness}, and applies to a more general class of models.
We also prove an estimation lower-bound for the setting where the embedding dimension is chosen too small, which captures the cost of failing to capture all signal dimensions present in the data.
Our main technical contribution, which may be of independent interest, shows that under the RDPG, all eigenvectors of the adjacency matrix delocalize under mild conditions.
Our results extend earlier work in random matrix theory \citep{Erdos} on the eigenvectors of a deterministic rank-one matrix subject to additive noise.

Before proceeding, we pause to establish notation to be used in the remainder of this paper.
%The remainder of this paper is structured as follows: in Section \ref{sec:setup}, we introduce the signal-plus-noise model and adjacency spectral embedding; Section \ref{sec:theory} showcases our main result, consistency of the adjacency spectral embedding under model misspecification; in Section \ref{sec:theory}, we show numerical simulations; in Section \ref{sec:disco}, we discuss the implications of our main result and offer potential follow-up work. %% commenting this for now, because I'm not sure we need to bother with a roadmap.

\subsection{Notation}
Let $\norm{\mZ}$ and $\norm{\mZ}_F$ denote the spectral norm and Frobenius norm of $\mZ$ respectively.
Matrices will be denoted by capital bold letters, e.g., $\mZ$. 
%The $i$\ts{th} column vector of a matrix $\mZ$ will be denoted by a lowercase version with the subscript $i$: e.g., $z_i$. 
Given the full spectral decomposition of a $n\times n$ symmetric matrix $\mZ = \mU \Lambda \mU^\top$, we assume the eigenvalues and associated eigenvectors are sorted to be nonincreasing.
%We will use the notation $\mZ_{a:b}$ to mean the submatrix of $\mZ$ comprised of columns $a$ through $b$ of $\mZ$.
For a matrix $\mZ \in \R^{n_1 \times n_2}$, we write $\mZ_{a:b}$ to be the $n_1 \times (b-a)$ submatrix whose columns consist of columns $a$ through $b$ of $\mZ$.
We write $\bbO_r$ to denote the set of all $r$-by-$r$ orthogonal matrices.
The symbol $:=$ will be used to indicate definitions. 
To simplify presentation, for a sequence of $a_n,b_n > 0$, we will use the notation $a_n \lesssim b_n$ to mean that for some $C > 0$, $a_n \le Cb_n$ for sufficiently large $n$, and the notation $a_n \ll b_n$ to mean that $a_n/b_n \rightarrow 0$ in the large-$n$ limit.
%Similarly, $a_n \gtrsim b_n$ denotes that for some $C > 0$, $a_n \ge Cb_n$ for sufficiently large $n$ in the main text. 
We use $a_n \asymp b_n$ to mean that there exist constants $c, C > 0$ such that for sequences $a_n, b_n > 0$, $ca_n \le b_n \le Ca_n$ for all $n$.

\section{Setup and Modeling} \label{sec:setup}

We are concerned in the present work with network models in which the adjacency matrix given by a low-rank signal matrix subject to additive noise.
Models of this form are often termed ``signal-plus-noise'' models in the literature 
\citep[see, e.g.,][]{cape_signal-plus-noise_2019,hao_neurips}.
We thus consider
\begin{equation} \label{eqn:model_low_rank_plus_noise}
    \mA = \mP + \mE ,
\end{equation}
where $\mE$ is a mean-zero noise matrix and $\mP$ is an $n \times n$ symmetric matrix of rank $r$ whose spectral decomposition is given by
\begin{equation} \label{eqn:P-spec-decomp}
    \mP  = \mU_{1:r} \mS_{1:r} \mU_{1:r}^\top ,
\end{equation}
where $\mS_{1:r}$ is a diagonal matrix of non-zero eigenvalues sorted in non-increasing order ($\spop_1 \ge \spop_2 \ge \cdots \ge \spop_r > 0$) and $\mU_{1:r} \in \R^{n \times r}$ is the matrix of corresponding orthonormal eigenvectors. 

In latent variable network models such as the SBM, RDPG and variants thereof discussed in Section~\ref{sec:intro}, $\mP$ is expressible as 
\begin{equation*}
    \mP = \rho_n\latent\latent^\top 
    = \left( \rho_n^{1/2} \latent \right) \left( \rho_n^{1/2} \latent \right)^\top,
\end{equation*}
where the rows of $\latent \in \R^{n \times r}$ encode vertex-level $r$-dimensional latent variables and $\rho_n \in (0, 1]$ can be interpreted as a signal strength parameter.
%By construction, under this model, the entries of $\mP$ are given by the inner products between the rows of $\latent$, $P_{ij} = \rho_n\latent_i^\top\latent_j$, and 
We consider in this work the task of recovering the scaled latent structure $\rho_n^{1/2} \latent$ from the observed $\mA$.
From the spectral decomposition of $\mP$ in Equation~\eqref{eqn:P-spec-decomp}, we may take 
\begin{equation} \label{eq:latent2eigen}
    \rho_n^{1/2} \latent = \mU_{1:r} \mS^{1/2}_{1:r} . 
\end{equation}
An obstacle arises from the observation that for any $\mW \in \bbO_r$, 
\begin{equation} \label{eq:orthogNonID}
    \latent\latent^\top = (\latent\mW)(\latent\mW)^\top .
\end{equation}
That is, any orthogonal rotation of $\latent$ gives rise to the same expected adjacency matrix $\mP$.
Thus, in estimating the scaled latent positions $\rho_n^{1/2} \latent$, we must account for this non-identifiability.
Specifically, we will be content with estimators that recover $\rho_n^{1/2} \latent$ up to orthogonal rotation as in the above display equation.

The expression in Equation~\eqref{eq:latent2eigen} suggests a natural approach to estimating $\rho_n^{1/2} \latent$, namely using the leading eigenvalues and eigenvectors of the observed $\mA$ in place of those of $\mP$.
This motivates the adjacency spectral embedding, first considered in \cite{ASE_spectral_clustering_original}.

\begin{definition}[Adjacency Spectral Embedding] \label{def:ASE}
Given a real, symmetric matrix $\mA$, whose spectral decomposition is given by
\begin{equation*}
   \mA = \mUhat\mShat\mUhat^\top,
\end{equation*}
we define the $d$-dimensional adjacency spectral embedding \citep[ASE;][]{ASE_spectral_clustering_original} of $\mA$ to be
\begin{equation*}
   \mXhat_{1:d} = \mUhat_{1:d}|\mShat|^{1/2}_{1:d} \in \R^d.
\end{equation*}
\end{definition}

The vast majority of recent work on models like the one just described evaluates estimation error via the maximum Euclidean norm between the true structure $\rho_n^{1/2} \latent$ and its estimate \citep[see, e.g.,][]{RDPG_survey,GRDPG_original,levin_central_2019,singular_subspace,levin_recovering_nodate}, after accounting for non-identifiability.
That is, a good estimator $\mXtilde \in \bbO_r$ will be one that achieves small values of
\begin{equation*}
\min_{\mW \in \bbO_r} \twoinfnorm{ \mXtilde \mW - \rho_n^{1/2} \latent } ,
\end{equation*}
where $\twoinfnorm{ \mZ }$ is defined the be the maximum Euclidean norm of the rows of matrix $\mZ$.
Convergence in this $(\tti)$-norm implies uniform recovery of vertex-level structures encoded in the rows of $\latent$ \citep{consistent_dot_product_SBM}.
Such convergence can be guaranteed under a number of latent variable network models, among which we consider the Random Dot Product Graph \citep[RDPG;][]{RDPG_survey,original_RDPG_Young}, a model in which network structure is driven by latent vertex-level geometric structure.

\begin{definition}[Random Dot Product Graph] \label{def:RDPG}
    Let $F$ be a probability distribution on $\R^r$, and let $\latent_1, \latent_2, \dots, \latent_n$ be drawn i.i.d.~according to $F$ and collect these into the rows of $\latent \in \R^{n \times r}$
    and let $(\rho_n)_{n=1}^\infty$ be a sequence with $\rho_n \in [0,1]$ for all $n=1,2,\dots$.
    We say that $\mA$ is distributed according to a {\em weighted RDPG} with latent positions $\latent$ and signal strength $\rho_n$ if, conditional on $\latent$, $\mE = \mA - \rho_n \latent \latent^\top$ has independent entries (up to symmetry).
    When $F$ is such that  $\bf x^\top \bf y \in [0, 1]$ for all $\bf x, \bf y$ in the support of $F$, we say that $\mA$ is distributed according to a {\em binary RDPG} with latent positions $\latent$ and sparsity parameter $\rho_n$ if
    \begin{equation}\label{eqn:conditional_prob}
        \Prob{\mA \mid \latent} = \prod_{i < j}(\rho_n\latent_i^\top\latent_j)^{A_{ij}}(1-\rho_n\latent_i^\top\latent_j)^{1-A_{ij}} .
    \end{equation}
    In either case, we write $(\mA, \latent) \sim \RDPG(F, n)$ with dimension $r$ and parameter $\rho_n$ clear from context, and refer to the rows of $\latent$ as the latent positions of the network.
\end{definition} 

The RDPG extends a number of widely-used network models.
For example, it recovers the SBM as a special case by taking the latent positions to be drawn from a suitably-chosen mixture of point masses.
Similarly, the RDPG recovers the degree-corrected SBM \citep{dcSBM_original} as a special case by taking $\latent_i = c_iY_i$, where $c_i \in (0,1)$ and $Y_i$ is randomly chosen to be one of $r$ distinct vectors $y_1,  y_2, \dots, y_r \in \R^r$.
The RDPG can also be extended and generalized in a number of ways.
For example, one may modify the model so that edge probabilities are given by a kernel function on the latent space $\kappa: \calX\times\calX \rightarrow [0,1]$ \citep{TanSusPri2013}. 
In the RDPG as given in Definition~\ref{def:RDPG}, $\kappa$ is simply the inner product and $\calX = \R^r$. 
The positive semidefinite structure of $\mP = \rho_n \latent \latent^\top$ can also be relaxed \citep{GRDPG_original}.
See \cite{RDPG_survey} for a more in-depth discussion of the RDPG.

\begin{remark} \label{remark:signal-strength-parameter}
The reader may have noticed that the parameter $\rho_n$ in Definition~\ref{def:ASE} is playing two distinct but related roles.
Under the binary RDPG, $\rho_n$ describes sparsity, in the sense familiar to most network researchers: as the number of vertices $n$ grows, a vanishing fraction of the possible $n^2$ edges are actually present in the observed network.
That is, the probability of an edge appearing between two particular nodes shrinks to zero as the number of nodes grows.
As discussed in Definition~\ref{def:RDPG}, we incorporate this into our model through the parameter $\rho_n$, which, for binary networks, we will refer to as the sparsity parameter. 
In the case of the weighted RDPG, $\rho_n$ describes a signal strength: if $\rho_n=1$, the eigenvalues of $\mP$ would grow linearly with $n$ (see Lemma~\ref{lemma:pop_eigenvalues_bound} in Appendix~\ref{apx:spectral}).
Having $\rho_n$ shrink to zero as the network size grows allows us to make the estimation problem harder by having the eigenvalues of $\mP$ grow sublinearly.
The key distinction between these two settings is that in the weighted RDPG, we allow for the variance of the entries of $\mE$ to be specified separately from the signal strength $\rho_n$, while in the case of the binary RDPG, the variance of the edges, being Bernoullis, is inherently tied to the behavior of $\rho_n$.
\end{remark}

\subsection{Model Misspecification}

The consistency of the ASE has been studied extensively in the context where the embedding dimension is correctly specified \citep[i.e., $k = 0$ in our notation above; see][]{RDPG_survey, levin_recovering_nodate, ASE_spectral_clustering_original, GRDPG_original, Levin-Laplacian, Tang-Laplacian, lyzinski_community_2016}.
Of course, in practice, the model rank $r$ is typically unknown and needs to be estimated from observed data.
Methods to estimate the embedding dimension have been widely explored in the literature, as discussed in Section \ref{sec:intro}.
Even with the best embedding dimension estimation techniques, however, it is still possible to misspecify the embedding dimension.
Thus, our goal in the present paper is to characterize the behavior of the ASE when the embedding dimension is incorrectly chosen.
To the best of our knowledge, this is the first work to do so.

Showing consistency of the ASE in the $(\tti)$-norm when the true embedding dimension $r$ is known require that we control the quantity
\begin{equation}\label{eqn:main:norm}
    \min_{\mW \in \bbO_{r}} \twoinfnorm{\mXhat_{1:r} \mW - \rho_n^{1/2}\mX_{1:r} } ,
\end{equation}
where we recall that the minimization over $\mW$ accounts for the non-identifiability of the latent positions.
The quantity described in Equation \eqref{eqn:main:norm}, along with related ones, has been controlled in a range of settings \citep{lyzinski_community_2016, levin_recovering_nodate, levin_central_2019, GRDPG_original, singular_subspace} when the true embedding dimension is known.
These results are contingent on being able to apply the Davis-Kahan $\sin\Theta$ theorem or variants thereof \citep[][]{Davis-Kahan, Samworth, Belkin}, which requires a non-zero eigengap.
When we embed into $d\neq r$ dimensions, these approaches are no longer feasible.
In particular, when the embedding dimension is chosen too large, the ASE includes eigenvectors whose associated population eigenvalues are zero, resulting in a zero eigengap.
Thus, if we are to characterize the behavior of the ASE when the embedding dimension is misspecified, we will need to  deploy different tools.

In investigating the effects of dimension misspecification, we will find it easiest to speak of the true embedding dimension $r$ and the difference $k = d-r$ between the selected dimension and true dimension.
We will then consider the $(r+k)$-dimensional adjacency spectral embedding,
    \begin{equation*}
   \mXhat_{1:r+k} = \mUhat_{1:r+k}|\mShat |^{1/2}_{1:r+k} ,
    \end{equation*}
where $r$ is the rank of $\mP$ and $k=d-r\in \{-r, -r +  1, \dots, n - r \}$ describes the difference between the true model rank and the selected embedding dimension.

In adapting the quantity in Equation \eqref{eqn:main:norm} to the setting where the embedding dimension is incorrectly chosen, we will need the following notation to precisely formulate the estimation error for the ASE. 
When we choose our embedding dimension too small (i.e., $k<0$), we will need to pad our estimate $ \mXhat_{1:r+k}$ with additional columns of zeros so that it can sensibly be compared with the $n$-by-$r$ matrix $\latent_{1:r}$.
Similarly, when the embedding dimension is chosen too large (i.e., $k > 0$), we will need to pad $\latent_{1:r}$ with additional columns of zeros to ensure that our matrix of embeddings and true latent positions are conformable.
Hence, in order to study the setting where $k<0$, we define $\mXhat_{1:r}^\circ \in \R^{n \times r}$ to be a zero-padded version of $\mXhat_{1:r+k}$. 
That is, when $k < 0$, we define
\begin{equation}\label{eq:def:Xcirc}
   \mXhat_{1:r}^\circ :=  \begin{bmatrix}
        \mXhat_{1:r+k} & \mathbf{0}_{r+k+1:r}
    \end{bmatrix} .
\end{equation}
Similarly, when $k> 0$, we define $\mX_{1:r+k}$ as the $n\times (r+k)$ matrix
\begin{equation}\label{eq:def:X_r+k}
    \mX_{1:r+k} :=  \begin{bmatrix}
        \mX_{1:r} & \mathbf{0}_{r+1:r+k}
    \end{bmatrix} .
\end{equation}

In understanding the estimation error of the ASE, in the setting where the embedding dimension is chosen too large ($k>0$), we perform the following decomposition, ignoring the non-identifiability for the time being:
\begin{equation*}
        \begin{aligned}
        \twoinfnorm{\mXhat_{1:r+k} - \rho_n^{1/2}\mX_{1:r+k}} 
         & \le  \twoinfnorm{\mXhat_{1:r} - \rho_n^{1/2}\mX_{1:r}} + \twoinfnorm{\mXhat_{r+1:r+k} - \rho_n^{1/2}\mX_{r+1:r+k} } \\
         & = \underbrace{\twoinfnorm{\mXhat_{1:r} - \rho_n^{1/2}\mX_{1:r}}}_{\text{Controlled via standard techniques}} \hspace{-0.1 cm} + \underbrace{\twoinfnorm{\mXhat_{r+1:r+k} } }_{\text{Goal}} .
        \end{aligned}
\end{equation*}

From this decomposition, we see that there are two terms, one corresponding to the rate that would hold if the embedding dimension were correctly chosen, and the other corresponding to the \enquote{trailing} dimensions. 
We formalize this intuition in the following lemma, which gives results both for the case where the embedding dimension is chosen too large and for the case where it is chosen too small.
A proof is given in Appendix~\ref{apx:estimation-error}.

\begin{lemma}[Misspecified Model Bounds]\label{lemma:decomp}
Suppose that $\mA = \rho_n\latent \latent^\top + \mE$, and denote the $(r+k)$-dimensional ASE of $\mA$ as $\mXhat_{1:r+k} = \mUhat_{1:r+k}|\mShat|^{1/2}_{1:r+k}$.
When $k<0$,
\begin{equation*} 
\min_{\mW \in \bbO_{r + k}} \twoinfnorm{\mXhat_{1:r}^\circ \mW - \rho^{1/2}_n \latent_{1:r} }
\ge \sqrt{ \sum_{j=r+k+1}^r \frac{ \spop_j }{ n } } .
\end{equation*}
On the other hand, when $k > 0$, suppose that there exists a sequence of $\mW^* \in \bbO_r$ such that
\begin{equation} \label{eq:def:etan}
    \twoinfnorm{\mXhat_{1:r}\mW^* - \rho_n^{1/2}\mX_{1:r}} \lesssim \etan,
\end{equation}
for some sequence $(\etan)_{n=1}^\infty$.
Then there exists a sequence of $\mW \in \bbO_{r+k}$ such that
\begin{equation*}
    \twoinfnorm{\mXhat_{1:r+k}\mW - \rho^{1/2}_n\latent_{1:r+k}} 
    \lesssim \etan
    + \norm{\mUhat_{r+1:r+k}}_{\tti}\norm{\mE}^{1/2} .
\end{equation*}
\end{lemma}

The term $\etan$ in Equation~\eqref{eq:def:etan} is meant to capture the estimation rate that would hold if the embedding dimensions were correctly chosen.
As mentioned earlier, this setting has been studied extensively in the literature, and the precise behavior of $\etan$ is known under several different variants of the RDPG corresponding to different modeling choices for $\latent$ and the edge noise $\mE$.
We collect a few representative results below, by way of example.

In the case of binary networks, we present the following theorem adapted from \cite{lyzinski_community_2016}.
\begin{theorem}[\cite{lyzinski_community_2016}, Theorem 5] \label{thm:lyzinski}
     Suppose that $(\mA,\latent) \sim \RDPG(F,n)$, then there exists a sequence of $\mW \in \bbO_{r}$ such that
     \begin{equation*}
        \twoinfnorm{\mXhat_{1:r} \mW- \rho_n^{1/2}\mX_{1:r}} \lesssim \frac{r^{1/2} \log^2 n}{\sqrt{\rho_n n}} 
     \end{equation*}
     with high probability. 
\end{theorem}

For the next example, we introduce the generalized random dot product graph (GRDPG) from \cite{GRDPG_original}, which extends the RDPG from Definition~\ref{def:RDPG} to allow $\mP$ to be indefinite.

\begin{definition}[\cite{GRDPG_original}, Definition 2]
    Denote $\mI_{p,q}$ as a $(p+q) \times (p+q)$ diagonal matrix with $1$ in the first $p$ diagonal entries and $-1$ in the last $q$ diagonal entries. Suppose $\calX \subset \R^r$ with the property that $x^\top \mI_{p,q}y \in [0,1]$ for all $x,y$. Let $F$ be a distribution on $\calX^n$, and $\latent_1, \latent_2, \dots, \latent_n$ be drawn i.i.d according to $F$. Collect these latent positions in the rows of $\latent \in \R^{n\times r}$. Then conditional on $\mX$, if the adjacency matrix $\mA$ is obtained by
    \begin{equation*}
        A_{ij} \stackrel{\text{ind}}{\sim} \Bernoulli(\rho_n\latent_i^\top \mI_{p,q}\latent_j) 
    \end{equation*}
    for all $i > j$. Then we say $(\mA, \latent) \sim \GRDPG(F, n)$ with signature $(p,q)$.
\end{definition}

\begin{theorem}[\cite{GRDPG_original}, Theorem 1]
Suppose that $(\mA, \latent) \sim \GRDPG(F, n)$ with signature $(p,q)$. Further suppose $\rank(\mP)$ is fixed. There exists a universal constant $c > 1$ such that, provided the sparsity factor satisfies $\rho_n n = \omega(\log^{4c} n)$, there exists a sequence of $\mW \in \bbO_{r}$ such that
\begin{equation}\label{thm:grdpg-rate}
    \twoinfnorm{\mXhat_{1:r} \mW - \rho_n^{1/2}\latent_{1:r}} \lesssim \frac{\log^c n}{\sqrt{n}} 
\end{equation}
with high probability. 
\end{theorem}

The following theorem, adapted from \cite{levin_recovering_nodate}, provides a bound for $\etan$ in the weighted network setting, when the entries of $\mE$ obey sub-gamma tail decay \citep[see Chapter 2 in][]{BLM2013}.
\begin{theorem}[\cite{levin_recovering_nodate}, Theorem 6]
    Suppose $\mA = \mP + \mE = \rho_n\latent\latent^\top + \mE$, and suppose that conditional on $\mX \in \R^{n\times r}$, the entries of $\mE$ are independent $(\nu, b)$-sub-gamma random variables, with $ n(\nu + b^2)\log^2 n \ll \spop_r^2$.
    Then there exists a sequence of $\mW \in \bbO_r$ such that
    \begin{equation*}
        \twoinfnorm{\mXhat_{1:r} \mW - \rho_n^{1/2}\latent_{1:r}}
        \lesssim \frac{r(\log n)\sqrt{\nu + b^2}}{\spop_r^{1/2}} 
                + \frac{rn(\log n)^2(\nu + b^2) \spop_1 }{\spop_r^{5/2}}
    \end{equation*}
    with high probability.
\end{theorem}

From Lemma~\ref{lemma:decomp}, we see that showing consistency when the embedding dimension is chosen too large (i.e., $k > 0$) amounts to being able to control $\twoinfnorm{\mXhat_{r+1:r+k}}$.
By Definition~\ref{def:ASE}, the columns of $\mXhat_{r+1:r+k}$ can be expressed as
\begin{equation*}
    \mXhat_{r+1:r+k} 
    = \begin{bmatrix}
    \uhat_{r+1}|\shat_{r+1}|^{1/2} & \uhat_{r+2}|\shat_{r+2}|^{1/2} &\dots 
    & \uhat_{r+k}|\shat_{r+k}|^{1/2}
    \end{bmatrix} \in \R^{n \times (r+k)}.
\end{equation*}
Bounding the eigenvalues as $\shat_{r+k} \le \| \mE \|$, we may deploy a range of tools from the matrix concentration inequalities literature \citep[see, e.g.,][]{Lei_community_consistency, Erdos,Vershynin_2018}, depending on the distribution of the entries of $\mE$.
Thus, the challenge lies in controlling the entries of the eigenvectors associated with these \enquote{trailing} eigenvalues.
Intuitively, the entries of the ASE associated with these eigenvectors must shrink to zero, as these are zero in the true latent positions $\latent$.
Since $\norm{\mE} \lesssim \sqrt{n}$ (ignoring logarithmic factors), the entries of $\mUhat_{r+1:r+k}$ must vanish uniformly.
Many random matrix models produce eigenvectors whose entries are bounded as $C(\log n)^c/\sqrt{n}$ for some $c > 0$.
These eigenvectors are called \textit{delocalized}, and they have been widely explored in the random matrix literature \citep{Erdos, Erdos_book, ErdHos2007SemicircleLO, eigenvectors_survey, Benigni2020OptimalDF, veryshnin_rudelson}. 
To show delocalization for our setting of a real, symmetric matrix $\mA$, we assume that $\mE$ is an $n \times n$ real, symmetric matrix whose entries are mean-zero, conditionally independent random variables given $\mP$ (up to symmetry).
Moreover, we will make the following assumptions on the spectral properties of $\mP$ and the moments of $E_{ij}$, both of which are allowed to depend on $n$:
\begin{enumerate}[label=(\textbf{A\arabic*}),ref=\textbf{A\arabic*}]
    \item{Assumption 1 (Delocalization)}: The eigenvectors associated with non-zero eigenvalues of $\mL$ are delocalized.
    That is, for some $\gamma \ge 0$ with no dependence on $n$.,
    \begin{equation*} \label{assumption_main:delocalization}
        |u_{ij}| \lesssim \frac{(\log n)^{\gamma}}{\sqrt{n}}
    \end{equation*}
    for all $i \in [n]$, and $j \in [r]$ where $v_{ij}$ refers to the $i$-th entry of the $j$-th eigenvector of $\mP$. 
    \label{assumption_main:deloc} 
    \item{Assumption 2}: The signal strength parameter $\rho_n = \omega(1/\sqrt{n})$\label{assumption_main:signal_strength}
    
    \item{Assumption 3}: The non-zero eigenvalues of $\mP$ are distinct. \label{assumption_main:eigen_nondegen}

    \item{Assumption 4}: The non-zero eigenvalues of $\mP$ are of order $\rho_n n$: $\spop_j \asymp \rho_n n$ for $j \in [r]$. \label{assumption_main:eigen_bound}

    \item{Assumption 5}: $r = \rank(\mP) \ll (\log n)^{\zeta}$ for some $\zeta \ge 0$. \label{assumption_main:rank}

    \item{Assumption 6 (Eigengap)}: The eigengap $\Delta := \min_{j \in [r]} \min_{i\ne j} |\spop_i - \spop_j|$ obeys $\Delta = \Omega\left(\rho_n n\right)$. \label{assumption_main:eigengap}

    \item{Assumption 7}: The entries of $\mE$ satisfy the moment conditions \label{assumption_main:moments}
    \begin{equation*} 
        \E{E_{ij}} = 0, \quad \E{|E_{ij}|^2} = \sigma^2, \quad \E{|E_{ij}|^{p}} \le \frac{C_1^p}{\tilde{q}_n^{(p/2-1)}} \quad(p >2)
    \end{equation*}
    where $\tilde{q}_n \in \left[ n^{-1/2} (\log n)^6, 1 \right]$ is a parameter the controls the tail behavior of $\mE$, corresponding to the sparsity parameter in the context of binary networks.
\end{enumerate}

\begin{remark}
We pause to briefly discuss some of these assumptions.
Assumption \ref{assumption_main:deloc} is similar to the incoherence condition often found in spectral methods \citep{spectral_methods_textbook, hao_neurips, Zheng_inference}. 
It is crucial for our main delocalization result to hold. In general, without such conditions, we do not necessarily expect eigenvectors of random matrices to delocalize \citep[see][for further discussion]{eigenvectors_survey}. 
In the case of the RDPG, it can be shown that $\twoinfnorm{\mU} \lesssim 1/\sqrt{n}$ (see Corollary \ref{cor:RDPG_delocalization:binary} below).
Assumption \ref{assumption_main:eigen_nondegen} assumes distinct non-zero eigenvalues in $\mP$.
We expect that this assumption can be relaxed at the expense of increased notational complexity in tracking the eigenspaces corresponding to repeated eigenvalues in the proof of Theorem \ref{theorem:main:deloc}, a task we leave to future work.
Assumptions \ref{assumption_main:signal_strength} and \ref{assumption_main:eigengap} are necessary to use results from spectral methods. 
We suspect delocalization of the eigenvectors of $\mA$ still hold even when these assumptions are relaxed, as suggested by results in \cite{hao_neurips}.
Finally, the moment conditions for $\mE$ in Assumption \ref{assumption_main:moments} allow for sub-exponential decay. 
See Remark 2.5 from \cite{Erdos} for details. 
\end{remark}

\section{Theoretical Results} \label{sec:theory}
We now state our main theoretical results for the case of weighted networks.
These results rely basic properties of the $(\tti)$-norm, as well as on fundamental results in spectral methods and random matrix theory.
Our first result establishes that the trailing eigenvectors of the adjacency matrix $\mA$ delocalize under the assumptions given above.
A proof is given in Appendix~\ref{apx:deloc}.

\begin{theorem} \label{theorem:main:deloc}
Let $\mA = \mP + \mE$, and suppose Assumptions \eqref{assumption_main:deloc} -- \eqref{assumption_main:moments} hold. Then, we have
\begin{equation}
    \max_{\alpha > r} \abs{\uhat_{j\alpha}} \lesssim \frac{r^2(\log n)^{4+6\gamma}}{\sqrt{n}}
\end{equation}
with high probability.
\end{theorem}

The proof of this result is primarily based on the approach given in \cite{Erdos}, where they prove the semicircle law for the case where $\mP$ in Equation~\eqref{eqn:model_low_rank_plus_noise} has rank one.
% a rank-1 signal matrix when considering Equation~\eqref{eqn:model_low_rank_plus_noise}. 
Our extension to a more general low-rank signal matrix requires us to deal with a non-trivial interaction between the eigenspaces corresponding to the non-zero signal eigenvalues of $\mP$.
We first obtain a local law for a low-rank expectation matrix whose spectral norm is not bounded by a constant after scaling, which is possible due to our assumption on $\twoinfnorm{\mU_{1:r}}$.
Eigenvector delocalization is then a corollary of the local law.
There are similar works that prove a local law for an expectation matrix whose spectral norm is either bounded by a constant or assumed to be diagonal but with moment conditions different to our setting \citep{erdos_cusp_2025, diagonal_RMT}. 
\cite{largest_eigenvalue_finite_rank} investigates the distribution of the largest eigenvalue for a low-rank-plus-noise model, but in the setting where the spectral norm of $\mP$ is bounded by a constant.
Our proof requires the matrix of entry-wise variances of $\mE$ to be doubly stochastic, which binary networks do not satisfy, in general.
However, we suspect that delocalization still occurs even when this assumption is relaxed.
We offer more discussion of this point in Section \ref{subsec:consistencyForBinary}.
In another direction, we expect that Assumption \eqref{assumption_main:rank} can be relaxed to allow $r$ to grow as fast as $r \ll n^{1/4}$, possibly at the expense of a small polynomial factor in the delocalization bound.

\begin{remark}
Results from random matrix theory concerning the distribution of the eigenvalues have received attention in the networks literature for use in model selection \citep[i.e., embedding dimension selection; see][]{GOF_SBM, bickel_hypothesis_2015}.
We emphasize that we are concerned with the behavior of the ASE when it is (potentially) misspecified, which occurs downstream of selecting an embedding dimension and requires understanding the behavior of the eigenvectors associated with the perturbed null space, rather than the behaviors of the eigenvalues alone.
\end{remark}

Most important for the purposes of the present work, Theorem \ref{theorem:main:deloc} allows us to characterize the eigenvectors associated with the perturbed null eigenvalues.
In particular, we are able to show that for some $c \ge 4$
\begin{equation*}
    \twoinfnorm{\mUhat_{r+1:r+k}} \lesssim \frac{\sqrt{k}\log^c n}{\sqrt{n}} 
\end{equation*}
with high probability.
We highlight the use of this bound in our main contribution---quantifying the behavior of the ASE when its embedding dimension is misspecified. 
A proof is given in Appendix~\ref{apx:estimation-error}.

\begin{theorem}[Estimation Bounds for Weighted Networks] \label{theorem:main_result:weighted} 
Suppose that $(\mA,\latent) \sim \RDPG(F,n)$ is an $r$-dimensional weighted RDPG with signal strength $\rho_n$, and suppose Assumptions \eqref{assumption_main:deloc} -- \eqref{assumption_main:moments} hold.
Denote the $(r+k)$-dimensional adjacency spectral embedding of $\mA$ as $\mXhat_{1:r+k} = \mUhat_{1:r+k}|\mShat|^{1/2}_{1:r+k}$ and assume the existence of a sequence of $\mW^* \in \bbO_r$ and $(\etan)_{n=1}^\infty$ as in Equation~\eqref{eq:def:etan}.
Then there exists a sequence of $\mW \in \bbO_{r+k}$ such that for $k > 0$,
\begin{equation}
    \twoinfnorm{\mXhat_{1:r+k} - \rho_n^{1/2}\latent_{1:r+k}\mW} \lesssim \etan + \sqrt{\sigma^2 k}\frac{r^2(\log n)^{5 + 6\gamma}}{n^{1/4}}
\end{equation} 
with high probability.
On the other hand, for $k < 0$,
\begin{equation*}
 \min_{\mW \in \bbO_{r}} \twoinfnorm{\mXhat_{1:r}^\circ \mW - \rho_n^{1/2}\latent_{1:r}} 
 \gtrsim \sqrt{ |k| \rho_n } .
\end{equation*}
\end{theorem}

\subsection{Consistency for Binary Networks} \label{subsec:consistencyForBinary}

Theorem~\ref{theorem:main_result:weighted} shows consistency in $(\tti)$-norm for weighted networks when the embedding dimension is chosen at least as large as $r$.
The proof relies on being able to show that the eigenvectors of the observed matrix $\mA$ are delocalized, as we have done in Theorem~\ref{theorem:main:deloc}. Binary networks do not always satisfy the moment conditions as outlined in Assumption \ref{assumption_main:moments} required by our current work because of the heterogeneous entries (probabilities) in $\mP$. These probabilities determine determine the variance of the entries of $\mE$, since the entries of $\mA$ are Bernoulli distributed. However, we strongly suspect the core ideas developed in the previous section can be extended to binary networks. Binary networks, in general, violate the doubly stochastic assumption of the matrix of variances, which is needed for us to show delocalization of the eigenvectors of $\mA$. Recent developments in random matrix theory \citep{wigner-type, QVE} have allowed for the doubly stochastic assumption to be relaxed to allow for more general variances. Based on these recent theoretical advancements, along with the experimental results for binary networks shown in Section \ref{subsec:expt:binary}, we suspect we can also relax the doubly stochastic assumption, extending the delocalization result to more general binary networks.
In particular, we conjecture that the eigenvectors of $\mA$ are delocalized, analogous to Theorem~\ref{theorem:main:deloc} 
%KDL: you have all these places where you use a ref without tagging it as lemma, theorem, equation, whatever-- don't make your reader guess what they are about to jump to when they click the link (if their reader supports doc hypoerlinks in the first place!)
%RT: I'm confused. Did I not say "Theorem \ref{}"? Was it just "\ref{}"?
for general binary networks,
\begin{conjecture}[Delocalization] \label{conj:maintext:deloc}
    Suppose that $\mA = \mP + \mE$, and Assumptions \eqref{assumption_main:deloc} -- \eqref{assumption_main:moments} hold, with Assumption~\eqref{assumption_main:moments} relaxed to the condition that $c \le \E{|E_{ij}|^2} \le C$. 
    Then we have for all $j \in [n]$, 
    \begin{equation*}
    \max_{\alpha > r} \abs{\vhat_{j\alpha}} \lesssim \frac{r^2(\log n)^{4+6\gamma}}{\sqrt{n}}
    \end{equation*} 
\end{conjecture}

An immediate corollary of this conjecture, if it can be proven true, would be consistent estimation of the latent positions under the binary RDPG when the embedding dimension is chosen too large $(k >0)$, provided that the sparsity parameter $\rho_n$ does not converge to zero too fast.
For example, delocalization would yield the following,  using the same arguments outlined for the case of weighted networks and Theorem 5 from \cite{lyzinski_community_2016} (quoted above as Theorem~\ref{thm:lyzinski}).
 \begin{equation} \label{eq:XhatX:conjectured}
        \min_{\mW \in \bbO_{r + k}} \twoinfnorm{\mXhat_{1:r+k} \mW - \rho_n^{1/2}\latent_{1:r+k}} 
        \lesssim \sqrt{r}\frac{(\log n)^2}{\sqrt{\rho_n n}}
        + \sqrt{k}r^2(\log n)^4  \left(\frac{\rho_n}{n} \right)^{1/4} .
    \end{equation} 
    with high probability. Similarly, for $k < 0$,
    \begin{equation*} \label{eq:XhatX:conjectured_under}
    \min_{\mW \in \bbO_{r}} \twoinfnorm{\mXhat_{1:r}^\circ \mW - \rho_n^{1/2}\latent_{1:r}} 
    \gtrsim \sqrt{ |k|\rho_n } . 
\end{equation*}
Similarly, another corollary of Conjecture \eqref{conj:maintext:deloc} would yield the following for the generalized random dot product graph (GRDPG), based on Theorem \ref{thm:grdpg-rate} from above,
\begin{equation*}
    \min_{\mW \in \bbO_{r + k}} \twoinfnorm{\mXhat_{1:r+k} \mW - \rho_n^{1/2}\latent_{1:r+k}} 
        \lesssim \frac{ \log^c n}{\sqrt{n}} + \sqrt{k}r^2(\log n)^4  \left(\frac{\rho_n}{n} \right)^{1/4}
\end{equation*}
with high probability, where $c > 1$. Similarly, for $k < 0$,
   \begin{equation*}
    \min_{\mW \in \bbO_{r}} \twoinfnorm{\mXhat_{1:r}^\circ \mW - \rho_n^{1/2}\latent_{1:r}} 
    \gtrsim \sqrt{ |k|\rho_n } . 
\end{equation*}

\section{Experiments} \label{sec:expts}

We now turn to an experimental investigation of our theoretical results presented in Section~\ref{sec:theory}.
We begin by studying the behavior of the ASE for weighted networks, as predicted by our main results Theorems~\ref{theorem:main_result:weighted} and~\ref{theorem:main:deloc}.
In Section~\ref{subsec:expt:binary}, we turn our attention to the effects of sparsity in binary networks for different choices of embedding dimension to support our Conjecture~\ref{conj:maintext:deloc}.

\subsection{Weighted Networks}\label{subsec:weighted_expts}

We first consider the behavior of the ASE in the weighted RDPG setting of Theorem \ref{theorem:main_result:weighted}, in which the signal strength parameter $\rho_n$ is held constant.
To obtain $\mP = \rho_n\latent\latent^\top$, we consider latent positions drawn from a Dirichlet distribution, whereby the rows of $\latent \in \R^{n \times r}$ are drawn i.i.d.~ according to a Dirichlet distribution with parameter $\alpha = (1, 1, 1, 1, 1)$, so that the true latent dimension is $r=5$.
Given these latent positions, a network is specified by generating an adjacency matrix $\mA$ according to 
\begin{equation} \label{eq:def:weightedDense}
    \mA = \latent\latent^\top + \mE,
\end{equation}
where $\mE$ has entries $E_{ij} = E_{ji}$ for $i > j$ drawn i.i.d.~from a mean-zero distribution.
We consider four different noise models for the entries of $\mE$:
\begin{enumerate}
    \item[(a)] $E_{ij} \overset{\iid}{\sim} N(0,1)$; 
    \item[(b)] $E_{ij} + 1\overset{\iid}{\sim} \Laplace(0,1)$;
    \item[(c)] $E_{ij} + 1\overset{\iid}{\sim} \Exp(1)$;
    \item[(d)] $E_{ij} + 1\overset{\iid}{\sim} \Pois(1)$.
\end{enumerate} 
Note that we center the distribution in settings (c) and (d) about $1$ to ensure that the entries of $\mE$ are mean zero as required by our theory.
Given the adjacency matrix $\mA$, we construct its $(r+k)$-dimensional ASE.
Our goal is to examine how well the ASE recovers the latent positions $\latent$, but to do this, we must account for the orthogonal non-identifiability of the model, as discussed in Equation~\ref{eq:orthogNonID}.
To do this, we compute the singular value decomposition of 
\begin{equation*}
    \mXhat_{1:r + k}^\top \mX_{1:r + k} = \mU\Sigma\mV^\top,
\end{equation*} 
and set $\mQ = \mU\mV^\top$, which is the solution to the Procrustes alignment problem
\begin{equation}\label{eqn:expts:procrustes}
    \mQ = \mU\mV^\top = \argmin_{\substack{\mW \in \bbO_{r+k}}} \norm{\mXhat_{1:r+k}\mW - \mX_{1:r+k}}_F .
\end{equation}
We use this choice of $\mW$ to compute
\begin{equation*}
\twoinfnorm{\mXhat_{1:r+k} \mW - \mX_{1:r+k}}
\end{equation*}
as an estimate of the $(\tti)$-norm error between $\mXhat_{1:r+k}$ and the true latent positions.

We conducted the above experiment for varying networks sizes of $n = 300, 600, \dots, 7800$ vertices and varying embedding dimensions $r + k = 4,5,6,10,20,40$, with $80$ Monte Carlo trials per condition.
Figure~\ref{fig:nonbinary_dirichlet} shows the results of this experiment, plotting estimation error of the ASE in $(\tti)$-norm as a function of network size $n$ for varying choices of embedding dimension.
The axes are on a log-log scale to highlight the convergence rate predicted by Theorem~\ref{theorem:main_result:weighted} and demonstrate consistency of the ASE for when the embedding dimension is chosen larger than the true embedding dimension. 
The dashed black line indicates the $n^{-1/2}$ convergence rate (ignoring logarithmic factors) predicted by Theorem~\ref{theorem:main_result:weighted} when the dimension is correctly specified, while the dashed grey line indicates the $n^{-1/4}$ convergence rate (ignoring logarithmic factors) predicted to hold when the embedding dimension is chosen too large.

\begin{figure}[ht]\centering
\subfloat[Normal]{\includegraphics[width=.45\linewidth]{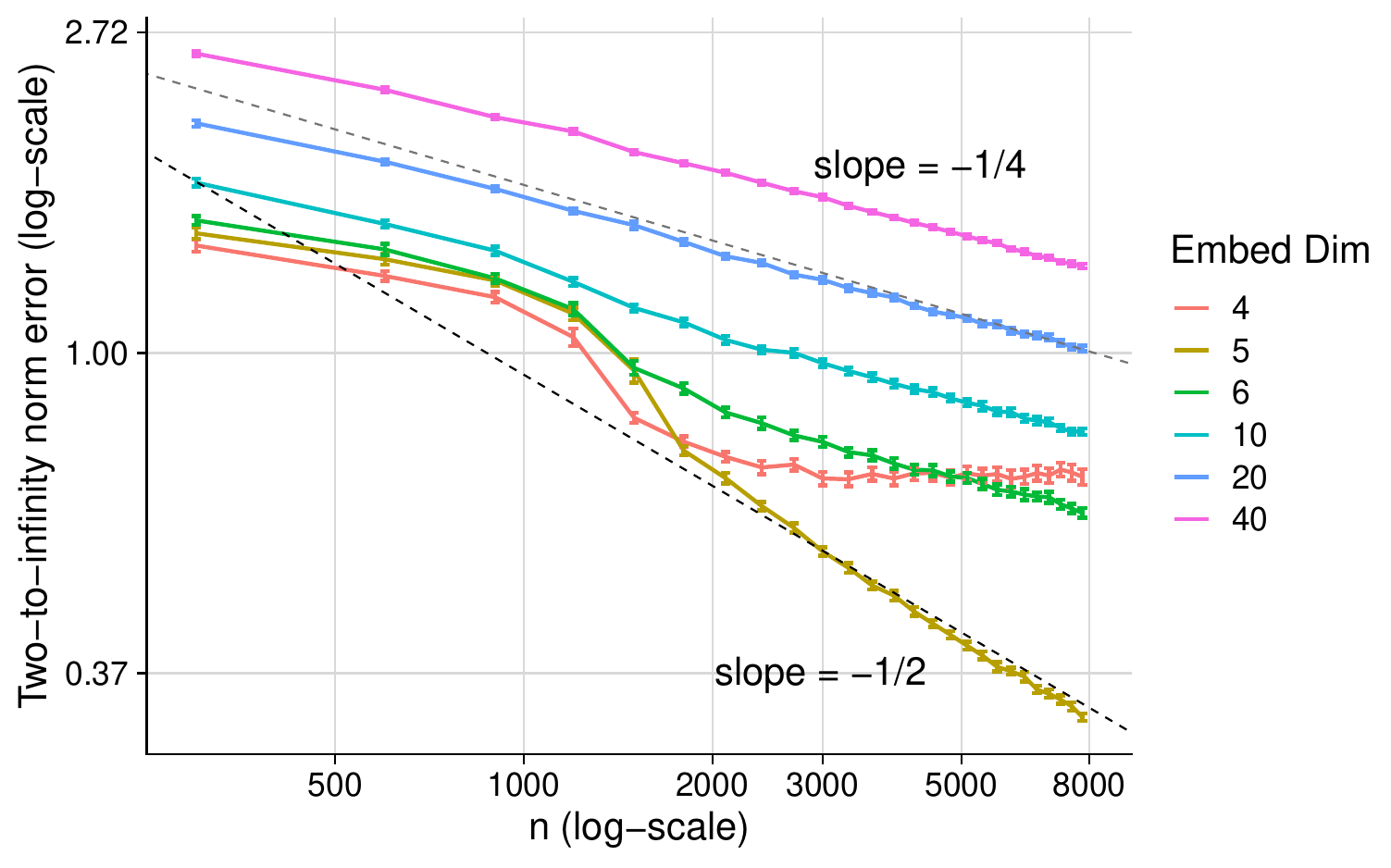}}\hfill
\subfloat[Laplace]{\includegraphics[width=.45\linewidth]{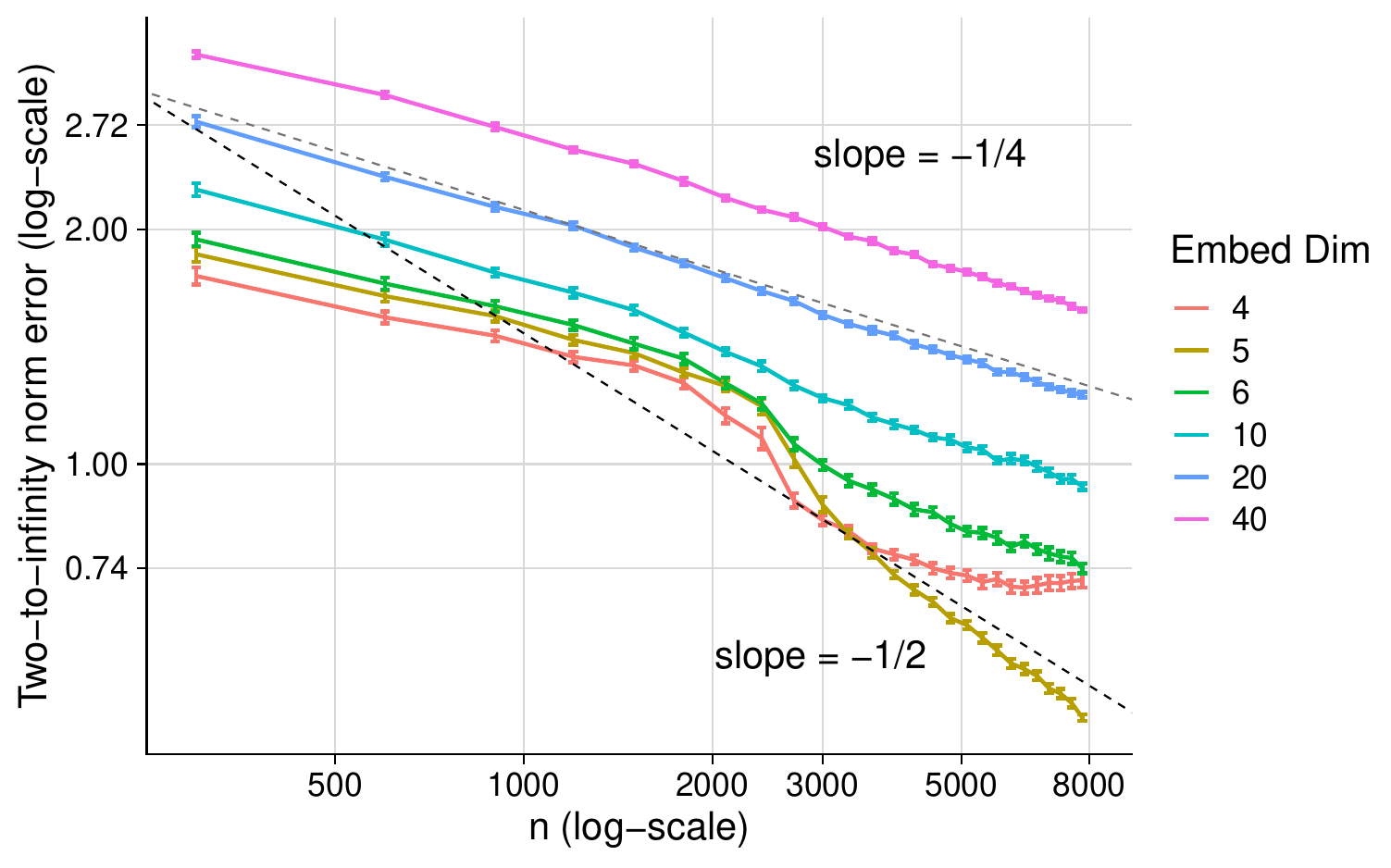}}\par 
\subfloat[Mean-zero Exponential]{\includegraphics[width=.45\linewidth]{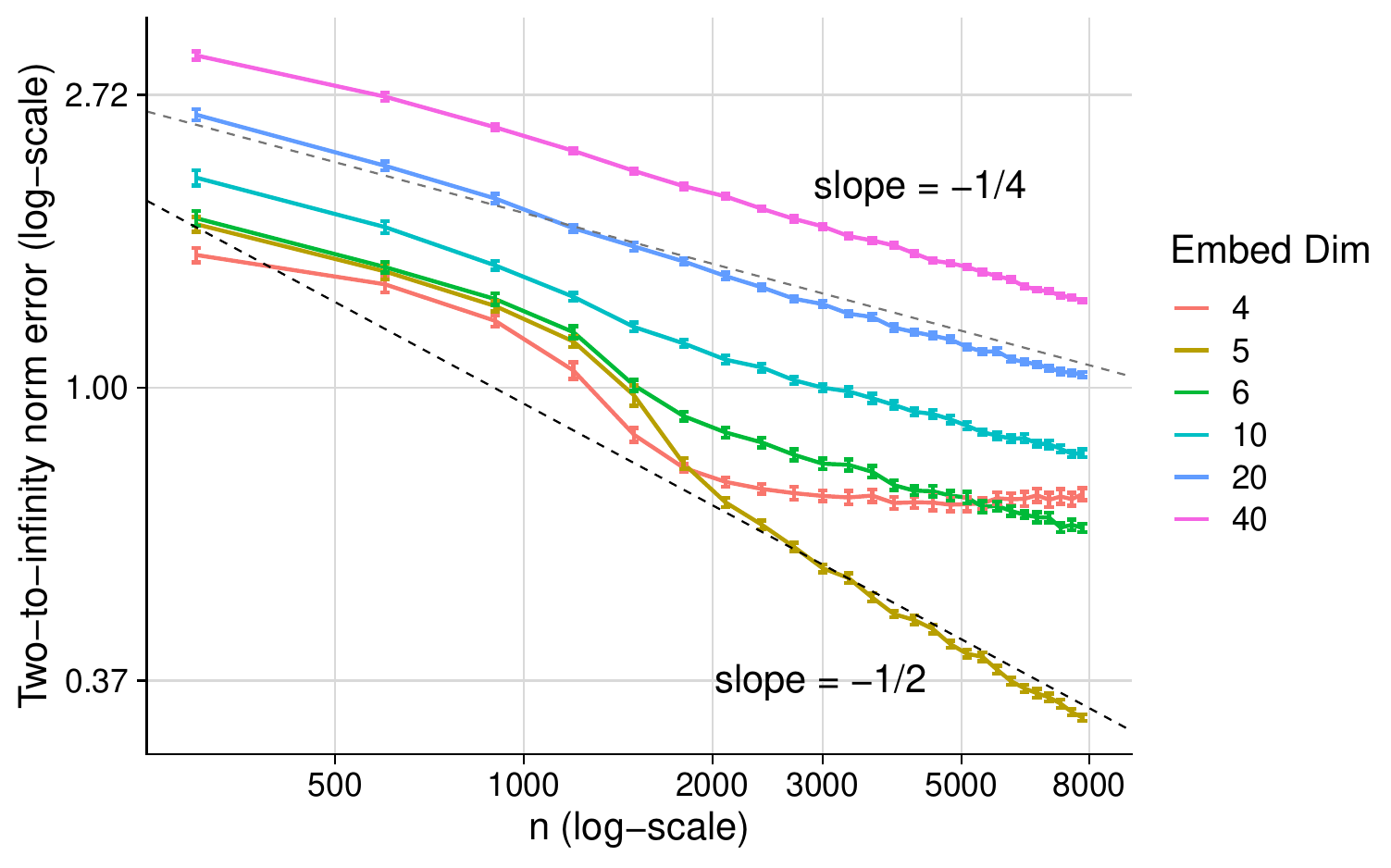}}\hfill
\subfloat[Mean-zero Poisson]{\includegraphics[width=.45\linewidth]{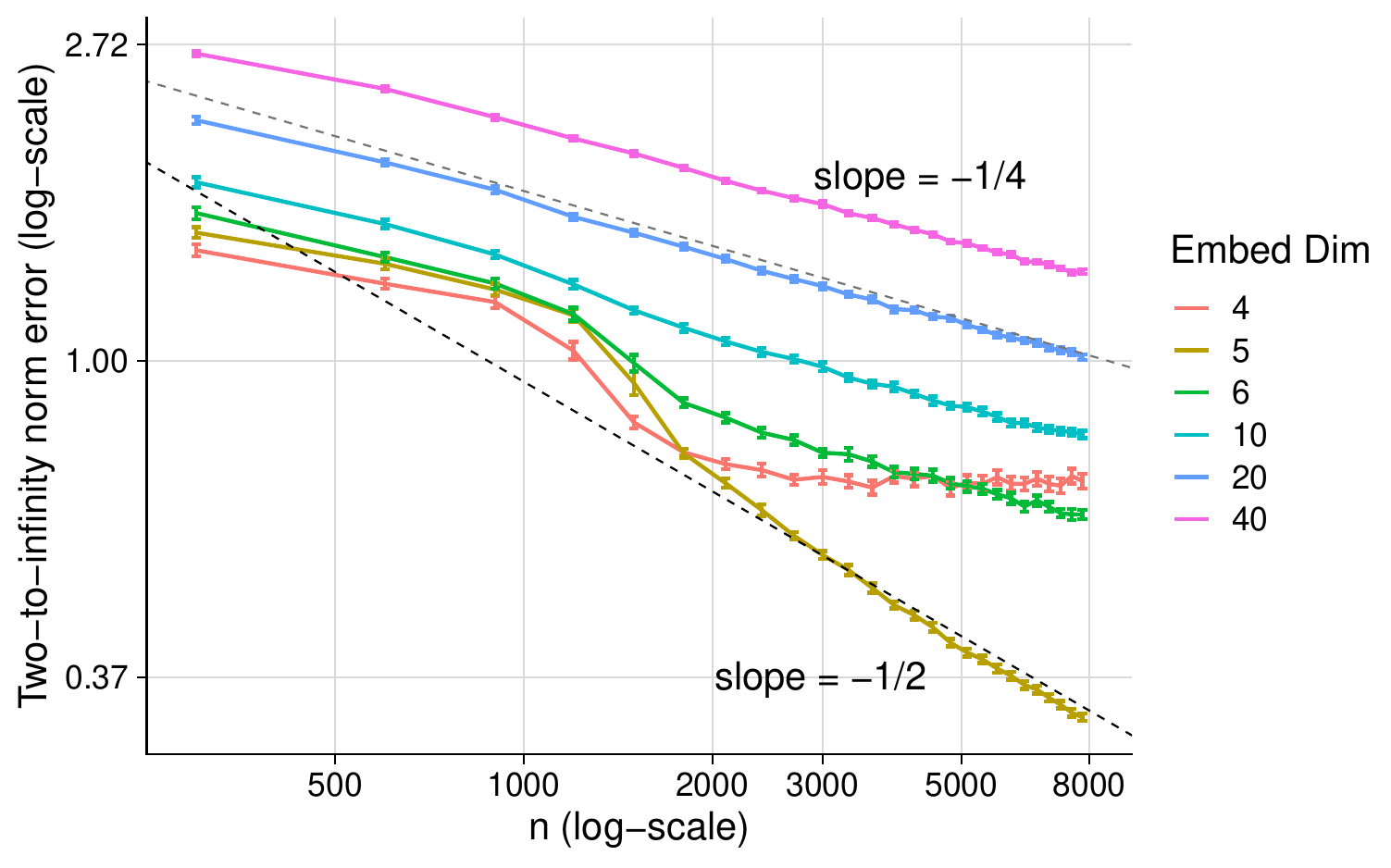}}
\caption{Estimation error of ASE in $(\tti)$-norm as a function of number of vertices $n$ for weighted networks as in Equation~\eqref{eq:def:weightedDense} under four different choices of noise models (a) normal (b) Laplace (c) exponential and (d) Poisson,
under six choices of embedding dimension (colored lines).
The correctly-specified embedding dimension ($r = 5, k= 0$; in gold) is display alongside five misspecified embedding dimensions ($r + k = 4, 6, 10, 20, 40$; orange, green, teal, purple and magenta, respectively).
 The black and gray dashed lines indicate, respectively, the convergence rate predicted for the correctly-specified setting and the setting where the embedding dimension is chosen too large. 
Error bars represent two standard errors of the mean.} 
\label{fig:nonbinary_dirichlet} 
\end{figure} 

Examining Figure \ref{fig:nonbinary_dirichlet}, consistency appears to hold for embedding dimensions that are at least as large as the true embedding dimension ($r=5$).
When the embedding dimension is chosen too small (in this case, $4$, indicated by the orange line), consistency no longer holds: the error between the ASE and the true latent positions appears to approach a constant for suitably large $n$ for all four error models.
With the value of $\rho_n = 1$ in these experiments, we see that the lower bound provided in Theorem \ref{theorem:main_result:weighted}, is concordant with the inconsistency when the embedding dimension is chosen too small here.
Turning our attention to the case of well-specified dimension, indicated by the gold line in Figure~\ref{fig:nonbinary_dirichlet}, we see that as $n$ grows, the estimation error appears to match the predicted $n^{-1/2}$ rate predicted by Theorem~\ref{theorem:main_result:weighted}.
When the dimension is chosen too large, indicated by the green, blue, purple and magenta lines (corresponding embeddings dimensions 6, 10, 20 and 40, respectively) in the figure, we see that consistency still appears to hold, albeit at the slower $n^{-1/4}$ rate, again in line with Theorem~\ref{theorem:main_result:weighted}. 

For fixed $n$, Theorem \ref{theorem:main_result:weighted} predicts that estimation error under the $(\tti)$-norm should exhibit square root behavior as the number of embedding dimensions increases past the true embedding dimension (i.e., $k > 0$).
We explore this experimentally under the setting of previous experiments, restricting our attention to the case where $E_{ij} + 1 \overset{\iid}{\sim} \Exp(1)$ for $i < j$.
We computed the ASE for embedding dimensions of $r+k = 1,  2, \dots, 10, 11$ and then for larger dimensions $r+k = 20, 40, 45$. 

Figure \ref{fig:embed_dim_vs_error} shows that our error in estimating the latent positions grows at most at a rate of $\sqrt{k}$ for when $k > 0$ and $n$ is held fixed. The rates observed in our experiment grows slightly slower that the $\sqrt{k}$ we predict from Theorem \eqref{theorem:main_result:weighted}. 
We suspect that this could be due to uncertainty induced by computing $\mW$ as in Equation~\eqref{eqn:expts:procrustes}.
We also observe that the $(\tti)$-norm is smallest (for $n > 1500$) when the embedding dimension of the ASE is the same as the true embedding dimension ($r = 5$). We remark that the $(\tti)$-norm is not minimized at $r + k = 5$ for $n = 300, 600, 900,$ and $1500$ because our results are asymptotic in nature. 

\begin{figure}[ht]
\centering
\includegraphics[width=0.8\linewidth]{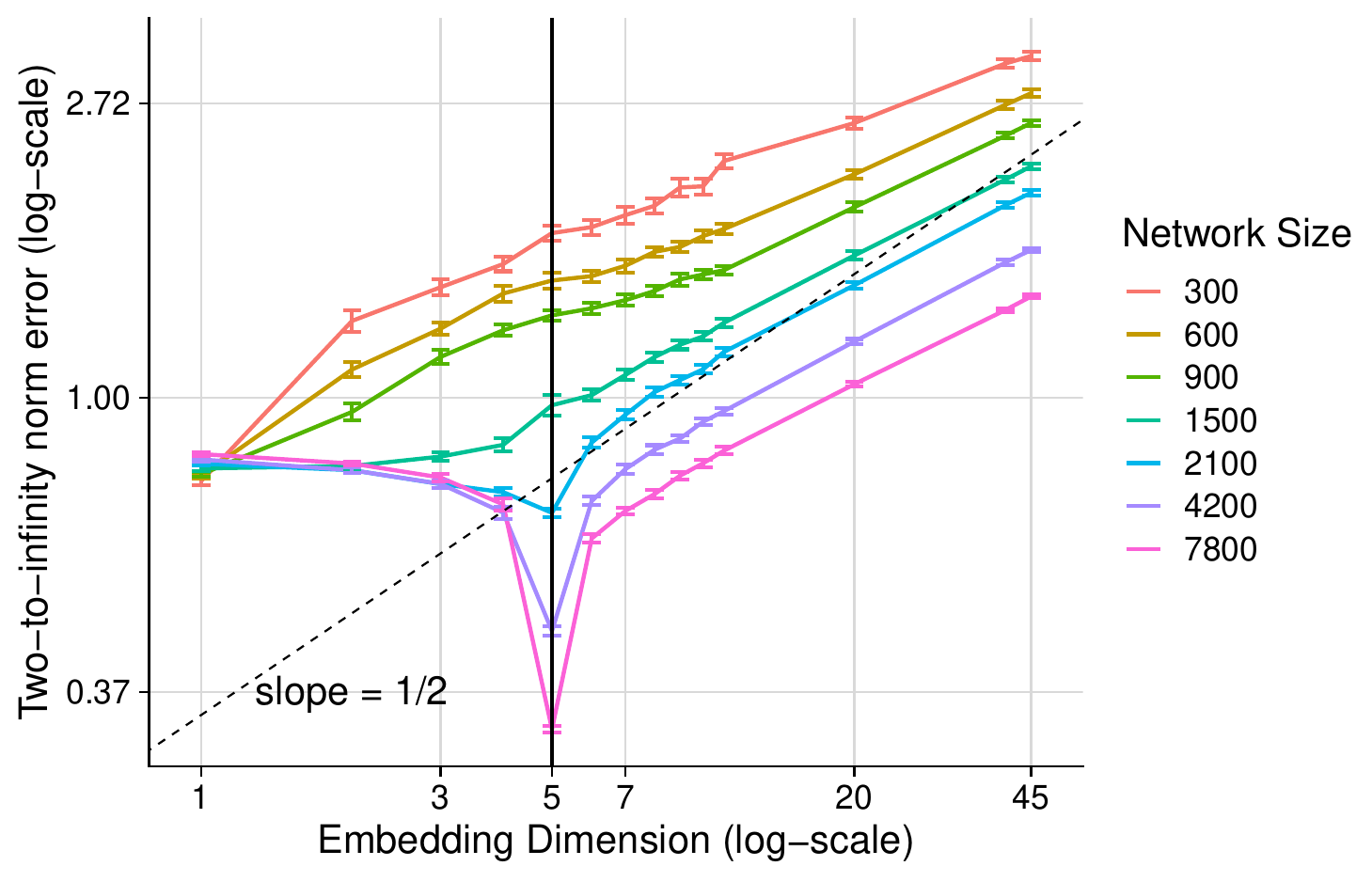} 
  \caption{
  Error in recovering the true latent positions in $(\tti)$-norm, as a function of embedding dimension for varying choices of network size $n$, indicated by line color.
  Both axes are on a logarithmic scale, and the dashed black line indicates the $\sqrt{k}$-like behavior predicted by our results.
  The black vertical line indicates the true embedding dimension, $\rank \mP = 5$.
  Error bars indicate two standard errors of the mean.} \label{fig:embed_dim_vs_error} 
\end{figure} 

In our next set of experiments for weighted networks, we examine the effects of the signal strength parameter, $\rho_n$, on the estimation error of the ASE and its interaction with embedding dimension.
We again sampled latent positions from a Dirichlet distribution as in our previous experiments.
That is, the rows of $\latent$ are drawn i.i.d according to $\Dir(\alpha)$ where $\alpha = (1, 1, 1, 1, 1)$, so that the true latent dimension is $r=5$.
The observed network is then generated according to
\begin{equation}
    \mA = \rho_n\latent\latent^\top + \mE 
\end{equation}
where $\rho_n \in [0, 1]$ is a signal strength parameter, which will shrink to zero with $n$ (see Remark \ref{remark:signal-strength-parameter}).
In particular, we take $\rho_n = n^{-\gamma}$ for $\gamma \ge 0$ and generate edge noise according to $E_{ij} \overset{\iid}{\sim} N(0, 0.1^2)$ for $i > j$.
Given the adjacency matrix $\mA$, we again constructed its $(r+k)$-dimensional ASE for varying choices of $k$ and computed the error in estimating the true latent positions in the $(\tti)$-norm, after solving the orthogonal Procrustes problem to account for the nonidentifiability present in the RDPG as done in Equation \eqref{eqn:expts:procrustes}.

We conducted this experiment for networks sizes $n = 1000, 2000, \dots, 8000$ and for various choices of embedding dimension $r + k = 4,5,6, 7,10,20$, with $40$ Monte Carlo replicates per condition.
The results of this experiment are given in Figure~\ref{fig:weighted-weak-signal}, with both axes given on a logarithmic scale to highlight the convergence rate.
We plot the predicted $n^{-1/4}$ convergence rate (for when the embedding dimension is chosen too large relative to the true dimension) as a gray, dashed line as well to highlight the accuracy of our theoretical results in Theorem \ref{theorem:main_result:weighted}, along with the $n^{-1/2}$ for the optimal convergence rate in a black, dashed line. 

\begin{figure}[ht]
\centering
\includegraphics[width=0.9\linewidth]{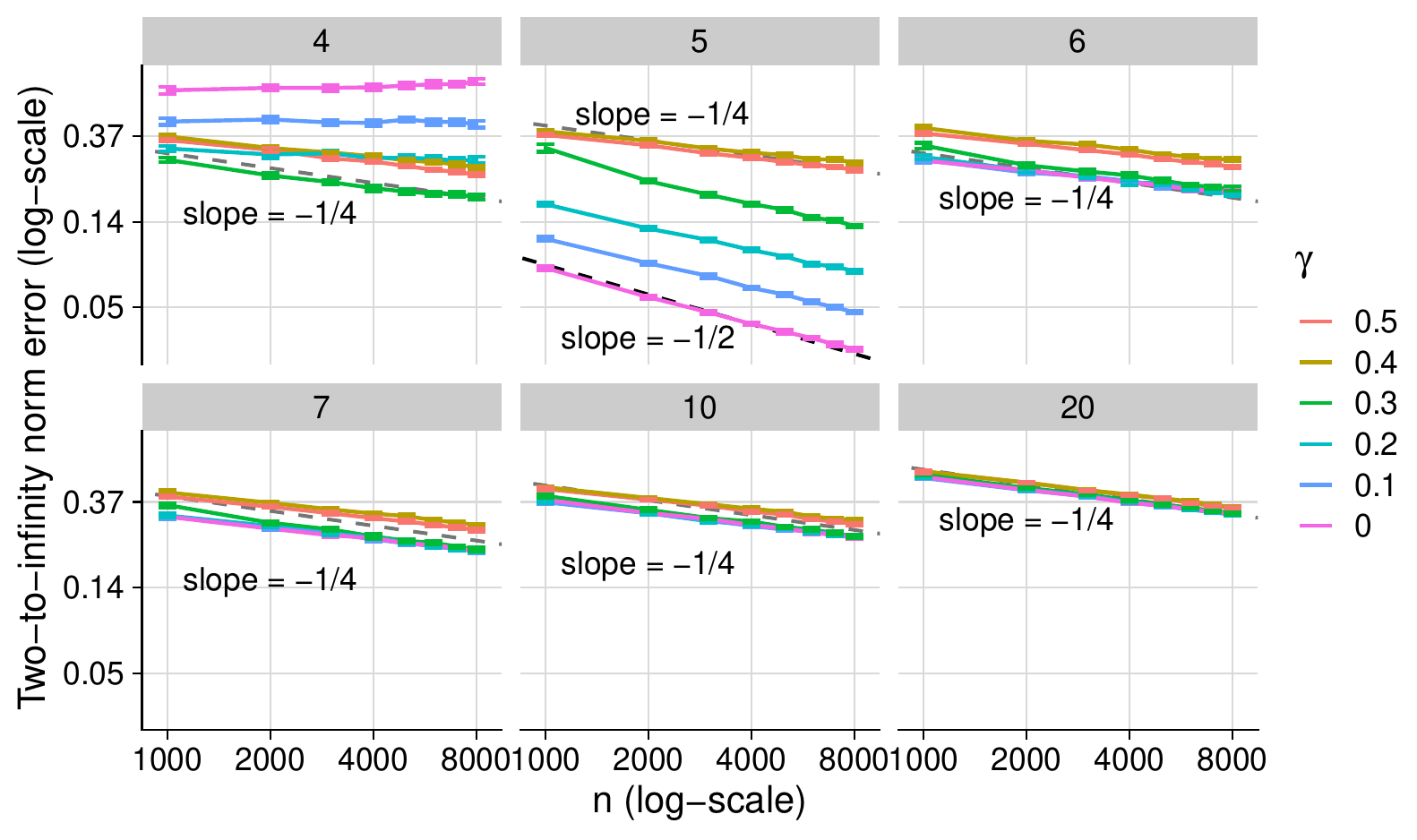}
\caption{Estimation error of ASE in $(\tti)$-norm as a function of network size for varying signal strength levels (indicated by line color) for six choices of embedding dimension.
Subplots show behavior under correctly-specified embedding dimension ($r+k = r = 5$; top row, middle panel), as well as when the dimension is chosen too low ($r+k=4$; top row, left) and too high ($r + k = 6, 7, 10, 20$; top-row right panel and all three bottom row panels).
The black (resp. gray) dashed line with a slope of $-1/2$ (resp. $-1/4$) indicates the predicted convergence rate of $n^{-1/2}$ (resp. $n^{-1/4}$) from Theorem \ref{theorem:main_result:weighted} for the correctly chosen (resp. incorrectly chosen) embedding dimension.
Error bars represent two standard errors of the mean. 
} \label{fig:weighted-weak-signal}
\end{figure}

\subsection{Binary Networks} \label{subsec:expt:binary}

We now investigate the behavior of the ASE under the setting of binary networks.
As discussed in Section~\ref{subsec:consistencyForBinary}, our main results presented in Section~\ref{sec:theory} are not applicable to binary networks because the matrix of entry-wise variances of the adjacency matrix is not doubly-stochastic.
In general, under models like the RDPG, binary matrices have heterogeneous entry-wise variances.
Nonetheless, we expect the eigenvectors of the observed adjacency matrix to still be delocalized, as described in Conjecture~\ref{conj:maintext:deloc}, along with similar behaviors to weighted networks for when the embedding dimension is incorrectly chosen. For our first set of binary network experiments, we consider the stochastic blockmodel.
The SBM models network formation by assigning each vertex to one of $r$ communities, and generating edges conditionally independently given these community assignments, in such a way that the probability of two vertices forming an edge is determined by their community memberships.
We store these probabilities in the entries of a symmetric matrix $\mB \in [0,1]^{r \times r}$, so that $B_{k,\ell}$ encodes the probability that $A_{ij} = 1$, given that vertex $i$ is in community $k$ and vertex $j$ is in community $\ell$. 
All told,
\begin{enumerate}
    \item Draw a random probability vector $\pi \sim \Dir(\alpha)$, where $\alpha \in \R^r$ is a vector of non-negative reals.
    \item Assign each node $i \in [n]$ to one of $r$ communities independently according to $\pi$ and record these memberships as one-hot encodings in the rows of $\mZ \in \R^{n\times r}$.
    \item Construct $\mB \in [0,1]^{r \times r}$ by setting the within-community probabilities (i.e., on-diagonal entries) to $0.9$ and the inter-community edge probabilities (i.e., off-diagonal entries) to $0.1$.
    \item Set $\mP = \mZ\mB\mZ^\top$ and obtain $\latent$ from the eigendecomposition of $\mP = \mU\mS\mU^\top$ as $\latent = \mU_{1:r} \mS^{1/2}_{1:r}$ .
    \item Generate $\mA \in \{0,1\}^{n \times n}$ according to
    \begin{equation} \label{eq:def:SBM}
        \Prob{\mA \mid \latent} = \prod_{i < j} P_{ij}^{A_{ij}}(1 - P_{ij})^{1-A_{ij}} .
    \end{equation}
\end{enumerate} 

We generated $\mA$ according to the procedure described above, with true number of communities (and hence true latent dimension) again set to $r=5$.
From $\mA$, we then calculated the $(r+k)$-dimensional ASE with the embedding dimension chosen correctly (i.e., $k=0$) and incorrectly ($r+k = 4, 6, 7, 10, 20$), and computed the error in $(\tti)$-norm after solving the Procrustes alignment described in Equation \eqref{eqn:expts:procrustes}. 
We repeated this process for network size $n = 300, 600, \dots, 7800$, with $80$ Monte Carlo replicates per condition. 
Figure \ref{fig:uSBM_binary} shows the results of this experiment. 
As in the previous figures, the axes are on a log-log scale to highlight the convergence rate and show consistency of the ASE when the embedding dimension is chosen suitably large.
The black dashed line indicates our conjecture $n^{-1/2}$ rate for the correctly-specified case.
The grey dashed line indicates our conjectured $n^{-1/4}$ convergence rate for when the embedding dimension is chosen too large.
Different choices of embedding dimension are indicated by different line colors, with the gold line indicating the case of correctly-specified dimension $r+k=5$.
Inspecting the figure, we see that the convergence behaviors are broadly similar to the weighted networks considered in Section~\ref{subsec:weighted_expts} (see Figure \ref{fig:nonbinary_dirichlet}).
The predicted $n^{-1/2}$ convergence rate is again obtained when the embedding dimension is chosen correctly ($r =5, k = 0$, in gold), and consistency still appears to hold when the embedding dimension is chosen too large (i.e., $k > 0$, green, teal, purple and magenta), albeit at the slower $n^{-1/4}$ rate, as conjectured.
When the dimension is chosen too small ($k < 0$, in orange), we no longer see consistency, again in line with our theoretical results and conjectured behavior for binary networks.

\begin{figure}[ht]
\centering
\includegraphics[width=0.8\textwidth]{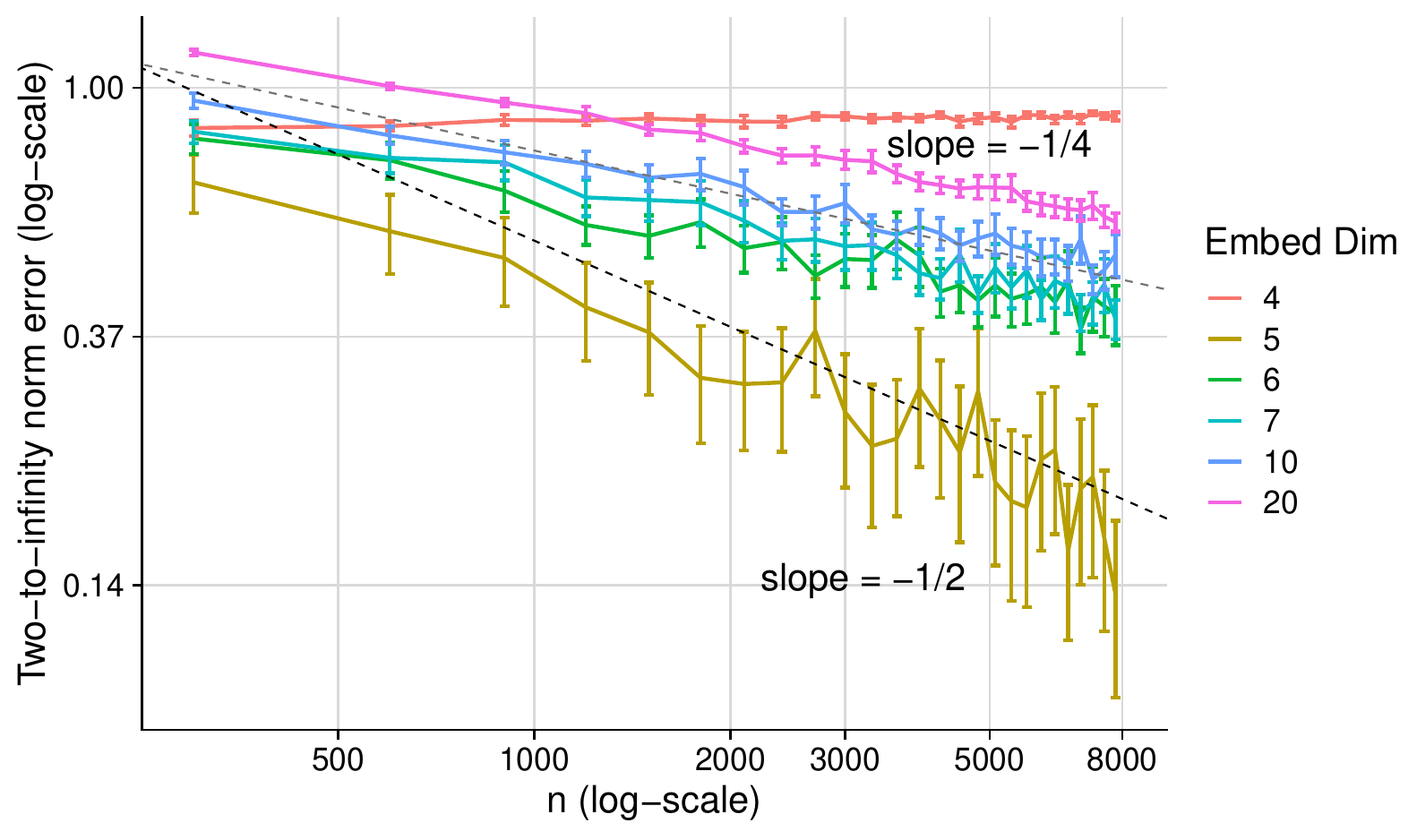}
  \caption{Average estimation error of ASE in $(\tti)$-norm as a function of number of vertices $n$ for different choices of embedding dimension $r+k$ (line colors) under the SBM described in Equation~\eqref{eq:def:SBM}.
  Both axes are on logarithmic scales, with dashed lines in black and grey indicating, respectively, the $n^{-1/2}$ and $n^{-1/4}$ rates predicted by our Conjecture~\ref{conj:maintext:deloc} under the settings where the dimension is chosen, respectively, either correctly or too large.
  Error bars indicate two standard errors of the mean.}
  \label{fig:uSBM_binary}
\end{figure}

In our next set of experiments for binary networks, we return to the setting where the latent positions are Dirichlet-distributed, this time in the sparse binary RDPG.
To obtain our latent positions, we simulated from a Dirichlet distribution with $\mX_i, \dots, \mX_n \overset{\iid}{\sim} \Dir(\alpha)$ and $\alpha = (1, 1, 1, 1, 1)$. The resulting adjacency matrix was then obtained according to
    \begin{equation} \label{eq:def:sparseDiri}
        \Prob{\mA \mid \latent} = \prod_{i < j}(\rho_n\latent_i^\top\latent_j)^{A_{ij}}(1-\rho_n\latent_i^\top\latent_j)^{1-A_{ij}}
    \end{equation}
where $\rho_n = n^{-\gamma}$ for $\gamma \ge 0$. 
Recall that smaller values of $\rho_n$ (i.e., higher levels of sparsity, i.e., larger values of $\gamma$) tends to correspond to more challenging estimation.
Given $\mA$ generated according to Equation~\eqref{eq:def:sparseDiri}, we calculated its ASE for varying choices of embedding dimension, and determined the error in $(\tti)$-norm between these estimated latent positions the true latent positions $\latent$ after finding an optimal alignment between them via solving the orthogonal Procrustes problem as described in Equation \eqref{eqn:expts:procrustes}. 
We conducted this experiment with network sizes $n = 300, 600, \dots, 7800$ and under various sparsity rates $\rho_n = n^{-\gamma}$ for $\gamma = 0,0.1,0.2,0.3,0.4,0.5$, with $80$ Monte Carlo replicates per condition.
Figure \ref{fig:sparse_binary_network} shows the results of this experiment.

\begin{figure}[ht]
\centering
\includegraphics[width=0.9\linewidth]{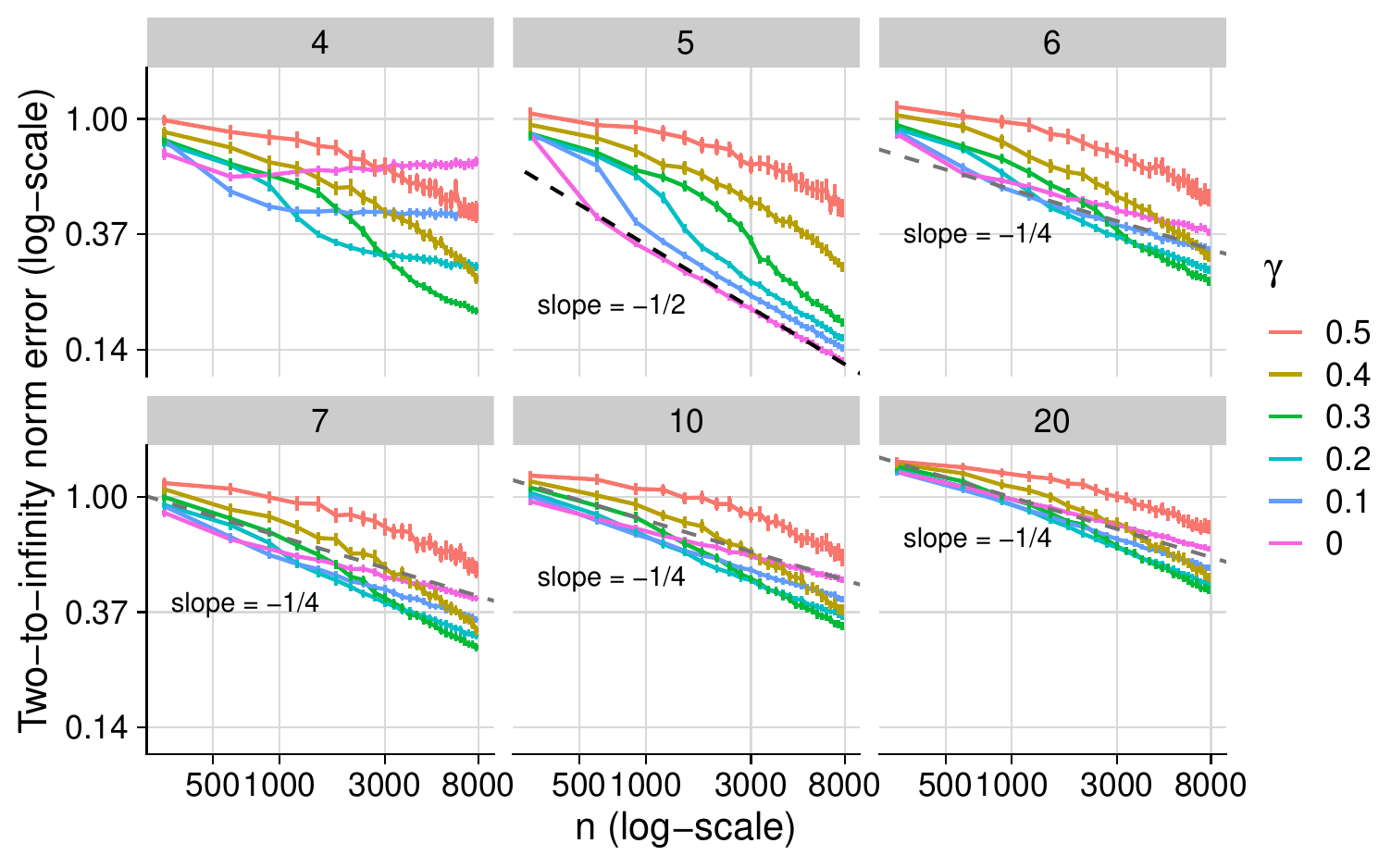}
\caption{Estimation error of ASE in $(\tti)$-norm as a function of network size $n$ under the sparse binary model in Equation~\eqref{eq:def:sparseDiri} for varying sparsity $\rho_n = n^{-\gamma}$ ($\gamma=0,0.1,0.2,0.3,0.4,0.5$; indicated by line color) for varying choices of embedding dimension $r+k=4,5,6,7,10,20$.
The $n^{-1/2}$ rate suggested by our theory for the well-specified case ($r+k=5$; top row, middle plot) is indicated by a black dashed line.
The $n^{-1/4}$ rate suggest by our theory when the embedding dimension is too large ($r+k>5$; top-row right-hand plot and bottom row) is indicated by a grey dashed line.
Errors bars indicate two standard errors of the mean.}
\label{fig:sparse_binary_network}
\end{figure}

Examining the figure, our simulations suggest that when the embedding dimension is correctly chosen ($r+k=r=5$; top row, center), the ASE converges in the $(\tti)$-norm at approximately a $n^{-1/2}$ rate for all values of $\gamma$, in agreement with existing results (up to logarithmic factors)  found in the literature \citep{lyzinski_community_2016, GRDPG_original}. 
When the embedding dimension is chosen too large (bottom row and top-row right), the ASE still converges for all values of $\gamma$, but it is a slower rate, with higher values of $\gamma$ appearing to approach this rate more quickly. 
The similarity in the convergence rate for when the embedding dimension is chosen slightly larger than the true embedding dimension ($r+k = 6$) compared to when it is chosen much larger than the true embedding dimension ($r+k = 20$) shows that the ASE is asymptotically comparatively insensitive to the embedding dimension, once it is chosen too large.
When the embedding dimension is chosen too small ($r+k=4$, top row, left), we see that the ASE is no longer consistent when the network is dense (i.e., for smaller values of $\gamma$); the performance at higher levels of sparsity (i.e., larger values of $\gamma$) is less conclusive, but extrapolating from the curves for larger values suggests the same inconsistency holds, with the curves leveling out at values of $n$ beyond $8000$.

\section{Discussion} \label{sec:disco}

In this work, we have considered how different choices of embedding dimension impact estimation of the latent positions under the RDPG, and more generally under signal-plus-noise matrix models $\mA = \mP + \mE$
Our results characterize how different choices of embedding dimension affect the estimation rate of the ASE, as measured in $(\tti)$-norm.
Theorem \eqref{theorem:main_result:weighted} and our experiments in Section~\ref{subsec:weighted_expts} show that under a weighted version of the RDPG, when the embedding dimension is chosen too small, consistency in the $(\tti)$-norm is not guaranteed for the ASE.
Indeed, we see from Theorem \eqref{theorem:main_result:weighted} that when the embedding dimension is chosen too small, the estimation error is lower bounded by the signal strength parameter $\rho_n$.
This implies that the ASE is inconsistent when $\rho_n$ is non-vanishing and the embedding dimension is chosen too small.
We may think of this as the embedding failing to capture all of the information contained in the network.
When the embedding dimension is correctly specified, all of the signal present in the network is captured, with no trailing noise eigenvalues, and we achieve the $n^{-1/2}$ convergence rate (ignoring log-factors) seen in the literature and and widely believed to be optimal \citep[see e.g., discussion in][]{YanLev2023,hao_neurips}.
On the other hand, Theorem \eqref{theorem:main_result:weighted} and our experiments in Section~\ref{subsec:weighted_expts} show that when the embedding dimension is chosen too large, consistency may still hold, so long as the embedding dimension is not too much larger than the true dimension.
This result confirms the conjecture in Section 7 of \cite{RDPG_survey}: as long as all of the signal information is captured, the ASE should still be consistent.
This consistency is a direct consequence of the delocalization of the eigenvectors associated with the perturbed null space. 
Because of the stark differences in the asymptotic behaviors under these three regimes (embedding dimension either too small, chosen correctly, or too large), our results lend support to the folklore in network embeddings and in model selection more broadly, that it is better to err on the side of choosing model rank too large.

The most immediate avenue for future work is to extend our delocalization proofs to the setting where the matrix of entry-wise variances of the adjacency matrix is no longer doubly stochastic, as described in Section \ref{subsec:consistencyForBinary}.
This would prove our Conjecture~\ref{conj:maintext:deloc}, with which we could extend Theorem \ref{theorem:main_result:weighted} to the binary RDPG.
Our experiments for binary networks described in Section \ref{subsec:expt:binary} show that the behavior of the ASE is similar to that predicted by Theorem \ref{theorem:main_result:weighted}, lending strong evidence to the validity of our delocalization conjecture.
Other directions for future work include investigating the effects of model misspecification on downstream tasks, such as community detection and hypothesis testing; and extending our results to other node embedding methods such as the Laplacian spectral embedding \citep{Tang-Laplacian}, Laplacian eigenmaps \citep{Belkin} and multiple-network versions of the ASE \citep{levin_central_2019}.

\bibliography{RMT_Embeddings}

\appendix

\section{Bounds on Estimation Error} \label{apx:estimation-error}

Here we provide proofs related to our RDPG estimation results, Lemma~\ref{lemma:decomp} and Theorem~\ref{theorem:main_result:weighted}.
We begin by establishing our basic upper- and lower-bounds on estimation error.
%where $\etan$ is as in Equation~\eqref{eq:def:etan}.

\begin{proof}[Proof of Lemma~\ref{lemma:decomp}]
We first consider the case where $k > 0$.
    Define $\mWtilde \in \bbO_{r+k}$ according to
    \begin{equation} \label{eq:def:Wtilde}
    \mWtilde = \begin{bmatrix}
    \mW^* & \bm 0 \\
    \bm 0 & \mI_{r+1:r+k}
    \end{bmatrix},
    \end{equation}
    where $\mW^* \in \bbO_r$ is the matrix guaranteed by our assumption in Equation~\eqref{eq:def:etan}. % orthogonal matrix should hit the estimate, not the truth, following convention in https://yuxinchen2020.github.io/publications/SpectralMethods.pdf Theorem 4.2.
    Trivially,
    \begin{equation} \label{eq:Wtilde:inplay}
    \min_{\mW \in \bbO_{r+k}} \twoinfnorm{\mXhat_{1:r+k} \mW - \rho_n^{1/2}\mX_{1:r+k} } 
    \le \twoinfnorm{\mXhat_{1:r+k} \mWtilde - \rho_n^{1/2}\mX_{1:r+k} }.
    \end{equation}
    
    Now, express $\mXhat_{1:r+k}$ as the block matrix
    \begin{equation*}
    \mXhat_{1:r+k} = \begin{bmatrix}
    \mXhat_{1:r} & \mXhat_{r+1:r+k}
    \end{bmatrix} .
    \end{equation*}
    Decomposing $\rho_n^{1/2}\mX_{1:r+k}$ similarly, recalling from Equation~\eqref{eq:def:X_r+k} that the trailing $k$ columns of $\mX$ are zero by construction, and using the definition of $\mWtilde$ from Equation~\eqref{eq:def:Wtilde}, we have
    \begin{equation*}
        \mXhat_{1:r_k}\mWtilde - \rho_n^{1/2}\mX_{1:r+k}  
        =  \begin{bmatrix} \left( \mXhat_{1:r} \mW^* - \rho_n^{1/2} \mX_{1:r}\right) & \mXhat_{r+1:r+k} \end{bmatrix} \in \R^{n \times (r+k)}.
    \end{equation*}
    By the triangle inequality, it follows that
    \begin{equation*} 
    \twoinfnorm{\mXhat_{1:r+k} \mWtilde - \rho_n^{1/2}\mX_{1:r+k} } 
    \le \twoinfnorm{\mXhat_{1:r} \mW^* - \mX_{1:r} } + \twoinfnorm{\mXhat_{r+1:r+k} } .
    \end{equation*}
    By our choice of $\mW^*$, Equation~\eqref{eq:def:etan} allows us to bound the first right-hand term and write
    \begin{equation*} 
    \twoinfnorm{\mXhat_{1:r+k} \mWtilde - \rho_n^{1/2}\mX_{1:r+k} } 
    \le C\etan + \twoinfnorm{\mXhat_{r+1:r+k} } .
    \end{equation*}
    Applying this to Equation~\eqref{eq:Wtilde:inplay},
    \begin{equation}\label{eq:intermediate:estimation_error}
      \min_{\mW \in \bbO_{r+k}} \twoinfnorm{\mXhat_{1:r+k} \mW - \rho_n^{1/2}\mX_{1:r+k} } 
    \le C\etan + \twoinfnorm{\mXhat_{r+1:r+k} \mW^* } .
    \end{equation}
    
    Recall that by construction, $\mXhat_{r+1:r+k} = \mUhat_{r+1:r+k}|\mShat_{r+1:r+k}|^{1/2}$.
    Thus, using basic properties of the $(\tti)$-norm \citep[see, e.g.,][]{singular_subspace},
    \begin{equation}\label{eq:delocalized_estimate_1}
        \twoinfnorm{\mXhat_{r+1:r+k} \mW^*}
        \le \twoinfnorm{\mUhat_{r+1:r+k} } \norm{ \mShat_{r+1:r+k}^{1/2} \mW^* }
        %= \norm{\mUhat_{r+1:r+k}\mShat_{r+1:r+k}^{1/2}}_{\tti} \le \norm{\mUhat_{r+1:r+k}}_{\tti}\norm{|\mShat_{r+1:r+k}|^{1/2}} \le \norm{\mUhat_{r+1:r+k}}_{\tti}\norm{\mE}^{1/2}
        \le  \twoinfnorm{\mUhat_{r+1:r+k} } \norm{\mE}^{1/2},
    \end{equation}
    where the second inequality follows from unitary invariance of the spectral norm and Weyl's inequality.
Substituting Equation \eqref{eq:delocalized_estimate_1} into Equation \eqref{eq:intermediate:estimation_error} completes the claim for the case of $k > 0$.

For the lower-bound half of our claim, which holds when $k < 0$, notice that by definition of the $(\tti)$-norm as a maximum, recalling the definition of $\mXhat_{1:r}^\circ$ from Equation~\eqref{eq:def:Xcirc}, we have
\begin{equation} \label{eq:tti:keyLB}
\min_{\mW \in \bbO_{r}} \twoinfnorm{\mXhat_{1:r}^\circ \mW - \rho^{1/2}_n\latent_{1:r}}^2
\ge
\frac{1}{n}
\min_{\mW \in \bbO_{r}}
\left\| \mXhat_{1:r}^\circ \mW - \rho^{1/2}_n\latent_{1:r} \right\|_F^2 .
\end{equation}
Expanding the Frobenius norm, for any $\mW \in \bbO_{r}$, 
\begin{equation*} 
\left\| \mXhat_{1:r}^\circ \mW - \rho^{1/2}_n\latent_{1:r} \right\|_F^2
= \left\| \mXhat_{1:r}^\circ \mW \right\|_F^2 + \left\| \rho^{1/2}_n\latent_{1:r} \right\|_F^2
- 2 \Tr (\mXhat_{1:r}^\circ \mW)^\top \left( \rho_n^{1/2} \latent_{1:r} \right) .
\end{equation*}
Using unitary invariance of the Frobenius norm and noting that $\norm{\mXhat_{1:r}^\circ}_F = \norm{\mXhat_{1:r+k}}_F$ by construction,
\begin{equation} \label{eq:frobnorm:decomp}
\left\| \mXhat_{1:r}^\circ \mW - \rho^{1/2}_n\latent_{1:r} \right\|_F^2
= \left\| \mXhat_{1:r+k} \right\|_F^2 + \left\| \rho^{1/2}_n \latent_{1:r} \right\|_F^2
- 2 \Tr (\mXhat_{1:r}^\circ \mW)^\top \left( \rho_n^{1/2} \latent_{1:r} \right) .
\end{equation}

Recalling that $\sqrt{\shat_j}$ and $\sqrt{\spop_j}$ are the singular values of $\mXhat_{1:r}$ and $\rho_n^{1/2} \latent_{1:r}$, respectively, the Von Neumann trace inequality \citep[see, e.g.,][Theorem 7.4.1.1]{Horn_Johnson} yields
\begin{equation*}
\left| \Tr (\mXhat_{1:r}^\circ \mW)^\top (\rho_n^{1/2} \latent_{1:r}) \right|
\le \sum_{j=1}^r \sqrt{ \spop_j \shat_j }
= \sum_{j=1}^{r+k} \sqrt{ \spop_j \shat_j },
\end{equation*}
where we have used the fact that the singular values of $\mXhat_{1:r}^\circ$ are invariant to right-multiplication by $\mW$ and the fact that $\mXhat_{1:r+k}$ has at most $r+k$ non-zero singular values by construction.
Applying this bound in Equation~\eqref{eq:frobnorm:decomp},  
\begin{equation*} \begin{aligned}
\left\| \mXhat_{1:r}^\circ \mW - \rho^{1/2}_n\latent_{1:r} \right\|_F^2
&\ge \sum_{j=1}^{r+k} \shat_j + \sum_{j=1}^r \spop_j
- 2 \sum_{j=1}^{r+k} \sqrt{ \spop_j \shat_j } \\
&= \sum_{j=1}^{r+k} \left( \sqrt{ \shat_j} - \sqrt{\spop_j} \right)^2
+ \sum_{j=r+k+1}^r \spop_j .
\end{aligned} \end{equation*}
Trivially lower-bounding the first sum by zero,
\begin{equation*}
\left\| \mXhat_{1:r}^\circ \mW - \rho^{1/2}_n\latent_{1:r} \right\|_F^2
\ge \sum_{j=r+k+1}^r \spop_j .
\end{equation*}
Applying this lower-bound to Equation~\eqref{eq:tti:keyLB} and taking square roots completes the proof.
\end{proof}

We now turn to a proof of our general result on asymptotic behavior of the ASE, using Lemma~\ref{lemma:decomp}.

\begin{proof}[Proof of Theorem~\ref{theorem:main_result:weighted}]
    We first consider the case of $k > 0$.
    Applying Lemma \ref{lemma:decomp}, there exists a sequence of $\mW \in \bbO_{r+k}$ such that
    \begin{equation}\label{eqn:apx_main:chkpt1}
    \twoinfnorm{\mXhat_{1:r+k} \mW - \rho_n^{1/2}\latent_{1:r+k}} 
    \le C\etan + |\shat_{r + 1}|^{1/2}\norm{\mUhat_{r+1:r+k}}_{\tti} .
    \end{equation}
    To bound $ |\shat_{r + 1}|$, standard matrix concentration inequalities \citep[see, e.g.,][]{Vershynin_2018} yield $\| \mE \| \le C \sigma \sqrt{n} \log n$ \citep[see also Lemma 4.3 in][]{Erdos}. %, $\norm{\mE}/\sqrt{n} \le 2 + o(1)$. 
    Combining this with Weyl's inequality,
    \begin{equation*}
       |\shat_{r + 1}|^{1/2}\norm{\mUhat_{r+1:r+k}}_{\tti}
       %\lesssim  n^{1/4}\norm{\mUhat_{r+1:r+k}}_{\tti}  
        \le C n^{1/4} (\sigma \log n)\max_{j \in [n]}\left(\sum_{\alpha = r + 1}^{r + k} \uhat_{j\alpha}^2\right)^{1/2} 
        \le C n^{1/4} (\sigma \log n) \left(k \max_{j \in [n]}\max_{\alpha > r} \uhat_{j\alpha}^2\right)^{1/2} .
    \end{equation*}
    Then by Theorem \ref{theorem:main:deloc}, we have
    \begin{equation}\label{apx:main_test_proofs_thm1_chkpt2}
      |\shat_{r + 1}|^{1/2}\norm{\mUhat_{r+1:r+k}}_{\tti} 
      \le C n^{1/4} (\log n)\sqrt{k \sigma^2} \left(\frac{r^4(\log n)^{8+12\gamma}}{n} \right)^{1/2}
    \le C \sqrt{k \sigma^2}\frac{r^2(\log n)^{5+6\gamma}}{n^{1/4}}
    \end{equation}
    with high probability. Substituting Equation \eqref{apx:main_test_proofs_thm1_chkpt2} back into Equation \eqref{eqn:apx_main:chkpt1}, we obtain
    \begin{equation*}
    \min_{\mW \in \bbO_{r + k}} \twoinfnorm{\mXhat_{1:r+k} \mW - \rho_n^{1/2}\latent_{1:r+k} } 
    \le C\etan + C\sqrt{k \sigma^2}\frac{r^2(\log n)^{5 + 6\gamma}}{n^{1/4}}
    \end{equation*}
    with high probability, which establishes our estimation error upper-bound.
    % For $k < 0$, we apply Lemma \ref{lemma:decomp} with $\rho_n = 1$ again to obtain
    
    For the the case $k < 0$, by Lemma \ref{lemma:decomp},
    \begin{equation*}
    \min_{\mW \in \bbO_{r}} \twoinfnorm{\mXhat_{1:r}^\circ \mW - \rho_n^{1/2}\latent_{1:r}}
    \ge \sqrt{ \sum_{j=r+k+1}^r \frac{ \spop_j }{ n } }.
    \end{equation*}
    By Assumption \ref{assumption_main:eigen_bound}, %($\spop_j \asymp \rho_n n$ for $j \in [r]$),
    it follows that
    \begin{equation*}
    \min_{\mW \in \bbO_{r}} \twoinfnorm{\mXhat_{1:r}^\circ \mW - \rho_n^{1/2}\latent_{1:r}}
    \ge C\sqrt{\rho_n |k|} ,
    \end{equation*}
    as desired.
\end{proof}

\section{Basic Results from Spectral Methods} \label{apx:spectral}

Here we collect a handful of basic results related to matrix spectra and eigenspaces for use in Appendix~\ref{apx:deloc}.

\begin{lemma} \label{lemma:perturbed_eigengap}
Let $\{\spop_i\}$ be the eigenvalues of a rank $r$ matrix $\mP$, and let $\{\shat_i\}$ be the eigenvalues of $\mA = \mP + \mE$. 
%There is an that eigengap obeys $\Delta = \Omega\left(\rho_n n\right)$. 
Suppose Assumption \eqref{assumption_main:eigen_nondegen}  holds and define $\Delta = \min_{j \in [r]} \min_{i\ne j} |\spop_i - \spop_j|$.
Then for $j \in [r]$ and $i \neq j$, 
\begin{equation}
     |\shat_i - \spop_j| \ge \Delta - \norm{\mE} .
\end{equation}
\end{lemma}
\begin{proof}
Fix $j \in [r]$ and let $i \neq j$.
By the reverse triangle inequality,
\begin{equation*}
    \left| \shat_i - \spop_j \right|
    \ge \left| | \shat_i-\spop_i | - | \spop_j-\spop_i| \right|
    \ge \Delta - \left| \shat_i - \spop_i \right|,
\end{equation*}
where the second inequality follows from the definition of $\Delta$.
Applying Weyl's inequality, it follows that
\begin{equation*}
    \left| \shat_i - \spop_j \right|
    \ge \Delta - \| \mE \|,
\end{equation*}
completing the proof.
\end{proof}

\begin{lemma}\label{lemma:pop_eigenvalues_bound} % checked
    Suppose that $\latent_1, \latent_2,\dots,\latent_n$ are identically distributed $r$-dimensional random vectors with $\norm{\mX_1} \le M < \infty$ almost surely and define $\latent = [\latent_1 \cdots \latent_n]^\top \in \R^{n \times r}$.
    Suppose that $r$ obeys Assumption~\eqref{assumption_main:rank}.
    Then with high probability, $\left\| \latent \latent^\top - n \E{\mX_1 \mX_1}) \right\| \le Cr \sqrt{n \log n}$.
\end{lemma}
\begin{proof}
    %Note that for $i \in [n]$, $\lambda_i(\mP) = \lambda_i(\rho_n\mX \mX^\top) = \lambda_i(\rho_n\mX^\top \mX)$.
    Since $\norm{\mX_1} \le M$ almost surely, the coordinates of $\mX_1$ are bounded random variables.
    By Hoeffding's inequality, we see that for any $j,k \in [r]$,
    \begin{equation*}
    \Prob{ \abs{ \left(\latent^\top\latent\right)_{jk} - n [\E{\latent_1\latent_1^\top}]_{jk}} \ge C\sqrt{n\log n} } \le \frac{2}{n^2} .
    \end{equation*}
    From a union bound over $j,k \in [r]$, we have
    \begin{equation*}
        \Prob{ \norm{ \latent^\top\latent - n \E{\latent_1\latent_1^\top}}^2_F \ge r^2(C\sqrt{n\log n})^2 } \le \frac{2r^2}{n^2} .
    \end{equation*}
    Using the fact that the Frobenius norm is an upper bound on the spectral norm, it follows that
    \begin{equation*} 
        \Prob{ \norm{ \latent^\top\latent - n \E{\latent_1\latent_1^\top}} \ge Cr\sqrt{n\log n} } \le \frac{2r^2}{n^2} ,
    \end{equation*}
    and the result follows after using our assumption in Equation~\eqref{assumption_main:rank} to ensure that this bound converges to zero suitably quickly. 
\end{proof}

Bounds similar to the following appear in \cite{levin_central_2019}.
We give this stand-alone version for ease of reference and to incorporate the signal strength parameter $\rho_n$.

\begin{lemma}\label{placeholder:RDPG_delocalization} % checked
Suppose that $(\mA,\latent) \sim \RDPG(F,n)$ with signal strength parameter $\rho_n$, with the condition that the latent positions are bounded almost surely. 
Then
\begin{equation*}
\left\| \mU \right\|_{\tti}
\lesssim \frac{\norm{\latent }_{\tti} }{ \sqrt{n} } 
\end{equation*}
with high probability.
\end{lemma}
\begin{proof}
Rearranging the identity $\mU \mS \mU^\top = \mP = \rho_n \latent \latent^\top$,
\begin{equation*}
\mU = \sqrt{\rho_n} \latent \mS^{-1/2}. % X = U S^{1/2} in the "typical" RDPG notation.
\end{equation*}
Applying basic properties of the $(\tti)$-norm \citep[see, e.g.,][]{singular_subspace},
\begin{equation} \label{eq:Utti:submult}
\norm{ \mU }_{\tti}
\le \norm{ \latent }_{\tti} \norm{ \sqrt{\rho_n} \mS^{-1/2} } .
\end{equation}
Noting that the diagonal entries of $\mS/\rho_n$ are precisely the eigenvalues of $\mP/\rho_n = \latent \latent^\top$,
Lemma~\ref{lemma:pop_eigenvalues_bound} implies
\begin{equation*}
\norm{ \rho_n^{1/2} \mS^{-1/2} }
= \frac{ \rho_n^{1/2} }{ \sqrt{\spop_{r}} }
= O\left( \frac{1}{\sqrt{n}} \right)
\end{equation*}
with high probability, where we have used Assumption \eqref{assumption_main:eigen_bound}.
Applying this to Equation~\eqref{eq:Utti:submult} completes the proof.
\end{proof}

\begin{corollary} \label{cor:RDPG_delocalization:binary} % checked
Suppose that $(\mA,\latent) \sim \RDPG(F,n)$ with binary edges and sparsity parameter $\rho_n$.
%with binary edges and that $\latent_1, \latent_2,\dots, \latent_n$ are identically distributed $r$-dimensional random vectors such that $0 \le \latent_i^\top \latent_j \le 1$ for all $i,j \in [n]$.
Then, with high probability,
\begin{equation*}
\norm{\mU }_{\tti}
\lesssim \frac{1}{ \sqrt{n} } ,
\end{equation*}
\end{corollary}
\begin{proof}
Applying Lemma~\ref{placeholder:RDPG_delocalization},
\begin{equation*}
\norm{ \mU }_{\tti}
= O\left( \frac{ \norm{ \latent }_{\tti} }{ \sqrt{n} } \right)
\end{equation*}
with high probability.
The proof is complete once we note that $\twoinfnorm{\mX } \le 1$ under the binary RDPG.
\end{proof}

\section{Delocalization Results} \label{apx:deloc}

In this section, we prove that the eigenvectors of the weighted RDPG are delocalized, in the sense that the entries of the eigenvectors are bounded by $C n^{-1/2} \log^c n$, for some constants $C,c > 0$.
Our approach is based on earlier work by \cite{Erdos} proving the local semicircle law for eigenvalues in the bulk under the Erd\H{o}s–R\'{e}nyi model, the expectation of which is a rank-one real symmetric matrix.
We extend their approach to encompass the more general setting of a low-rank, real, and symmetric expectation.
For background on the technical tools used in this section, we refer the interested reader to \cite{Erdos_general_RMT, Erdos_lecture_notes, Erdos_book,AndGuiZei2010}. 

The semicircle density is defined as
\begin{equation*}
    \mu_{\text{sc}}(x) := \frac{1}{2\pi}\sqrt{4 - x^2}
\end{equation*} 
for $x \in [ -2, 2]$, with its associated Stieltjes transform for $\Im z > 0$,
\begin{equation}\label{def:msc}
    \msc(z):= \int_{\mathbb{R}} \frac{\mu_{\text{sc}}(x)}{x-z}\text{d}x .
\end{equation}

Our goal in this section is to show that for a suitably recentered and rescaled version of our adjacency matrix $\mA$, the Stieltjes transform of its empirical eigenvalue density converges to $\msc(z)$.
This is done in Theorem~\ref{thm:semicirc} below.
This result may be of independent interest to researchers working with low-rank signal-plus-noise network models like the RDPG.
Most important for the present work, however, is that we may then use this result to show that the eigenvectors of $\mA$ are delocalized, which we do in Theorem \ref{theorem:perturbed_bulk_eigenvectors_delocalized}.

\subsection*{Notational Change: Eigenvalue Ordering and Submatrices} 
For this section and this section only, we make a notational change in order to agree with the convention of the vast majority of random matrix theory results \cite[and, in particular, the convention followed in ][, whose results we extend]{Erdos}.
In what follows, we number the eigenvalues from smallest to largest, rather than the largest-to-smallest indexing used in the main text.
Thus, in this section, the eigenvalues $d_i$ of some matrix $\mD$ will be indexed so that
\begin{equation*}
    d_1 \le d_2 \le \cdots \le d_N .
\end{equation*}
That is, $d_N$ is the largest eigenvalue of $\mD$.
To mark this change of convention in what follows, we use $N$ rather than $n$ to denote the number of nodes in the network, but stress that $N=n$. The notational change is only to remind the reader of the altered indexing convention we are following.
In this and a few other changes noted below, our notation largely follows that of \cite{Erdos}.
This change in notation is for easier comparison between the results presented here and the results that they extend.
For a matrix $\mD$, we let $\mD_{\cdot j}$ denote its $j$-th column. 
The entries of a matrix will be denoted with its non-bold variant, so that the $(i,j)$ entry of $\mD$ is denoted as $D_{ij}$.
We denote by $(\mD)^{\sm i} \in \R^{(N-1) \times (N-1)}$ the principal submatrix of a matrix $\mD \in \R^{N \times N}$ obtained by removing the $i$-th row and column.
More generally, for some index set $\T \subset [N]$, we write $(\mD)^{\sm \T}$ to mean the $(N - \abs{\T})$-by$(N - \abs{\T})$ principal submatrix of $\mD$ obtained by removing the $i$-th row and column from $\mD$ for every $i \in \T$.
If $\bm{d}$ and $d$ are used to denote the eigenvectors and eigenvalues, respectively of a matrix $\mD$, we denote the eigenvectors and eigenvalues of $(\mD)^{\sm \T}$ by $\bm{d}^{(\sm\T)}_i$ and $d^{(\sm\T)}_i$, respectively.
The indices of eigenvectors and eigenvalues will be relabeled to be in the set $[N - \abs{\T}]$, preserving their ordering described in beginning of this section: $d_1 \le d_2 \le \cdots \le d_{N - \abs{\T}}$. 
However, we keep the names of indices of $\mD$ when defining $\submatrix{\mD}{\T}$. 

As an example, consider the matrix
\begin{equation*}
\mD = \begin{bmatrix}  a & b & c \\
                       d & e & f \\
                       g & h & i \end{bmatrix} .
\end{equation*}
Then \begin{equation*}
(\mD)^{\sm 2} = \begin{bmatrix}  a & c \\
                       g  & i \end{bmatrix} .
\end{equation*}
We will overload the submatrix notation for vectors in the analogous way.
For example, if a vector is given by $\bm{d} = (a, b, c)^\top$, then the vector $\submatrix{\bm d}{2} = (a, c)^\top$ has indices 1 and 3, skipping index 2. 
For an $N$-dimensional vector $\bm{d}$, let $(\bm{d})^{\sm \T}$ denote the $ \left(N - \abs{\T} \right)$-dimensional vector obtained by removing every entry of $\bm{d}$ that is indexed by $\T$. 
Because we will often be working with submatrices and their individual entries, we adopt the following notation to indicate summing over indices:
\begin{equation*}
    \sum^{(\sm \T)} : = \sum_{i \notin \T} =  \sum_{i \in [N]\sm\T} .
\end{equation*}

Furthermore, given an index set $\T \subset [N]$, we define $\mD[\T^c] \in \R^{N \times N}$ to be the matrix obtained by setting all entries in the $i$-th column and row to zero for every element $i \in \T$. 
Similarly, we let $\mD[\T]$ denote the matrix obtained by setting all entries in $\mD$ to zero {\em except} for the columns and rows indexed by $\T$. 
When the instance the index set consists of a single element $i$, we will simply denote the set by that element.
That is, if $\T = \{i\}$, then $\mD[i] := \mD[\T]$.
We denote the eigenvectors and eigenvalues of these ``zeroed-out'' matrices with the corresponding superscript.
For example, if matrix $\mD$ has eigenvector $\bm{d}$, then the corresponding eigenvector of $\mD[*]$ will be denoted as $\bm{d}^{[*]}$. 
We continue use the constants $C$ and $c$ to denote positive constants whose value may change from line to line, but their value is always constant with respect to $N$. 

The following definitions are largely adapted from \cite{Erdos}.

\begin{definition}
    An $N$-dependent event $\Omega$ holds with $(\xi, \nu)$-high probability if 
    $\Prob{\Omega^c} \le \exp\{-\nu(\log N)^{\xi} \}$
    whenever $N \ge N_0(\nu, c_0, C_0)$ for $c_0, C_0$ in Equation \eqref{def:H}.
    Moreover, for a given event $\Omega_0$, we say that $\Omega$ holds with $(\xi, \nu)$-high probability on $\Omega_0$ if  
    $\Prob{\Omega_0 \cap\Omega^c} \le \exp\{-\nu(\log N)^{\xi}\}$
    for all $N \ge N_0(\nu, c_0, C_0)$.
\end{definition}

\begin{definition} \label{def:H} 
Let $\xi \equiv \xi(N)$ be such that 
\begin{equation}\label{eq:def:xi}
    1 + c_0 \le \xi \le C_0\log(\log N) , 
\end{equation} 
where $c_0 > 0$ and $C_0 \ge 10$. 
For $N=1,2,\dots$, consider an $N \times N$ random matrix $\mH = (H_{ij})$ whose entries are real and independent up to the symmetry constraint $H_{ij} = H_{ji}$. We assume the entries of $\mH$ satisfy the moment conditions 
\begin{equation} \label{eq:def:h_moments}
    \E{H_{ij}} = 0, \quad \E{|H_{ij}|^2} = \frac{1}{N}, \quad \E{|H_{ij}|^{p}} \le \frac{C_1^p}{Nq^{p-2}}
\end{equation}
where $3 \le p \le (\log N)^{C_0\log\log N},\ C_1>0$ and $q \equiv q(N)$ is a parameter that controls the tail behavior of the noise \citep{BLM2013}, satisfying
\begin{equation}\label{eq:def:q}
    r^2(\log N)^{3\xi + 3\gamma} \le q \le C_2N^{\frac{1}{2}}
\end{equation}
for some $C_2 > 0$.
We denote the eigenvalues of $\mH$ as $\omega_1 \le \omega_2 \le \cdots \le \omega_N$.
% We do not need the eigenvectors of $H$ here.
\end{definition}

    % Let $\latent_1, \dots, \latent_N$ be identically distributed $r$-dimensional random vectors in $\mathbb{R}^r$ such that $0 \le \latent_i^\top\latent_j \le 1$ for all $i,j \in [N]$. Denote $\latent := \begin{bmatrix} \latent_1 & \cdots & \latent_N\end{bmatrix}^\top \in \mathbb{R}^{N\times r}$. Denote the $N \times N$ symmetric probability matrix $\bm P = \latent\latent^\top$ with its spectral decomposition given as
     
    % \begin{equation}\label{def:P:spectral_decomp}
    %     \bm P = \frac{1}{\sqrt{N}} \sum_{\ell=N-r+1}^N\lambda_{\ell}\bm{u}_{\ell}\bm{u}_{\ell}^\top .
    % \end{equation}
    We next consider our non-zero expectation matrix, which is a scaled version of $\mP$ from the main text (see Table \ref{tab:notation-key} below for reference).
    Define an $N \times N$ real-valued, symmetric, and non-random rank $r$ matrix $\mL$ with its spectral decomposition given as
    \begin{equation}\label{def:L:spectral_decomp}
        \mL = \sum_{\ell=N-r+1}^N\lambda_{\ell}\bm{v}_{\ell}\bm{v}_{\ell}^\top,
    \end{equation}
    where the eigenvalues are given by
    \begin{equation*}
    \lambda_1 = \lambda_2 = \cdots = \lambda_{N-r} = 0 < \lambda_{N - r + 1} < \lambda_{N-r + 2} < \cdots < \lambda_N .
    \end{equation*}

Denote the sum of $\mH$ and $\mL$ by
\begin{equation} \label{eq:def:B} 
    \mB = \mL + \mH,
\end{equation}
with spectral decomposition given by
\begin{equation*}
    \mB = \sum_{j=1}^N\lambdahat_{j}\bvhat_{j}\bvhat_{j}^\top .
\end{equation*}
In relating the setup described here back to the main text, we note that 
\begin{equation*}
    \mH = \frac{1}{\sqrt{N}} \left(\mA - \rho_n\latent\latent^\top \right) = \frac{1}{\sqrt{N}} \left(\mA - \mP \right) .
\end{equation*}
In other words, $\mB$ is a version of the observed adjacency matrix $\mA$, scaled by a factor of $1/\sqrt{N}$.
The scaling factor $1/\sqrt{N}$ is necessary in this section, because our objects of interest are related to the limiting distribution of the eigenvalues of a random matrix. 
Without this scaling factor, the eigenvalues of $\mA$ would grow unbounded with $N$ for the RDPG (see, e.g., Lemma~\ref{lemma:pop_eigenvalues_bound}), and there would be no limiting spectral distribution. 
The need for a limiting distribution of the eigenvalues is also why we assume the variance of the entries of $\mH$ are $1/N$, although this assumption can be relaxed to allow for a more general doubly stochastic matrix of the variances of the entries of $\mH$ \citep[see Section 5 of][]{erdos_bulk_2012}. 
Our goal is to show that the eigenvectors of $\mB$ (and hence also those of $\mA$) associated with the perturbed null eigenvalues are delocalized.
To do this, we study the distribution of eigenvalues of $\mB$ and $\mH$.
%Suppose $\mH$ and $\mL$ satisfy Assumptions \ref{assumption:deloc} -- \ref{assumption:moments}.

For the sake of clarity and ease of reference, we restate the Assumptions from the main text here, rewritten to match the notation in this appendix. 
Denote $\mathcal{S} := \{N-r+1, \dots, N\}$.

\begin{enumerate}[label=(\textbf{\~{A}\arabic*}),ref=\textbf{\~{A}\arabic*}]
    \item{Assumption 1 (Delocalization)}: The eigenvectors associated with the non-zero eigenvalues of $\mL$ are delocalized.
            That is, for some $\gamma \ge 0$ with no dependence on $N$,
            \begin{equation*} 
                |v_{ij}| \lesssim \frac{\deloc}{\sqrt{N}} 
            \end{equation*}
            for all $i \in [N]$, and $j \in \mathcal{S}$ where $v_{ij}$ refers to the $i$-th entry of the $j$-th eigenvector of $\mL$. To avoid confusion, this can be thought as indexing into the $(i,j)$-th entry of the matrix of eigenvectors $\mV$. 
            \label{assumption:deloc}
          %   % Note that in the rank 1 case, this is equivalent to their definition of $\bv \equiv \bv_N:=\frac{1}{\sqrt{N}}(1, \dots, 1)^\top$.
          
    %\item{Assumption 2}: The signal strength parameter $\rhon = \omega(1/\sqrt{N})$ \label{assumption:signal_strength}
    \item{Assumption 2}: The signal strength parameter $\rhon$ satisfies $\rhon \sqrt{N} \gg 1$. \label{assumption:signal_strength}
      %Reminder why this is here: then we also need an assumption about how rho_n in the main text grows. We need that, in this appendix, \| E(T) \| = o( \Deltatilde ). Using the bound on E(T), that becomes (1+r rhon \delocsq) = o( rhon \sqrt{N} ). 
        % that becomes 1) r \delocsq = o( \sqrt{N} ) and 2) rhon = \omega( 1/\sqrt{N} ).
        % (1) is already implied by the assumption on the rank (assumption 5 below)
    
    \item{Assumption 3}: The non-zero eigenvalues of $\mL$ are distinct. \label{assumption:eigen_nondegen}

    \item{Assumption 4}: The non-zero eigenvalues of $\mL$ are of order $\rhon \sqrt{N}$: $\lambda_j \asymp \rhon \sqrt{N}$ for $j \in \mathcal{S}$. \label{assumption:eigen_bound}

    \item{Assumption 5}: $\rank(\mL) \lesssim (\log n)^{\zeta}$ for some $\zeta \ge 0$. \label{assumption:rank}

    \item{Assumption 6 (Eigengap)}: The eigengap $\Deltatilde := \min_{j \in \calS}\min_{i\ne j} |\lambda_i - \lambda_j|$ of $\mL$
                obeys $\Deltatilde = \Omega\left(\rhon\sqrt{N}\right)$. \label{assumption:eigengap}

    \item{Assumption 7}: The entries of $\mH$ satisfy the moment conditions 
    \begin{equation*} 
        \E{H_{ij}} = 0, \quad \E{|H_{ij}|^2} = \frac{1}{N}, \quad \text{ and } \quad \E{|H_{ij}|^{p}} \le \frac{C_1^p}{Nq^{(p-2)}} \quad(p >2).
    \end{equation*} \label{assumption:moments}
\end{enumerate}

% \begin{definition}
%     (from \cite{Erdos}) Let $\T \subset \{1,2, \dots, N\}$. Then we define $\submatrix{H}{\T}$ as the $(N - |\T|) \times (N-|\T|)$ submatrix of $\mH$ obtained by removing all rows and columns of $\mH$ indexed by $i \in \T$. Note that we keep the names of indices of $\mH$ when defining $\submatrix{H}{\T}$. As an example if a vector is given by $\bm{b} = (a, b, c)^\top$, then the vector $\submatrix{\mB}{2} = (a, c)^\top$ has indices 1 and 3, skipping index 2. The indices of eigenvectors and eigenvalues will be relabeled to be in the set $[N - \abs{\T}]$.
% \end{definition}

We summarize the notational changes we have made in this appendix and their corresponding symbols from the main text in Table \ref{tab:notation-key}.
The main differences are that in this appendix, we have rescaled the expectation matrix and random matrices by $1/\sqrt{N}$, and eigenvalue index ordering is reversed.
\begin{table}[ht]
    \centering
    \begin{tabular}{c|c}
       Main Text  & Appendix~\ref{apx:deloc}\\
       \hline
        $n$ & $N$ \\
        $\mA$ & $\mB$ \\
        $\mP$ & $\sqrt{N} \mL$ \\
        $\mE$ & $\sqrt{N} \mH$ \\
        $\Delta$ & $\sqrt{N}\Deltatilde$ \\
        $\spop_1 \ge \spop_2 \ge \cdots \ge \spop_n$ &  $\lambda_1 \le \lambda_2 \le \cdots \le \lambda_N$
        % Old comment kept for provenance purposes: also make sure you update the assumption in the main text. It shouldT read that \Delta = \Omega( \rho n ), NO sqrt (note how Delta in the main text relates to the eigengap in this appendix, which has been rescaled
    \end{tabular}
    \caption{Dictionary of the notational correspondence between the main text and this appendix.}
    \label{tab:notation-key}
\end{table}

We now introduce the empirical eigenvalue density functions for $\mB$ and $\mH$ respectively:
\begin{equation} \label{eq:def:stieltjesTransforms} %Note: Old label was eq:def:rho}
\tilde{\mu}(x) = \frac{1}{N}\sum_{\alpha}\delta(x - \lambdahat_{\alpha})
~~~\text{ and }~~~
\mu(x) = \frac{1}{N}\sum_{\alpha}\delta(x - \omega_{\alpha}),
\end{equation}
where $\delta( \cdot )$ is the Dirac measure. 
We define corresponding empirical cumulative distribution function (integrated empirical densities) according to
\begin{equation} \label{eq:def:ntilde}
    \ntilde(x) = \frac{1}{N} \left| \{\alpha: -\infty < \lambdahat_\alpha \le x \} \right|
\end{equation}
and
\begin{equation} \label{eq:def:n}
    n(x) = \frac{1}{N} \left| \{\alpha: -\infty < \omega_\alpha \le x \} \right| .
\end{equation}
These empirical densities are not the main focus for most of this section, as it is easier to work with their Stieltjes transforms.
Showing the convergence of the Stieltjes transforms is equivalent to showing the convergence of the densities \citep[see Lemma 2,1 in][]{Erdos_lecture_notes}. 

As elsewhere, to study the Stieltjes transform, we consider the matrix resolvent \citep[see, e.g.,][]{Kato1995}.
We denote the resolvents of $\mB$ and $\mH$ by, respectively,
\begin{equation} \label{eq:def:resolvents}
    \Gtildematrix(z) := (\mB-z\mI)^{-1} 
    ~~~\text{ and }~~~
    \mathbf G(z) := \left(\mH - z \mI \right)^{-1},
\end{equation}
where $z \in \{x + \iu y : y > 0; x,y \in \R \}$.
The trace of the resolvent is related to the Stieltjes transform of the empirical eigenvalue density function of $\mB$ as \citep[see, e.g.,][Equation (2.5)]{Erdos_lecture_notes} 
\begin{equation}\label{eq:def:m-mtilde}
    \mtilde(z) :=  \int\frac{\mutilde(x)}{x-z} = \frac{1}{N} \Tr \Gtildematrix(z) dx ~~~\text{ and }~~~ m(z) := \int\frac{\mu(x)}{x-z} dx = \frac{1}{N}\Tr \mathbf G(z) .
\end{equation}
In working with submatrices of $\mB$, we denote corresponding resolvents by
\begin{equation*}
    \Gtildetmatrix := \left(\submatrix{\mB}{\T} - z \submatrix{\mI}{\T} \right)^{-1} .
\end{equation*}
Note that in defining the corresponding Stieltjes transform $\mtildet(z)$, we keep the $1/N$ normalization:
\begin{equation}\label{eqn:def:mtildet}
    \mtildet(z) := \frac{1}{N}\Tr\Gtildetmatrix .
\end{equation}
We are interested in understanding the error between $\mtilde(z)$ and $\msc(z)$.
We denote these errors by
\begin{equation} \label{eq:def:stieltjes_errors}
    \Lambdatilde(z) := |\mtilde(z) - \msc(z)|, \quad \Lambda(z) := |m(z) - \msc(z)| . % See https://arxiv.org/pdf/1601.04055
\end{equation}

Below, we will consider complex numbers, parameterized by $E$ and $\eta$ as
\begin{equation*}
    z = E + \iu\eta,
\end{equation*}
where $E \in \R$ and $\eta > 0$ are chosen so that $z$ is near the support of the semicircle law.
In particular, we consider the domain
\begin{equation}
    D := \left\{z \in \C : |E| \le \Sigma, 0 < \eta \le 3 \right\},
\end{equation}
and a subset $L$-dependent domain
\begin{equation} \label{eq:def:DL}
    D_L := \left\{z \in \C : |E| \le \Sigma, \frac{r^4(\log N)^{L+12\gamma}}{N} \le \eta \le 3 \right\}
\end{equation}
where $\Sigma \ge 3$ is a fixed, arbitrary constant and $L \equiv L(N)$ satisfies 
\begin{equation}\label{eq:L_constraints}
L \ge 8\xi,
\end{equation}
where $\xi$ is given by Equation \eqref{eq:def:xi} in Definition~\ref{def:H}.

\subsection{Main Estimates}
The following bounds are adapted from Section 7 of \cite{Erdos}.
We make modifications to allow for a more general choice of $\mL$ having rank $r\ge1$.
In particular, we allow for $\mL$ to have finite, possibly slowly-growing rank, rather than the case of $\rank \mL = 1$ considered in \cite{Erdos}.
\begin{lemma} %[Adapted from \cite{Erdos}, Lemma 7.1] 
\label{lemma:7.1}
     Let $\mB$ satisfy Equation~\eqref{eq:def:B}. Then for any $z \in D_L$, we have 
    \begin{equation*}
        \left| \Lambdatilde(z) - \Lambda(z) \right| \le \frac{r\pi}{N\eta}
    \end{equation*}
\end{lemma}
\begin{proof}
The proof largely follows that of Lemma 7.1 in \cite{Erdos}.
The key difference is in the relation between $\ntilde(x)$ and $n(x)$.
In our setting, $|\ntilde(x) - n(x)| \le r/N$, which can be seen from Lemma \ref{lemma:weyl} and noticing that the eigenvalues of $\mH$ and $\mB$ are interlaced at most every $r$ indices, as we now demonstrate.
Recalling the definitions of $\Lambda$ and $\Lambdatilde$ from Equation~\eqref{eq:def:stieltjes_errors} and applying the reverse triangle inequality,
\begin{equation} \label{eq:Lambda2Lambdat:step1}
|\Lambdatilde(z) - \Lambda(z)|
\le |\mtilde(z) - \msc(z) - m(z) + \msc(z) | 
= |\mtilde(z) - m(z) | .
\end{equation}
Recalling the definitions of $\mtilde(z)$ and $m(z)$ from Equation~\eqref{eq:def:m-mtilde}, 
\begin{equation*}
\mtilde(z) - m(z)
= \int\frac{\mutilde(x) - \mu(x)}{x-z}dx 
= \int\frac{\ntilde(x) - n(x)}{(x-z)^2}dx ,
\end{equation*}
where the second equality follows from integration by parts. 
Substituting this into Equation~\eqref{eq:Lambda2Lambdat:step1} and applying Jensen's inequality,
\begin{equation} \label{eq:Lambdatilde2Lambda}
|\Lambdatilde(z) - \Lambda(z)|
\le \int\frac{| \ntilde(x)-n(x) |}{|x-z|^2}dx .
\end{equation}
Then Lemma \ref{lemma:weyl} implies the interlacing
\begin{equation*}
    \omega_{i} \le \omega_{i+1} \le \dots \le \omega_{i+r-1} \le \lambdahat_i \le \dots \le \lambdahat_{i+r-1} \le \omega_{i+r} \le \lambdahat_{i+r}.
\end{equation*}
From these inequalities, we can see that there are at most $r$ eigenvalues of $\mH$ between the $i$-th eigenvalue of $\mH$ and the $i$-th eigenvalue of $\mB$, since the $(i+r)$-th eigenvalue of $\mH$ must come after the $i$-th eigenvalue of $\mB$.
Thus, for any $x$,
\begin{equation*}
    \left| \left|\{\alpha: -\infty < \omega_\alpha \le x \} \right| - \left|\{\alpha: -\infty < \lambdahat_\alpha \le x \} \right|\right| \le r .
\end{equation*}
It then follows that $|\ntilde(x) - n(x)| \le r/N$.
Applying this to Equation~\eqref{eq:Lambdatilde2Lambda}, %% https://www.integral-table.com, equation 9
\begin{equation*}
|\Lambdatilde(z) - \Lambda(z)|
\le \frac{r}{N}\int_{-\infty}^{\infty}\frac{1}{|x-z|^2}dx
= \frac{r}{N}\int_{-\infty}^{\infty}\frac{1}{(x-E)^2 + \eta^2}dx
= \frac{r}{N\eta}\lim_{t \rightarrow \infty}\left[\tan^{-1}(t) - \tan^{-1}(-t) \right]
= \frac{r\pi}{N\eta} ,
\end{equation*}
as we set out to show.
\end{proof} 

\begin{lemma}\label{lemma:leave_one_out_E}
Let $\mB$ satisfy Equation \ref{eq:def:B}. For any $\T \subset [N]$ such that $0 < | \T |\le \tau < N$, where $\tau$ is a constant with respect to $N$, denote $\mE(\T) :=\mathbf  H[\T^c] - \mL[\T]$. Then 
\begin{equation}
     \norm{\mE(\T)} \le C(1 +  r \sqrt{\tau} \rhon\delocsq) 
\end{equation}
with $(\xi, \nu)$-high-probability.
\end{lemma}
\begin{proof}
Lemma 4.3 from \cite{Erdos} states that with $(\xi, \nu)$-high-probability, we have 
\begin{equation}
    \norm{\mH} \le 2 + \frac{(\log N)^{\xi}}{q^{1/2}}.
\end{equation}
Recalling that $r^2(\log N)^{3\xi + 3\gamma}\le q$ by Equation \eqref{eq:def:q}, 
\begin{equation} \label{eq:H:spectral} 
    \norm{\mH} %\le 2 + \frac{(\log N)^{\xi}}{q^{1/2}} 
    \le 2 + \frac{(\log N)^{\xi}}{(\log N)^{3\xi/2}} \le C. %2 + o(1).  
\end{equation}
Now for any $\T \subset [N]$, we have
\begin{equation*}
\norm{\mE(\T)}
= \norm{\mH[\T^c] - \mL[\T]}
\le \norm{\mH[\T^c]} + \norm{\mL[\T]}.
\end{equation*}
By Theorem 4.3.17 from \cite{Horn_Johnson}, we have that $\norm{\mH[\T^c]} \le \norm{\mH}$ for any $\T \subset [N]$.
Thus, using Equation~\eqref{eq:H:spectral}, it holds with $(\xi, \nu)$-high-probability that
%\begin{equation*}
%\norm{\mH[\T^c]} \le 2 + o(1),
%\end{equation*}
%so that with $(\xi, \nu)$-high-probability,
\begin{equation} \label{eq:normE:checpkpt}
   \norm{\mE(\T)} \le C + \norm{\mL[\T]} .
\end{equation}
To obtain an upper bound for $\norm{\mL[\T]}$, recall from Equation~\eqref{def:L:spectral_decomp} that
\begin{equation*}
   \mL = \sum_{\ell = N - r + 1}^N \lambda_{\ell}\bv_{\ell}\bv_{\ell}^\top .
\end{equation*}
By Assumption \eqref{assumption:eigen_bound}, $c\rhon\sqrt{N} \le \lambda_{\ell} \le C\rhon\sqrt{N}$ for all $\ell \in \{N-r+1, N-r+2, \dots, N \}$, and thus for each $k,j \in [N]$, we have
\begin{equation}\label{eqn:L-entry-bound}
\left| L_{k j} \right|
= \left| \left(\sum_{\ell = N - r + 1}^N \lambda_{\ell}\bv_{\ell}\bv_{\ell}^\top \right)_{k j} \right|
\le C\rhon\sqrt{N} \left| \left(\sum_{\ell = N - r + 1}^N\bv_{\ell}\bv_{\ell}^\top \right)_{k j} \right|.
\end{equation}
By Assumption \eqref{assumption:deloc}, $\bv_\ell\bv_\ell^\top$ has entries of order $\delocsq/N$, whence
\begin{equation*}
\left| L_{kj} \right|
\le C\left( \frac{r\rhon\delocsq}{\sqrt{N}} \right) .
\end{equation*}
It follows that, since $\mL[\T]$ has $2N\tau - \tau^2$ non-zero entries by construction,
\begin{equation*}
%\left\| P[i] \right\|^2 \le \left\| P[i] \right\|_F^2 = O( r^2 ) .
\norm{\mL[\T] }
\le \norm{\mL[\T] }_F
\le C( r \sqrt{\tau} \rhon\delocsq ) .
\end{equation*}
Applying this bound to Equation~\eqref{eq:normE:checpkpt} completes the proof.
%, we have our claim with $(\xi, \nu)$-high-probability,
%\begin{equation*} 
%   \norm{\mE(\T)} = O(1 +  r \sqrt{\tau} \rhon\delocsq).
%\end{equation*}
\end{proof} %Checked by KDL, Jan 5, 2026, 17:15 

\begin{lemma}\label{lemma:subspace_bound}
For each $\alpha=1,2,\dots,N$, let $\bvhat_{\alpha}^{[\T^c]}$ be the orthonormal eigenvectors of $\mB[\T^c] = \mL[\T^c] + \mH[\T^c]$. Denote $\mE(\T) = \mH[\T^c] - \mL[\T]$. 
Then for all $j \in \{N-r+1, N-r+2, \dots, N \}$ and any $\T \subset [N]$ where $0 < |\T| \le \tau$, we have with $\xinu$-high probability 
\begin{equation*}
     \sqrt{\sum_{\alpha\ne j}|\bv_{j}^\top\bvhat_{\alpha}^{[\T^c]}|^2 } 
     \le \frac{C r\delocsq }{\lambda_j} ,
\end{equation*}
where $\bv_j$ is the eigenvector associated with the $j$-th smallest eigenvalue of $\mL$.
\end{lemma}
\begin{proof}
Following the proof of the Davis-Kahan $\sin \Theta$ theorem from Chapter 2 in \cite{spectral_methods_textbook}, 
for a given $j \in \{N - r + 1, \dots, N \}$, we represent $\mL$ via the eigendecomposition
\begin{equation}\label{eq:weyl_P_decomp}
    \mL = \begin{bmatrix}\bv_j & \mU_{\perp} \end{bmatrix}
    \begin{bmatrix}\lambda_j & 0\\
                    0 & \mathbf M_{\perp} \end{bmatrix}
                    \begin{bmatrix}\bv_j & \mU_{\perp} \end{bmatrix} ^\top,
\end{equation}
where $\mU_{\perp}$ is the $(N-1) \times (N-1)$ matrix of eigenvectors that span the orthogonal complement of $\bv_j$, and $\mathbf M_{\perp}$ is the diagonal matrix of associated eigenvalues.
Write the eigendecomposition of $\mB[\T^c]$ similarly:
\begin{equation}\label{eq:weyl_A_decomp}
    \mB[\T^c] = \begin{bmatrix}\bvhat_j^{[\T^c]} &\Uhatperp^{[\T^c]}
    \end{bmatrix}
    \begin{bmatrix}\lambdahat_j^{[\T^c]} & 0\\
                    0 & \Mhatperp^{[\T^c]} \end{bmatrix}
                    \begin{bmatrix}\bvhat_j^{[\T^c]} &\Uhatperp^{[\T^c]} \end{bmatrix} ^\top,
\end{equation}
where $\bvhat_j^{[\T^c]}$ is the eigenvector associated with $\lambdahat_j^{[\T^c]}$, the $j$-th smallest eigenvalue of $\mB[\T^c]$, $\Uhatperp^{[\T^c]}$ is the matrix of orthonormal eigenvectors that span the orthogonal complement of $\bvhat_j^{[\T^c]}$, and $\Mhatperp^{[\T^c]}$ is the diagonal matrix of associated eigenvalues. 

By construction, $\mB[\T^c] = \mL + \mH[\T^c] - \mL[\T] = \mL + \mE(\T)$, and thus we may view $\mE(\T)$ as the error between $\mL$ and $\mB[\T^c]$. That is,
\begin{equation}\label{eqn:apx:davis-kahan-mE}
    \mE(\T) = \mB[\T^c] - \mL
\end{equation}
We will now employ a centering trick to bound the eigenvalues of $\Mhatperp^{[\T^c]}$ away from $0$, by working with the surrogate matrices 
\begin{equation}
    \mLtilde = \mL - \lambda_j\mI \quad \text{and} \quad \tilde{\mB}[\T^c] = \mB[\T^c] - \lambda_j\mI.
\end{equation}
From these surrogate matrices, we can see that
\begin{equation*}
    \mE(\T) = \mB[\T^c] - \mL = \mB[\T^c] - \lambda_j\mI  + \lambda_j\mI - \mL = \mBtilde[\T^c] - \mLtilde .
\end{equation*}

We observe that by Lemma \ref{lemma:perturbed_eigengap}, the eigenvalues on the diagonal of $\Mhatperp^{[\T^c]}$ are contained in the set $(-\infty,\lambda_j - \Deltatilde + \norm{\mE(\T)}] \cup [\lambda_j + \Deltatilde - \norm{\mE(\T)}, \infty)$.
The eigenvalues of $\mLtilde$ and $\tilde{\mB}[\T^c]$ are derived by applying the map $x \mapsto x - \lambda_j$ to the spectra of $\mL$ and $\mB[\T^c]$, respectively, and therefore,
\begin{equation} \label{eq:Mtilde:svLB}
\sigma_{\min}(\mMtildeperp^{[\T^c]}) \ge \Deltatilde - \norm{\mE(\T)}.
\end{equation}

We denote the eigendecomposition of these surrogate matrices using the same notation as in Equations~\eqref{eq:weyl_P_decomp} and \eqref{eq:weyl_A_decomp} but with an additional tilde (~$\tilde{}$~) above each object.
%For example, $\mUhattildeperp^{[\T^c]}$ spans the eigenspace orthogonal to $\tilde{\bvhat}_j^{[\T^c]}$ of $\mBtilde[\T^c]$ th
Noting that the surrogate matrices share the same eigenvectors as their original matrices, we have
\begin{equation} \label{eq:surragateEvecs}
\Utildeperp^{[\T^c]} = \mU_{\perp}
~~~\text{ and }~~~
\mUhattildeperp^{[\T^c]} = \Uhatperp^{[\T^c]} .
\end{equation}
We then examine the quantity
\begin{equation*}
\left( \mUhattildeperp^{[\T^c]} \right)^\top\mE(\T)\tilde{\bv}_j 
= \left( \Uhatperp^{[\T^c]} \right) ^\top\mE(\T)\bv_j
=\left( \Uhatperp^{[\T^c]} \right)^\top\left(\mBtilde[\T^c] - \mLtilde \right)\bv_j 
= \mMhattildeperp^{[\T^c]}\left( \Uhatperp^{[\T^c]} \right)^\top \bv_j ,
\end{equation*}
where the first equality follows from Equation~\eqref{eq:surragateEvecs}; 
the second equality is obtained from substituting $\mE(\T)$ from Equation~\eqref{eqn:apx:davis-kahan-mE}; 
and the third equality follows from the fact that $\mLtilde \bv_j = 0$ by construction.
Taking the Frobenius norm, we have
\begin{equation*}\begin{aligned}
\norm{\left( \Uhatperp^{[\T^c]}\right)^\top \mE(\T)\bv_j}_F 
&=\norm{\mMhattildeperp^{[\T^c]}\left( \Uhatperp^{[\T^c]}
\right)^\top\bv_j}_F
\ge \sigma_{\min}\left(\mMhattildeperp^{[\T^c]}\right) \norm{\left( \Uhatperp^{[\T^c]}
\right)^\top\bv_j}_F \\
   & \ge \left(\Deltatilde - \norm{\mE(\T)} \right)\norm{\left( \Uhatperp^{[\T^c]}
\right)^\top\bv_j}_F ,
\end{aligned} \end{equation*}
where the first inequality is from basic properties of unitarily invariant norms and the second inequality follows from Equation~\eqref{eq:Mtilde:svLB}. 
Hence, by basic properties of the norms,
\begin{equation*}
\left(\Deltatilde - \norm{\mE(\T)} \right)\norm{\left( \Uhatperp^{[\T^c]} \right)^\top\bv_j}_F
\le \norm{\left( \Uhatperp^{[\T^c]} \right)^\top\mE(\T)\bv_j}_F 
\le \norm{\left( \Uhatperp^{[\T^c]}\right)^\top}\norm{\mE(\T)\bv_j}_F
\le \norm{\mE(\T)\bv_j}_F .
\end{equation*}
Lower bounding $\Deltatilde$ using Assumption~\eqref{assumption:eigengap} and upper-bounding $\norm{\mE(\T)}$ using Lemma~\ref{lemma:leave_one_out_E}, Assumptions~\eqref{assumption:signal_strength} and~\eqref{assumption:rank} ensure that $\Deltatilde - \norm{\mE(\T)} > 0$ for suitably large $n$ and thus
\begin{equation*}
\norm{\left( \Uhatperp^{[\T^c]} \right)^\top\bv_j}_F 
\le \frac{\norm{\mE(\T)\bv_j}_F}{\Deltatilde - \norm{\mE(\T)} }.
\end{equation*}
Now since 
\begin{equation*}
    \sqrt{\sum_{\alpha \ne j}|\bv_{j}^\top\bvhat_{\alpha}^{[\T^c]} |^2 } = \norm{\left( \Uhatperp^{[\T^c]}
\right)^\top\bv_j}_F,
\end{equation*}
we have 
\begin{equation*}
\sqrt{\sum_{\alpha \ne j}|\bv_{j}^\top\bvhat_{\alpha}^{[\T^c]}|^2 }
\le \frac{\norm{\mE(\T)\bv_j}_F}{\Deltatilde - \norm{\mE(\T)}} \le \frac{\norm{\mE(\T)}}{\Deltatilde - \norm{\mE(\T)}} \le \frac{C(1 +  r  \rhon\delocsq)}{\Deltatilde - (1 +  r \rhon\delocsq)}
\end{equation*}
where the last inequality holds with $(\xi, \nu)$-high-probability, by Lemma \ref{lemma:leave_one_out_E}.
Note that we have absorbed $\tau$ into the constant $C$.

Hence, by Assumption~\ref{assumption:eigengap}, which states $\Deltatilde = \Omega\left(\rhon\sqrt{N}\right)$, we have shown that
\begin{equation*} \begin{aligned}
\sqrt{\sum_{\alpha \ne j}|\bv_{j}^\top\bvhat_{\alpha}^{[\T^c]} |^2 }
\le \frac{C(1 +  r\tau  \rhon\delocsq)}{\Deltatilde - (1 +  r\tau  \rhon\delocsq)}
& \le \frac{C(1 +  r \rhon\delocsq)}{\rhon\sqrt{N} - (1 +  r  \rhon\delocsq)} \\
&\hspace{-2cm}= \frac{C(1 +  r \rhon\delocsq)}{\rhon\sqrt{N}(1 - (1 +  r  \rhon\delocsq)/\rhon\sqrt{N})} .
\end{aligned} \end{equation*}
with $(\xi, \nu)$-high-probability. Then since
%(1) $\rhon = \omega(1/\sqrt{N})$ by Assumption \eqref{assumption:signal_strength} 
(1) $\rhon \sqrt{N} \gg 1$ by Assumption \eqref{assumption:signal_strength} 
and (2) $r\delocsq \ll \sqrt{N}$ by Assumption \eqref{assumption:rank}, we have that $1 + r\rhon\delocsq \ll \rhon\sqrt{N}$. It then follows that,
\begin{equation*}
\sqrt{\sum_{\alpha \ne j}|\bv_{j}^\top\bvhat_{\alpha}^{[\T^c]} |^2 }
\le 
\frac{C(1 +  r \rhon\delocsq)}{\rhon\sqrt{N}} .
\end{equation*}
with $(\xi, \nu)$-high-probability. To simplify presentation in the following sections, we will $\rhon$ trivially by 1 and absorb the other summand.
So then
\begin{equation}
\sqrt{\sum_{\alpha \ne j}|\bv_{j}^\top\bvhat_{\alpha}^{[\T^c]} |^2 }
\le \frac{C}{\rhon\sqrt{N}} r\delocsq 
\end{equation}
with $(\xi, \nu)$-high-probability.
Recalling that $\lambda_j \asymp \rhon\sqrt{N}$ by Assumption~\eqref{assumption:eigen_bound}, we have that with $(\xi, \nu)$-high-probability,
\begin{equation}
\sqrt{\sum_{\alpha \ne j}|\bv_{j}^\top\bvhat_{\alpha}^{[\T^c]} |^2 }
\le \frac{C}{\lambda_j} r\delocsq  .
\end{equation}

Noting that the above argument only relied on the probability guarantee in Lemma~\ref{lemma:leave_one_out_E} and our choice of $\T$ and $j$ were otherwise arbitrary, this estimate holds for any $\T \subset [N]$ and $j \in \{N-r+1, N-r+2, \dots, N \}$ with $(\xi, \nu)$-high-probability, as we set out to show.
\end{proof} 

The following lemma largely mirrors Lemma 7.2 from \cite{Erdos}, but there are more terms in the matrix $\calAtilde_i(z)$ owing to the fact that we now have a rank-$r$ matrix $\mL$, instead of the rank-one case considered by \cite{Erdos}.
Below, for a vector $F = (F_1,F_2,\dots,F_N)$, we write
\begin{equation} \label{eq:def:backetAverage}
\average{F} := \frac{1}{N} \sum_{i=1}^N F_i .
\end{equation}

\begin{lemma} \label{lemma:7.2}
Let $\mB$ be as given in Equation \eqref{eq:def:B} and denote the resolvent of $\mB$ by $\Gtildematrix(z) = \left(\mB - z \mI \right)^{-1}$ as in Equation~\eqref{eq:def:resolvents}.
Then
\begin{equation*}
    \Gtilde_{ii}(z) = \frac{1}{-z-\msc(z) - \left(\average{\psitilde(z)} - \Ytilde_i(z) \right)},
\end{equation*}
where
\begin{equation} \label{eq:def:Ytilde}
    \Ytilde_{i}(z) = H_{ii}
        - \Ztilde_i(z)
        +\calAtilde_i(z),
\end{equation}
and
\begin{equation}\label{eq:def:Ztilde}
    \Ztilde_i(z) = Q_i\left( (\Bi)^\top\Gtildeimatrix(z)\Bi \right)
\end{equation}
and, for a matrix $\mD$,
\begin{equation}\label{eq:def:Q_i}
    Q_i\left(\mD \right) = \mD - \E{\mD \mid \submatrix{\mH}{i} } 
\end{equation}
and
\begin{equation} \label{eq:def:calAtilde}
\begin{aligned}
    \calAtilde_i(z) &= L_{ii} - \sum_{\ell = N - r + 1}^N \lambda^2_{\ell}(\VM{\ell})^\top \Gtildeimatrix(z) \VM{\ell}  -  \sum_{\ell_1 \ne \ell_2} \lambda_{\ell_1}\lambda_{\ell_2}(\VM{\ell_1})^\top \Gtildeimatrix(z) \VM{\ell_2} \\
    % &~~~~~~~~~ - f_Nf_{N-1} (\EM{N})^\top\Gtildei(z)\EM{N-1} -f_{N-1}f_N (\EM{N-1})^\top\Gtildei(z)\EM{N} \\
    &~~~~~~~~~ + \frac{1}{N}\sum_j\frac{\Gtilde_{ij}(z)\Gtilde_{ji}(z)}{\Gtilde_{ii}(z)} .
\end{aligned}
\end{equation}
Moreover,
\begin{equation}\label{eqn:def:psi}
    \psitilde_i(z) := \Gtilde_{ii}(z) - \msc(z) .
\end{equation}
%Note that in the case where $\mL$ is rank $r=1$ and $u_N = \frac{1}{\sqrt{N}}(1, 1, \dots, 1_N)^\top$, this reduces exactly to $\calAtilde_i(z)$ in \cite{Erdos}. 
\end{lemma} 
\begin{proof}
By definition of $\Gtildematrix(z)$, taking Schur complements yields
\begin{equation} \label{eq:Gtildeii:start}
    \Gtilde_{ii}(z)
    = \left(- z + B_{ii} - \left(\Bi\right)^\top\Gtildeimatrix(z)\Bi \right)^{-1} .
\end{equation}
Looking at the quadratic form, by the definition of $\mB$ from \eqref{eq:def:B}, we have 
\begin{equation*} 
    \left(\Bi\right)^\top\Gtildeimatrix(z)\Bi
    = \left[ (\Hi)^\top 
    + \hspace{-0.5em} \sum_{\ell = N - r + 1}^{N} \hspace{-1em} \lambda_{\ell}(\VM{\ell})^\top \right] \Gtildeimatrix(z) \left[ \Hi + \hspace{-0.5em} \sum_{\ell = N - r + 1}^{N} \hspace{-1em} \lambda_{\ell}(\VM{\ell}) \right] ,
%     &= (\Hi)^\top\Gtildei(z)\Hi \\
%     &+ (\Hi)^\top\Gtildei(z)f_N\EM{N} + (\Hi)^\top\Gtildei(z) f_{N-1}\EM{N-1} \\
%     &+ f_N(\EM{N})^\top \Gtildei(z)\Hi +  f_N(\EM{N})^\top \Gtildei(z) f_N\EM{N}  \\
%     &+  f_N(\EM{N})^\top \Gtildei(z) f_{N-1}\EM{N-1} \\
%     &+  f_{N-1}(\EM{N-1})^\top \Gtildei(z)\Hi +   f_{N-1}(\EM{N-1})^\top \Gtildei(z) f_N\EM{N}  \\
%     &+   f_{N-1}(\EM{N-1})^\top \Gtildei(z) f_{N-1}\EM{N-1} 
% %\end{aligned}
\end{equation*}
where we remind the reader that $\submatrix{\mH}{i}$ denotes the $(N - 1)\times(N - 1)$ submatrix of $\mH$ where the $i$-th rows and columns were removed; the notation $\mH_{\cdot i} \in \R^N$ denotes the $i$-th row of $\mH$; and hence $\Hi \in \R^{N-1}$ denotes the $i$-th row of $\mH$ with its $i$-th entry removed.
Expanding and taking the conditional expectation given $\submatrix{\mH}{i}$, the terms that are linear in $\Hi$ disappear, since the elements of $\Hi$ are not included in the matrix $\submatrix{\mH}{i}$, and we have
\begin{equation} \label{eqn:quad_expansion_7.1}
\begin{aligned}
    \E{(\Bi)^\top\Gtildeimatrix(z)\Bi \mid \submatrix{\mH}{i}} 
    &= \E{(\Hi)^\top\Gtildeimatrix(z)\Hi \mid \submatrix{\mH}{i}} \\
    & ~~~~~~~~~+  \sum_{\ell = N - r + 1}^N \lambda^2_{\ell}~(\VM{\ell})^\top \Gtildeimatrix(z) \VM{\ell}  \\
    & ~~~~~~~~~+  \sum_{\ell_1 \ne \ell_2} \lambda_{\ell_1}~\lambda_{\ell_2}~
                (\VM{\ell_1})^\top \Gtildeimatrix(z) \VM{\ell_2} .
     % &~~~~~~~~~+  f_N(\EM{N})^\top \Gtildei(z) f_{N-1}\EM{N-1} \\
     % &~~~~~~~~~+  f_{N-1}(\EM{N-1})^\top \Gtildei(z) f_N\EM{N}   .
\end{aligned} 
\end{equation}
Recalling the definition of $\calAtilde_i(z)$ from Equation~\eqref{eq:def:calAtilde}, it follows that
\begin{equation}\label{eq:expected_quad_form_7.1}
\begin{aligned}
&\E{(\Bi)^\top\Gtildeimatrix(z) \Bi \mid \submatrix{\mH}{i}} \\
&~~~~~~~~~~~~~~~~~~~= \E{(\Hi)^\top\Gtildeimatrix(z)\Hi \mid \submatrix{\mH}{i}} - \calAtilde_i(z) + L_{ii} + \frac{1}{N}\sum_j\frac{\Gtilde_{ij}(z)\Gtilde_{ji}(z)}{\Gtilde_{ii}(z)} .
\end{aligned}
\end{equation} 
The first term in Equation \eqref{eq:expected_quad_form_7.1} can be expressed as 
\begin{equation} \label{eq:conditional_expectation_resolvent}
\begin{aligned}
    \E{(\Hi)^\top\Gtildeimatrix(z)\Hi \mid \Hi} &= \frac{1}{N} \sum_j^{(\sm i)}\Gtildei_{jj}(z)
    = \frac{1}{N}\sum_j \left(\Gtilde_{jj}(z) - \frac{\Gtilde_{ji}(z)\Gtilde_{ij}(z)}{\Gtilde_{ii}(z)} \right) \\
    &=  \frac{1}{N}\sum_j\Gtilde_{jj}(z) - \frac{1}{N}\sum_j\frac{\Gtilde_{ji}(z)\Gtilde_{ij}(z)}{\Gtilde_{ii}(z)} ,
\end{aligned}
\end{equation}
where in the first equality, we have used the fact the variances of $H_{ii}$ are $1/N$, along with the fact that $\E{H_{ij}} = 0$, and in the second equality, we have used basic properties of the entries of resolvents, given in Lemma~\ref{lemma:resolvent_identities}. 
Substituting Equation~\eqref{eq:conditional_expectation_resolvent} into Equation~\eqref{eq:expected_quad_form_7.1} yields 
%\begin{equation*}
%    \E{(\Hi)^\top\Gtildeimatrix(z)\Hi \mid \submatrix{\mH}{i}} - \calAtilde_i(z) + L_{ii} + \frac{1}{N}\sum_j\frac{\Gtilde_{ij}(z)\Gtilde_{ji}(z)}{\Gtilde_{ii}(z)} 
%    =  \frac{1}{N}\sum_j\Gtilde_{jj}(z) - \calAtilde_i(z) + L_{ii}
%\end{equation*}
%Hence,
\begin{equation} \label{eq:conditional_expectation_resolvent:simplified}
\E{(\Bi)^\top\Gtildeimatrix(z)\Bi \mid \submatrix{\mH}{i}} = \frac{1}{N}\sum_j\Gtilde_{jj}(z) - \calAtilde_i(z) + L_{ii}
\end{equation}
Recalling the definition of $Q_i( \cdot )$ from Equation~\eqref{eq:def:Q_i}, let us write
\begin{equation*}
\Ztilde_i(z) = Q_i\left( (\Bi)^\top\Gtildeimatrix(z)\Bi \right)
\end{equation*}
for ease of notation.
Now, consider
\begin{equation*}
\begin{aligned}
B_{ii} - (\Bi)^\top\Gtildeimatrix(z)\Bi &= B_{ii} - (\Bi)^\top\Gtildei(z)\Bi + \E{(\Bi)^\top\Gtildeimatrix(z)\Bi \mid \submatrix{\mH}{i}} \\
& ~~~~~~~ - \E{(\Bi)^\top\Gtildeimatrix(z)\Bi \mid \submatrix{\mH}{i}}  \\
&= B_{ii}  - \E{(\Bi)^\top\Gtildeimatrix(z)\Bi \mid \submatrix{\mH}{i}} - \Ztilde_i(z) \\
& =  B_{ii} - \left(\frac{1}{N}\sum_j\Gtilde_{jj}(z) - \calAtilde_i(z) + L_{ii} \right) - \Ztilde_i(z) \\
 & =  L_{ii} + H_{ii} - \frac{1}{N}\sum_j\Gtilde_{jj}(z) + \calAtilde_i(z) - L_{ii} - \Ztilde_i(z)  + \msc(z) - \msc(z) \\
 & = H_{ii} - \frac{1}{N}\sum_j\Gtilde_{jj}(z) + \calAtilde_i(z) - \Ztilde_i(z)  + \msc(z) - \msc(z).
\end{aligned}
\end{equation*}
where in the second equality, we have substituted the definition of $\Ztilde_i(z)$; and in the third equality, we have substituted the expansion of the conditional expectation derived in Equation \eqref{eq:conditional_expectation_resolvent:simplified}.
Rearranging terms and recalling our bracket notation defined in Equation~\eqref{eq:def:backetAverage},
we find
\begin{equation*}
\begin{aligned}
B_{ii} - (\Bi)^\top\Gtildei(z)\Bi
&= -\msc(z) +\msc(z)  -\frac{1}{N}\sum_j\Gtilde_{jj}(z) + H_{ii} - &\Ztilde_i(z) + \calAtilde_i(z) \\
&= -\msc(z) - \average{\psitilde(z)} + \Ytilde_i(z),
\end{aligned}
\end{equation*}
where we have absorbed the $\msc(z)$ term into $\average{\psitilde(z)}$ (see Equation \eqref{eqn:def:psi} for the definition of $\psitilde$) 
and we have used the definition of $\Ytilde_i(z)$ from Equation~\eqref{eq:def:Ytilde}.
Substituting this into Equation~\eqref{eq:Gtildeii:start},
\begin{equation*}
\Gtilde_{ii}(z)
    = \left(- z - \msc(z) - \average{\psitilde(z)} + \Ytilde_i(z) \right)^{-1},
\end{equation*}
as we set out to show.
%Since $\Gtildematrix(z) = (\mB - z\mI)^{-1}$, we have 
%\begin{equation*}  
%    \Gtilde_{ii}(z) = \frac{1}{-z - \msc(z) - \left(\average{\psitilde(z)} - \Ytilde_i(z) \right)} ,
%\end{equation*}
%as we set out to show.
\end{proof} 

% KDL: useful reference. https://arxiv.org/pdf/1103.1919
We now introduce two new quantities:
\begin{equation} \label{eq:def:Ld-Lo}
\Lo(z) := \max_{i\ne j} |\Gtilde_{ij}(z)| \quad \text{ and } \quad\Ld(z) := \max_{i} |\Gtilde_{ii}(z) - \msc(z)| .
\end{equation}
The error terms from Lemma~\ref{lemma:7.2}---$\calAtilde_i(z), \Ztilde_i(z), \text{ and } \Ytilde_i(z)$---are often estimated in terms of these quantities, which in turn are bounded by a deterministic $\Phi(z)$. 
% See (\cite{Erdos_general_RMT})

Next, for $\eta$ satisfying $N^{-1} r^4(\log N)^{L+12\gamma} \le \eta \le 3$, define the set $D(\eta) := \{z \in D_L : \Im z = \eta\}$ and the event 
\begin{equation} \label{eq:def:Ot}
    \Ot(\eta) := \left\{ \sup_{z\in D(\eta)} (\Ld(z) + \Lo(z))\le (\log N)^{-\xi} \right\} .
\end{equation}
We will need Lemma 7.4 from \cite{Erdos}. 
No adaptation to the present $r>1$ setting is necessary, since the proof only involves basic properties of resolvents (see the proof of Lemma 3.12 in \cite{Erdos}).
\begin{lemma}[\cite{Erdos}, Lemma 7.4]\label{lemma:trace_difference}
    Let $z = E + \iu\eta \in D_L$, 
    we have for any $i$ and $\T \subset \{1, \dots, N\}$ satisfying $i \notin \T$ and $|\T| \le \tau < N $, where $\tau$ is constant with respect to $N$, that
    \begin{equation} 
        \mtilde^{(\{i\}\cup\T)}(z) =  \mtilde^{(\T)}(z) + O\left(\frac{1}{N\eta}\right)
    \end{equation}
    holds on $\Ot(\eta)$. 
\end{lemma}

Next, we adapt Lemma 7.5 from \cite{Erdos}. 
To do so, we first introduce a few simple concentration inequalities.

\begin{lemma}[\cite{Erdos}, Lemma 3.8]\label{lemma:concetration}
    Let $a_1, \dots, a_N$ be centered and independent random variables satisfying 
    \begin{equation}\label{eqn:moment_condition}
        \E{|a_i|^p} \le \frac{C_1^p}{Nq^{p-2}}
    \end{equation}
    for $2 \le p \le (\log N)^{C_0\log\log N}$. Then there is a $\nu > 0$, depending only on $C_1$, 
    such that for all $\xi$, and for any $A_i, F_{ij} \in \C$, we have with $(\xi, \nu)$-high-probability 
    \begin{equation}
        \left| \sum_{i=1}^N A_ia_i \right| \le (\log N)^{\xi} \left[\frac{\max_i|A_i|}{q}  + \left(\frac{1}{N} \sum_{i=1}^N |A_i|^2 \right)^{1/2}\right],
    \end{equation}
    \begin{equation}
        \abs{\sum_{i=1}^N a_iF_{ii}a_i - \sum_{i=1}^N\sigma_i^2 F_{ii}} \le (\log N)^{\xi} \frac{\max_i|F_{ii}|}{q},
    \end{equation}
    and 
    \begin{equation}
        \abs{\sum_{1 \le i \ne j \le N} a_iF_{ij}a_j} \le (\log N)^{2\xi} \left(\frac{\max_{i\ne j}|F_{ij}|}{q} + \left(\frac{1}{N^2} \sum_{i\ne j} \abs{F_{ij}}^2 \right)^{1/2} \right).
    \end{equation}
    Let $a_1, \dots, a_N$ and $b_1, \dots, f_N$ be mutually independent random variables, each satisfying Equation \eqref{eqn:moment_condition}.
    Then there is a $\nu >0$ depending only on $C$ such that for all $\xi$ and $F_{ij} \in \C$, we have with $(\xi, \nu)$-high-probability
        \begin{equation}\label{lemma:concentration:quadratic_form}
        \left| \sum_{i,j=1}^N a_iF_{ij}b_j \right| \le (\log N)^{2\xi} \left[\frac{\max_i|F_{ii}|}{q^2}  +  \frac{\max_{i\ne j}|F_{ij}|}{q} + \left(\frac{1}{N^2} \sum_{i\ne j} |F_{ij}|^2 \right)^{1/2}\right]
    \end{equation}
% Below is Lemma A.1 from \cite{Erdos}:
% Let $(a_i)$ be a family of centered and independent random variables satisfying
% \begin{equation}\label{eq:erdos:A1}
%     \E{\abs{a_i}^p} \ge \frac{C_1^p}{N^{\gamma}q^{\alpha p + \beta}}
% \end{equation}
%     for all $2 \le p \le (\log N)^{A_0 \log \log N}$, where $\alpha \ge 0$ and $\beta, \gamma \in \R$. Then for all even $p$ satisfying $2 \le p \le (\log N)^{A_0 \log \log N}$, we have
%     \begin{equation}
%         \E{\abs{\sum_{i} A_i a_i}^p} \le (C p)^p \left[ \frac{\sup_i \abs{A_i}}{q^{\alpha}} + \left( \frac{1}{N^{\gamma}q^{\beta + 2\alpha}} \sum_{i} \abs{A_i}^2 \right)^{1/2}\right]^p
%     \end{equation}
%     where $C >0$ is a constant that depends on $C_1$.

%     Below is A.4 from \cite{Erdos}:
%     Let $a_1,\dots, a_N$ be centered and independent random variables satisfying
%     \begin{equation}
%     \E{\abs{a_i}^p} \ge \frac{C_1^p}{Nq^{p - 2}}
%     \end{equation}
%      for all $2 \le p \le (\log N)^{A_0 \log \log N}$. Then for all even $p$ satisfying $2 \le p \le (\log N)^{A_0 \log \log N}$, we have
%      \begin{equation}
%          \E{ \abs{\sum_{i\ne j} \overline{a}_i B_{ij} a_j}^p} 
%          \le (C_2p)^{2p}\left[ \frac{\max_{i\ne j} \abs{B_{ij}}}{q} + \left( \frac{1}{N^{2}} \sum_{i\ne j} \abs{B_{ij}}^2 \right)^{1/2}\right]^p
%      \end{equation}
\end{lemma}

\begin{lemma}[\cite{Erdos}, Lemma 7.5, adapted]\label{lemma:hb_bound}
    Let $\tau < N$ be constant with respect to $N$, and let $\mB$ be as in Equation \eqref{eq:def:B}. 
    Let $z = E + \iu\eta \in D_L$, we have for any $\T \subset [N]$ satisfying $|\T| \le \tau < N$, as long as $i \in \T$, that for $j \in \{N-r+1, N-r+2, \dots, N \}$, 
    \begin{equation}\label{ineq:hb_bound}
        \left|\lambda_j(\VMt{j})^\top\Gtildetmatrix(z)\Ht \right| 
        \le Cr(\log N)^{\xi+ 3\gamma}\left(\frac{1}{q} + \sqrt{\frac{\Im \mtilde(z)}{N\eta}} + \frac{1}{N\eta}\right)
    \end{equation}
    with $\xinu$-high probability on the event $\Ot(\eta)$, where $q$ satisfies Equation \eqref{eq:def:q}.
\end{lemma}
\begin{proof}
We note that our proof of this bound deviates slightly from \cite{Erdos}.
Our proof is simplified because we assume that the signal eigenvalues of $\mL$ obey $\lambda_i \asymp \rhon\sqrt{N}$ by Assumption \eqref{assumption:eigen_bound}. 

%We will need the following to allow for sparsity.
First, with $\calS = \{N-r + 1 - |\T|, N-r+2 - |\T|, \dots, N - |\T| \}$, define
\begin{equation*}
    R := \max_{|\T| \le \tau}\max_{\alpha\notin \mathcal{S}}\max_j|\vhatt_{j\alpha}| .
\end{equation*}

Recall the event $\Ot(\eta)$ from Equation~\eqref{eq:def:Ot}.
On $\Ot(\eta)$, we have by Lemma \ref{lemma:bounded_resolvents},
\begin{equation*}
c \le |\Gtildet_{ii}(z) | \le C ,
\end{equation*}
so that for $z := \lambdahatt_{\alpha} + \iu \eta \in D_L$ and $\alpha \notin \calS$,
\begin{equation*}\begin{aligned}
    C &\ge \Im \Gtildet_{ii}(\lambdahatt_{\alpha} + \iu\eta) 
    = \sum_{\beta \in [N - \abs{\T}]}\frac{\eta|\vhatt_{i\beta}|^2}{(\lambdahatt_{\beta}- \lambdahatt_{\alpha})^2 + \eta^2} 
    =  \frac{\eta|\vhatt_{i\alpha}|^2}{\eta^2} + \sum_{\beta \ne \alpha}\frac{\eta|\vhatt_{i\beta}|^2}{(\lambdahatt_{\beta}- \lambdahatt_{\alpha})^2 + \eta^2} \\
    &\ge \frac{|\vhatt_{i\alpha}|^2}{\eta}.
\end{aligned}\end{equation*}
Taking square roots and recalling that our choices of $\alpha \in \calS$, $\T \subset [N]$, and $j \in [N]$ were arbitrary, we obtain
\begin{equation} \label{eq:R:UB}
R \le C\sqrt{\eta}.
\end{equation}

Now, notice that the spectral decomposition of $\Gtildet(z)$ is given by
\begin{equation*}
    \Gtildetmatrix(z)= \sum_{\alpha =1}^{N-|\T|} \frac{\bvhatt_{\alpha}(\bvhatt_{\alpha})^\top}{\lambdahatt_{\alpha} - z}.
\end{equation*}
Hence, applying the triangle inequality, 
\begin{equation}\label{eq:something2}
\begin{aligned}
    \abs{\lambda_j(\VMt{j})^\top\Gtildetmatrix(z)\Ht } &= \abs{\lambda_j(\VMt{j})^\top \left(\sum_{\alpha}\frac{\bvhatt_{\alpha}(\bvhatt_{\alpha})^\top}{\lambdahatt_{\alpha} - z}\right)\Ht } \\
    & = \abs{\sum_{\alpha}\frac{\lambda_j(\VMt{j})^\top\bvhatt_{\alpha}(\bvhatt_{\alpha})^\top\Ht}{\lambdahatt_{\alpha} - z} } \\
    &\le  \left| \frac{\lambda_j(\VMt{j})^\top\bvhatt_{j - |\T|}(\bvhatt_{j - |\T|})^\top\Ht}{\lambdahatt_{j - \abs{\T}} - z} \right| \\
    & ~~~~~~~~~~~~~~~+ \left| \sum_{\alpha \ne j - |\T|}\frac{\lambda_j(\VMt{j})^\top\bvhatt_{\alpha}(\bvhatt_{\alpha})^\top\Ht}{\lambdahatt_{\alpha} - z} \right| .
\end{aligned}
\end{equation}
Note that by Assumption~\eqref{assumption:deloc} and the fact that the eigenvectors are normalized, we have
\begin{equation} \label{eq:innerprod:bound}
\abs{(\VMt{j})^\top\bvhatt_{j - |\T|} }
= |v_{ij}|\abs{(\submatrix{\bv_j}{\T})^\top\bvhatt_{j - |\T|}}
\le \frac{C\deloc}{\sqrt{N}}\abs{(\submatrix{\bv_j}{\T})^\top\bvhatt_{j - |\T|} }
\le \frac{C\deloc}{\sqrt{N}} .
\end{equation}
Applying submultiplicativity to the first term of the right-hand side of the inequality in Equation \eqref{eq:something2} and using Equation~\eqref{eq:innerprod:bound}, we see that 
\begin{equation} \label{eq:term1:submult}
     \abs{ \frac{\lambda_j(\VMt{j})^\top\bvhatt_{j - |\T|}(\bvhatt_{j - |\T|})^\top\Ht}{\lambdahatt_{j - \abs{\T}} - z} }
     \le \frac{C\deloc}{\sqrt{N}}\abs{ \frac{\lambda_j}{\lambdahatt_{j - \abs{\T}} - z}}\abs{ (\bvhatt_{j - |\T|})^\top\Ht } .
\end{equation}
Applying Lemma \ref{lemma:concetration} to control the inner product, we have with $(\xi, \nu)$-high-probability,
\begin{equation*}
\abs{ \frac{\lambda_j(\VMt{j})^\top\bvhatt_{j - |\T|}(\bvhatt_{j - |\T|})^\top\Ht}{\lambdahatt_{j - \abs{\T}} - z} }
     \le
     \frac{\deloc}{\sqrt{N}}\left| \frac{\lambda_j}{\lambdahatt_{j - \abs{\T}} - z}\right|(C\log N)^{\xi} \left[\frac{1}{q}  + \frac{1}{\sqrt{N}}\right] .
\end{equation*}

%\begin{equation*}
%    \left| \frac{f_j}{\mu_{j - |\T|}^{(\T)} - z}\right|\left| %(\bv^{(\T)}_{j - |\T|})^\top\Ht \right| 
%    \le 
%    \frac{C}{\sqrt{N}}\left| \frac{f_j}{\mu_{j - |\T|}^{(\T)} - z}\right|(\log N)^{\xi} \left[\frac{1}{q}  + \frac{1}{\sqrt{N}}\right]
%\end{equation*}
When $\lambda_j \ge \Sigma + 3$, which occurs due to Assumptions \eqref{assumption:signal_strength} and \eqref{assumption:eigen_bound},
we have $|\lambdahat^{(\T)}_{j - |\T|} - z| \ge c\lambda_j$, so with $(\xi, \nu)$-high-probability
%\begin{equation*}
%    \frac{C}{\sqrt{N}}\left| \frac{f_j}{f_j}\right|(\log N)^{\xi} \left[\frac{1}{q}  + \frac{1}{\sqrt{N}}\right]
%    \le    \frac{C(\log N)^{\xi}}{\sqrt{N}q}
%\end{equation*}
\begin{equation} \label{eq:term1:finalbound}
     \abs{ \frac{\lambda_j(\VMt{j})^\top\bvhatt_{j - |\T|}(\bvhatt_{j - |\T|})^\top\Ht}{\lambdahatt_{j - \abs{\T}} - z} } \le \frac{C\deloc(\log N)^{\xi}}{\sqrt{N}q} .
\end{equation}
The second term in Equation \eqref{eq:something2}, consisting of eigenvalues associated with the bulk, is approached similarly.

Applying the triangle inequality,
\begin{equation*}
 \left| \sum_{\alpha \ne j - |\T|}\frac{\lambda_j(\VMt{j})^\top\bvhatt_{\alpha}(\bvhatt_{\alpha})^\top\Ht}{\lambdahatt_{\alpha} - z} \right| 
     \le \lambda_j\sum_{\alpha \ne j - |\T|} \left|\frac{(\VMt{j})^\top\bvhatt_{\alpha}}{\lambdahatt_{\alpha} - z} \right| |(\bvhatt_{\alpha})^\top\Ht|.
\end{equation*}
By Lemma \ref{lemma:concetration}, we have with $(\xi, \nu)$-high-probability,
\begin{equation} \label{eq:term2:start}
     \left| \sum_{\alpha \ne j - |\T|} \! \frac{\lambda_j(\VMt{j})^\top\bvhatt_{\alpha}(\bvhatt_{\alpha})^\top\Ht}{\lambdahatt_{\alpha} - z} \right| 
\le 
     \lambda_j\!\sum_{\alpha \ne j - |\T|} \left|\frac{(\VMt{j})^\top\bvhatt_{\alpha}}{\lambdahatt_{\alpha} - z} \right|(\log N)^{\xi} \left[\frac{C\sqrt{\eta}}{q} \! + \! \frac{1}{\sqrt{N}}\right] .
\end{equation}
By the Cauchy-Schwarz inequality,
\begin{equation} \label{eq:term2:sum2bound}
    \sum_{\alpha \ne j - |\T|} \left|\frac{(\VMt{j})^\top\bvhatt_{\alpha}}{\lambdahatt_{\alpha} - z} \right| \le \left(\sum_{\alpha \ne j - |\T|} |(\VMt{j})^\top\bvhatt_{\alpha}|^2\right)^{1/2} \left(\sum_{\alpha} \frac{1}{|\lambdahatt_{\alpha} - z|^2}\right)^{1/2} .
\end{equation}
By Lemma \ref{lemma:trace_difference}, on the event $\Ot(\eta)$,
\begin{equation} \label{eq:muinvsum}
    \sum_{\alpha} \frac{1}{|\lambdahatt_{\alpha} - z|^2} = \frac{1}{\eta}\Im \sum_{\alpha} \frac{1}{\lambdahatt_{\alpha} - z} =  \frac{1}{\eta}\Im \Tr \Gtildetmatrix(z) = \frac{N}{\eta} \Im \mtilde(z) + O\left( \frac{1}{\eta^2} \right) .
\end{equation}
Moreover,
\begin{equation*} % \label{eq:sumoveralpha}
\sum_{\alpha \ne j - |\T|} |(\VMt{j})^\top\bvhatt_{\alpha}|^2 \le |v_{ij}|^2\sum_{\alpha \ne j - |\T|} |(\submatrix{\bv_j}{\T})^\top\bvhatt_{\alpha}|^2 
\le \frac{C\delocsq}{N} \left( \frac{r\logN^{ 2\gamma}}{\lambda_j} \right)^2
\end{equation*}
with $\xinu$-high probability, where the last inequality follows from Lemma \ref{lemma:subspace_bound} and Assumption \eqref{assumption:deloc} on the entries of $\bv_j$. 
Taking square rools, we have with $\xinu$-high probability,
\begin{equation} \label{eq:sumoveralpha}
    \left(\sum_{\alpha \ne j - |\T|} |(\VMt{j})^\top\bvhatt_{\alpha}|^2\right)^{1/2} \le \frac{C\deloc}{\sqrt{N}} \left( \frac{r\logN^{2\gamma}}{\lambda_j} \right) .
\end{equation}
Applying the bounds in Equations~\eqref{eq:muinvsum} and ~\eqref{eq:sumoveralpha} to Equation~\eqref{eq:term2:sum2bound},
\begin{equation*}
    \sum_{\alpha \ne j - |\T|} \left|\frac{(\VMt{j})^\top\bvhatt_{\alpha}}{\lambdahatt_{\alpha} - z} \right| 
    \le \frac{C\deloc}{\sqrt{N}} \left( \frac{r\logN^{2\gamma}}{\lambda_j} \right) \sqrt{  \frac{N}{\eta} \Im \mtilde(z) + O\left( \frac{1}{\eta^2} \right) }
\end{equation*}
with $\xinu$-high probability on the event $\Ot(\eta)$. Applying this bound to Equation~\eqref{eq:term2:start},
\begin{equation}\label{eq:something1} \begin{aligned}
     &\left| \sum_{\alpha \ne j - |\T|}\frac{\lambda_j(\VMt{j})^\top\bvhatt_{\alpha}(\bvhatt_{\alpha})^\top\Ht}{\lambdahatt_{\alpha} - z} \right| \\
     &~~~~~~\le C\left( \frac{r\logN^{\xi + 3\gamma}}{\sqrt{N}} \right) \left[ \frac{ C\sqrt{\eta}}{q} + \frac{1}{\sqrt{N}} \right]
    \sqrt{  \frac{N}{\eta} \Im \mtilde(z) + O\left( \frac{1}{\eta^2} \right) } 
\end{aligned} \end{equation}
with $\xinu$-high probability on the event $\Ot(\eta)$. 
Altogether, Equations \eqref{eq:term1:finalbound} and \eqref{eq:something1} imply
\begin{equation*} %\label{eq:hb_control_last}
\begin{aligned}
    &\left| \frac{\lambda_j(\VMt{j})^\top\bvhatt_{j - |\T|}(\bvhatt_{j - |\T|})^\top\Ht}{\lambdahatt_{j - \abs{\T}} - z} \right| + \left| \sum_{\alpha \ne j - |\T|}\frac{\lambda_j(\VMt{j})^\top\bvhatt_{\alpha}(\bvhatt_{\alpha})^\top\Ht}{\lambdahatt_{\alpha} - z} \right| \\
    &~~~\le \frac{C\logN^{\xi + \gamma}}{\sqrt{N}q} +  C\left( \frac{r\logN^{\xi + 3\gamma}}{\sqrt{N}} \right) \left[ \frac{ C\sqrt{\eta}}{q} + \frac{1}{\sqrt{N}} \right] \left( \sqrt{  \frac{N}{\eta} \Im \mtilde(z) } + \frac{1}{\eta} \right) \\
    &~~~\le Cr(\log N)^{\xi+ 3\gamma}\left(\frac{1}{q} + \sqrt{\frac{\Im \mtilde(z)}{N\eta}} + \frac{1}{N\eta}\right)
\end{aligned}
\end{equation*}
with $\xinu$-high probability on the event $\Ot(\eta)$. 
%Hence, the right-hand side of the last inequality in Equation \eqref{eq:hb_control_last} bounds 
Applying this to Equation \eqref{eq:something2},
    \begin{equation*}
        \left|\lambda_j(\VMt{j})^\top\Gtildetmatrix(z)\Ht \right| 
        \le Cr(\log N)^{\xi + 3\gamma}\left(\frac{1}{q} + \sqrt{\frac{\Im \mtilde(z)}{N\eta}} + \frac{1}{N\eta} \right)
    \end{equation*}
    with $\xinu$-high probability on the event $\Ot(\eta)$, completing the proof.
\end{proof} %KDL checked Jan 6, 2026, 00:30

\begin{lemma}\label{lemma:main_quad}
Let $\tau < N$ be constant with respect to $N$ and let $z = E + \iu\eta \in D_L$.
    We have, uniformly over $j \in \{N-r + 1, \dots, N\}$ and $\T \subset \{1, \dots, N\}$ satisfying $|\T| \le \tau$, 
    that with $\xinu$-high probability,
    \begin{equation}\label{ineq:main_quad}
        \left|\lambda_j^2(\VMt{j})^\top\Gtildetmatrix(z) \VMt{j}\right| \le \frac{C\lambda_j\delocsq}{N} + C \frac{r^2\logN^{6\gamma}}{N\eta} .
    \end{equation}
\end{lemma}
\begin{proof}
Applying the triangle inequality to the spectral decomposition of $\Gtildetmatrix(z)$,
%\begin{equation*}
%    \Gtildetmatrix(z)= \sum_{\alpha =1}^{N-|\T|} \frac{\bvhatt_{\alpha}(\bvhatt_{\alpha})^\top}{\lambdahatt_{\alpha} - z},
%\end{equation*}
\begin{equation}\label{eq:something3}
\begin{aligned}
\left|\lambda_j^2(\VMt{j})^\top \left(\sum_{\alpha}\frac{\bvhatt_{\alpha}(\bvhatt_{\alpha})^\top}{\lambdahatt_{\alpha} - z} \right) \VMt{j}\right| 
&\le  \left| \frac{\lambda_j^2(\VMt{j})^\top\bvhatt_{j - |\T|}(\bvhatt_{j - |\T|})^\top\VMt{j}}{\lambdahatt_{j - \abs{\T}} - z} \right| \\
&~~~+ \left| \sum_{\alpha \ne j - |\T|}\frac{\lambda_j^2(\VMt{j})^\top\bvhatt_{\alpha}(\bvhatt_{\alpha})^\top\VMt{j}}{\lambdahatt_{\alpha} - z} \right| .
\end{aligned}
\end{equation}
Using similar arguments as in the proof of Lemma \ref{lemma:hb_bound} (e.g., as in Equation \eqref{eq:term1:submult}), the first term of Equation \eqref{eq:something3} is controlled as 
\begin{equation}\label{eq:something4}
    \abs{ \frac{\lambda_j^2(\VMt{j})^\top\bvhatt_{j - |\T|}(\bvhatt_{j - |\T|})^\top\VMt{j}}{\lambdahatt_{j - \abs{\T}} - z} } \le \frac{C\lambda_j^2}{\lambda_j}\frac{\delocsq}{N} = \frac{C\lambda_j\delocsq}{N}
\end{equation}
with $\xinu$-high probability.
Turning our attention to the second term on the right-hand side of Equation \eqref{eq:something3}, noticing that $|\lambdahatt_{j - \abs{\T}} - z| \ge \eta$ since $\lambdahatt_{j - \abs{\T}}$ is real, 
\begin{equation} \label{eq:something5}
\begin{aligned}
    \left| \sum_{\alpha \ne j - |\T|}\frac{\lambda_j^2(\VMt{j})^\top\bvhatt_{\alpha}(\bvhatt_{\alpha})^\top\VMt{j}}{\lambdahatt_{\alpha} - z} \right|
    &\le \left(\frac{C\lambda_j^2\delocsq}{N\eta} \right)&\sum_{\alpha \ne j - |\T|} 
            \left| (\submatrix{\bv_j}{\T})^\top\bvhatt_{j - |\T|} \right|^2 \\
    &\le C \frac{r^2\logN^{6\gamma}}{N\eta},
\end{aligned} \end{equation}
where the last inequality holds with $\xinu$-high probability by Lemma \ref{lemma:subspace_bound}.
Applying Equations \eqref{eq:something4} and \eqref{eq:something5} to the right-hand side of Equation \eqref{eq:something3}, we obtain
\begin{equation*}
\left|\lambda_j^2(\VMt{j})^\top\Gtildetmatrix(z) \VMt{j}\right| \le \frac{C\lambda_j\delocsq}{N} + C \frac{r^2\logN^{6\gamma}}{N\eta}
\end{equation*}
with $\xinu$-high probability.
The choice of $j \in [N]$ was arbitrary, so this bound holds for all $j \in [N]$ uniformly with $\xinu$-high probability, as we set out to show.
\end{proof} 

The following lemma has no analogue in \cite{Erdos}, as it is required only when we consider a rank $r>1$ signal matrix.

\begin{lemma} \label{lemma:cross_term}
    Let $\mB$ satisfy Equation \eqref{eq:def:B}, and let $\tau < N$ be constant with respect to $N$
    Let $z = E + \iu\eta \in D_L$, 
    we have uniformly for any $\T \subset \{1, \dots, N\}$ satisfying $|\T| \le \tau < N$, $\ell_1, \ell_2 \in \{N - r +1, \dots, N\}$ and $\ell_1 \ne \ell_2$, and for possibly $i = j \in [N] \sm \T$, that with $\xinu$-high probability,
    \begin{equation}\label{ineq:cross_term}
        \left|\lambda_{\ell_2}\lambda_{\ell_1}(\VMt{\ell_1})^\top\Gtildetmatrix(z) \VMtj{\ell_2}\right| \le C r^2\frac{ \logN^{6\gamma}}{N \eta} .
    \end{equation}
\end{lemma}
\begin{proof}
    We follow an argument similar to that used in Lemmas~\ref{lemma:hb_bound} and~\ref{lemma:main_quad}. 
    Let $\T \subset [N]$ and $\ell_1, \ell_2 \in  \{N - r +1, \dots, N\}$ distinct.
    By the spectral decomposition of $\Gtildetmatrix(z)$ and Lemma \ref{lemma:pop_eigenvalues_bound}, we have
    \begin{equation}\label{eq:something6}
    \begin{aligned}
         \abs{\lambda_{\ell_2}\lambda_{\ell_1}(\VMt{\ell_1})^\top\Gtildetmatrix(z) \VMtj{\ell_2}} 
         &\le  C\lambda_N^2\abs{(\VMt{\ell_1})^\top \left(\sum_{\alpha}\frac{\bvhatt_{\alpha}(\bvhatt_{\alpha})^\top}{\lambdahatt_{\alpha} - z} \right) \VMtj{\ell_2}} \\
         &\le C\lambda_N^2 \abs{v_{i\ell_1}}\abs{v_{j\ell_2}} \left|(\vt{\ell_1})^\top \!\left(\!\sum_{\alpha}\frac{\bvhatt_{\alpha}(\bvhatt_{\alpha})^\top}{\lambdahatt_{\alpha} - z} \! \right) \vt{\ell_2}\right|.
    \end{aligned} \end{equation}
    with $\xinu$-high probability. 
    %By standard properties of absolute value,
    %\begin{equation}\label{eq:something7}
    %     \left|\lambda_{\ell_2}\lambda_{\ell_1}(\VMt{\ell_1})^\top\Gtildetmatrix(z) \VMtj{\ell_2}\right| \le C\lambda_N^2 \abs{v_{i\ell_1}}\abs{v_{j\ell_2}} \left|(\vt{\ell_1})^\top \left(\sum_{\alpha}\frac{\bvhatt_{\alpha}(\bvhatt_{\alpha})^\top}{\lambdahatt_{\alpha} - z} \right) \vt{\ell_2}\right|.
    %\end{equation}
    Recalling our assumption on the entries of $\bv_k$, for $k \in \{N-r+1, \dots, N\}$ from Assumption \eqref{assumption:deloc}, we have $N\abs{v_{i\ell_1}}\abs{v_{j\ell_2}} \le C \delocsq$. Applying this inequality to Equation \eqref{eq:something6} yields
    \begin{equation*} %\label{eq:something8}
         \left|\lambda_{\ell_2}\lambda_{\ell_1}(\VMt{\ell_1})^\top\Gtildetmatrix(z) \VMtj{\ell_2}\right| 
         \le \frac{C\lambda_N^2\delocsq}{N} \left|(\vt{\ell_1})^\top \left(\sum_{\alpha}\frac{\bvhatt_{\alpha}(\bvhatt_{\alpha})^\top}{\lambdahatt_{\alpha} - z} \right) \vt{\ell_2}\right|
    \end{equation*}
    with $\xinu$-high probability.
    Distributing and applying the triangle inequality, % to the right-hand side of Equation \eqref{eq:something8},
    \begin{equation} \label{eq:something9}
    \begin{aligned}
        &\left|\lambda_{\ell_2}\lambda_{\ell_1}(\VMt{\ell_1})^\top\Gtildetmatrix(z) \VMtj{\ell_2}\right| \\
        &~~~\le \frac{C\lambda_N^2\delocsq}{N} \Bigg( \abs{\frac{(\vt{\ell_1})^\top\bvhatt_{\ell_1 - |\T|}(\bvhatt_{\ell_1 - |\T|})^\top\vt{\ell_2}}{\lambdahat_{\ell_1 - |\T|}^{(\T)} - z}} 
          + \abs{ \frac{(\vt{\ell_1})^\top\bvhatt_{\ell_2 - |\T|}(\bvhatt_{\ell_2 - |\T|})^\top\vt{\ell_2}}{\lambdahat_{\ell_2 - |\T|}^{(\T)} - z}} \\ 
           & \quad \quad \quad \quad \quad \quad \quad \quad \quad \quad \quad \quad 
           + \sum_{\alpha \notin \mathcal{S}} \abs{ \frac{(\vt{\ell_1})^\top\bvhatt_{\alpha}(\bvhatt_{\alpha})^\top\vt{\ell_2}}{\lambdahatt_{\alpha} - z}} \Bigg)
     \end{aligned}
    \end{equation}
    where $\mathcal{S} = \{\ell_1 - |\T|, \ell_2 - |\T| \}$. 
    
    The first two terms of Equation \eqref{eq:something9} are controlled similarly. We will examine the first term as an example. 
    Using the fact that both $\bv_{\ell_1}$ and $\bvhatt_{\ell_1 - |\T|}$ are normalized eigenvectors, we have $\left|(\vt{\ell_1})^\top\bvhatt_{\ell_1 - |\T|} \right| \le 1$. 
    Moreover, $|\lambdahat^{(\T)}_{j - |\T|} - z| \ge c\lambda_j$ by Assumption \eqref{assumption:eigen_bound}.
    These two bounds together imply that with $\xinu$-high probability,
    \begin{equation}\label{eq:something10}
        \abs{ \frac{(\vt{\ell_1})^\top\bvhatt_{\ell_1 - |\T|}(\bvhatt_{\ell_1 - |\T|})^\top\vt{\ell_2}}{\lambdahat_{\ell_1 - |\T|}^{(\T)} - z}} \le \frac{C}{\lambda_{\ell_1}} \left|(\bvhatt_{\ell_1 - |\T|})^\top\vt{\ell_2} \right|
        \le \frac{C}{\lambda_{\ell_1}}\frac{r\logN^{2\gamma}}{\lambda_{\ell_2}} ,
    \end{equation}
    where the second inequality follows from Lemma \ref{lemma:subspace_bound} and Assumption \eqref{assumption:deloc}. %Note: this is where we are using \ell_1 \neq \ell_2-- Lemma \ref{lemma:subspace_bound} needs vhat and v to correspond to orthogonal population eigenspaces
    Using the fact that $\lambda_{\ell_1} \ge \lambda_{N-r+1}$ and substituting Equation \eqref{eq:something10} into \eqref{eq:something9} yields that with $\xinu$-high probability,
\begin{equation}\label{eq:something11}
\begin{aligned}
        &\left|\lambda_{\ell_2}\lambda_{\ell_1}(\VMt{\ell_1})^\top\Gtildetmatrix(z) \VMtj{\ell_2}\right| \\
        &~~~~~~\le \frac{C\lambda_N^2\delocsq}{N} \Bigg(\frac{Cr\logN^{2\gamma}}{\lambda_{N-r+1}^2} +
            \sum_{\alpha \notin \mathcal{S}} \left| \frac{(\vt{\ell_1})^\top\bvhatt_{\alpha}(\bvhatt_{\alpha})^\top\vt{\ell_2}}{\lambdahatt_{\alpha} - z}\right| \Bigg) .
\end{aligned} \end{equation}
Since $\abs{\lambdahatt_{\alpha} - z} \ge \eta$, %the last term in Equation \eqref{eq:something11} is controlled as
the triangle inequality yields
    \begin{equation*} %\label{eq:something12}
    \begin{aligned}
        \sum_{\alpha \notin \mathcal{S}} \left| \frac{(\vt{\ell_1})^\top\bvhatt_{\alpha}(\bvhatt_{\alpha})^\top\vt{\ell_2}}{\lambdahatt_{\alpha} - z}\right| 
        &\le \frac{1}{\eta} \sum_{\alpha \notin \mathcal{S}} |(\vt{\ell_1})^\top\bvhatt_{\alpha}||(\bvhatt_{\alpha})^\top\vt{\ell_2}| \\
        &\le \frac{1}{\eta}\left(\sum_{\alpha \notin \mathcal{S}} |(\vt{\ell_1})^\top\bvhatt_{\alpha}|^2 \right)^{1/2}\left(\sum_{\alpha \notin \mathcal{S}} |(\bvhatt_{\alpha})^\top\vt{\ell_2}|^2 \right)^{1/2} \\
        & \le \frac{C}{\eta} \left(\frac{r\logN^{2\gamma}}{\lambda_{N-r+1}}\right)^2,
    \end{aligned}
    \end{equation*}
    where the second inequality follows from the Cauchy-Schwarz inequality
    and the last inequality holds with $\xinu$-high probability by Lemma \ref{lemma:subspace_bound}. 
    Applying this to Equation \eqref{eq:something11} and using Assumption~\ref{assumption:eigen_bound}, we obtain 
    \begin{equation*}
        \left|\lambda_{\ell_2}\lambda_{\ell_1}(\VMt{\ell_1})^\top\Gtildetmatrix(z) \VMtj{\ell_2}\right| 
        \le  C r^2\frac{ \logN^{6\gamma}}{N \eta},
    \end{equation*}
     with $\xinu$-high probability.
     The choice of $\ell_1, \ell_2 \in \{N - r +1, \dots, N\}$ and $\T \subset [N]$ were arbitrary, so this bound holds for any choice of $\ell_1, \ell_2 \in \{N - r +1, \dots, N\}$ and $\T$ such that $|\T| \le \tau$ with $\xinu$-high probability, completing the proof.
\end{proof} 

Following \cite{Erdos}, we will introduce a control parameter,
\begin{equation}\label{eq:def:Phi}
    \Phi(z) := \frac{r^2(\log N)^{\xi + 3\gamma}}{q} + r^2(\log N)^{2\xi + 6\gamma} \left(\sqrt{\frac{\Im \mtilde(z)}{N\eta} } + \frac{r^2}{N\eta}\right)
\end{equation}
where $q$ satisfies Equation \eqref{eq:def:q}.

\begin{lemma}[\cite{Erdos} Proposition 7.6, adapted] \label{lemma:Lo_Zi_Ai_estimates}
Let $\mB$ satisfy Equation \eqref{eq:def:B}. For $z = E + \iu\eta \in D_L$, we have that on on $\Ot(\eta)$ (see Equation \eqref{eq:def:Ot}) with $(\xi, \nu)$-high probability,
    \begin{equation} \label{eq:Lo:bound}
        \Lo(z) \le C\Phi(z),
    \end{equation} 
    \begin{equation} \label{eq:Ztilde_i}
        \max_{i \in [N]} |\Ztilde_i(z)| \le C\Phi(z)
    \end{equation}
    and
    \begin{equation} \label{eq:calAtilde_i}
        \max_{i \in [N]} |\calAtilde_i(z)| \le  \frac{Cr\logN^{2\gamma}}{\sqrt{N}} + \frac{Cr^4\logN^{6\gamma}}{N\eta} ,
    \end{equation}
     where $\Ztilde$ and $\calAtilde$ are as in Equations \eqref{eq:def:Ztilde} and \eqref{eq:def:calAtilde}, respectively,
     and $\Lo(z)$ is as in Equation \eqref{eq:def:Ld-Lo}.
\end{lemma}
\begin{proof}
    Starting with the estimate for $\Lo(z)$, we see that from Lemma \ref{lemma:resolvent_identity_corollary}, we have
    \begin{equation*}
        \Gtilde_{ij}(z) = \Gtilde_{ii}(z)\Gtildei_{jj}(z)\left((\submatrix{\mB_{\cdot i}}{ij})^\top\Gtildeijmatrix(z) \submatrix{\mB_{\cdot j}}{ij} \right) - \Gtilde_{ii}(z)\Gtildei_{jj}(z)B_{ij}.
    \end{equation*}
    By Lemma \ref{lemma:bounded_resolvents}, we have $\abs{\Gtilde_{ii}(z)} \le C$ and $|\Gtildei_{jj}(z)| \le C$ for all $i,j \in [N]$.
    Therefore,
    \begin{equation} \label{eq:Gtildeij:checkpt}
        \abs{\Gtilde_{ij}(z)} \le C\abs{(\submatrix{\mB_{\cdot i}}{ij})^\top\Gtildeijmatrix(z) \submatrix{\mB_{\cdot j}}{ij}} + C\abs{B_{ij}} = C\abs{ \sum_{k,l}^{(\sm ij)}B_{ik} \Gtildeij_{kl}(z) B_{lj} }+ C|B_{ij}|.
    \end{equation}
    By Lemma \ref{lemma:B_control}, we have that with $(\xi, \nu)$-high probability, it holds uniformly over for all $i,j \in [N]$ that 
    \begin{equation*} %\label{eq:Gtildeij:boundA} 
    \abs{B_{ij}} \le  C \frac{r\rhon\delocsq}{\sqrt{N}} + \frac{C}{q} . 
    \end{equation*} 
    Substituting this into Equation \eqref{eq:Gtildeij:checkpt}, we have 
    \begin{equation} \label{eq:initial_estimates:Gtildeij:checkpt2}
        |\Gtilde_{ij}(z)| \le C\left| \sum_{k,l}^{(\sm ij)}B_{ik} \Gtildeij_{kl}(z) B_{lj} \right| + C\frac{r\rhon\delocsq}{\sqrt{N}} + \frac{C}{q} .
    \end{equation}
    %To bound $C\abs{(\submatrix{\mB_{\cdot i}}{ij})^\top\Gtildeijmatrix(z) \submatrix{\mB_{\cdot j}}{ij}}$, observe that from
    Recalling the definition of $\mB$ from Equation~\eqref{eq:def:B},
    \begin{equation}\label{ineq:Lo_expansion_2}
       \left| \sum_{k,l}^{(\sm ij)}B_{ik} \Gtildeij_{kl}(z)B_{lj} \right| 
       =  \left| \sum_{k,l}^{(\sm ij)} \left(h_{ik} + \sum_{\alpha = N - r + 1}^N \lambda_{\alpha}v_{i\alpha}v_{k\alpha} \right) \Gtildeij_{kl}(z) \left(h_{lj} + \sum_{\alpha = N - r + 1}^N \lambda_{\alpha}v_{l\alpha}v_{j\alpha} \right)\right|.
    \end{equation}
    %     \\
    %     &\le
    %     \left|\sum_{k,l}^{(ij)} h_{ik}\Gtildeij_{kl}h_{lj} \right| +  \left| f_N (\EMij{N}{i})^\top\Gtildeij(z)\submatrix{\mH_{\cdot j}}{ij}  \right| + \left| f_{N-1} (\EMij{N-1}{i})^\top\Gtildeij(z)\submatrix{\mH_{\cdot j}}{ij}  \right| \\
    %     &+ \left|\submatrix{\mH_{\cdot i}}{ij} \Gtildeij(z) f_N \EMij{N}{j}\right| + \left|\submatrix{\mH_{\cdot i}}{ij} \Gtildeij(z) f_{N-1} \EMij{N-1}{j}\right| \\
    %     &+ \left| f_N (\EMij{N}{i})^\top\Gtildeij(z) f_N \EMij{N}{j} \right|  \\
    %     &+ \left| f_{N-1} (\EMij{N-1}{i})^\top\Gtildeij(z) f_{N-1} \EMij{N-1}{j} \right| \\
    %     &+ \left| f_{N-1} (\EMij{N-1}{i})^\top\Gtildeij(z) f_{N} \EMij{N}{j} \right| \\
    %     &+ \left| f_N (\EMij{N}{i})^\top\Gtildeij(z) f_{N-1} \EMij{N-1}{j} \right| .
    % \end{aligned}
    % \end{equation}
    
    After applying triangle inequality to the right-hand side of Equation \eqref{ineq:Lo_expansion_2}, there will be a term controllable by Lemma \ref{lemma:HGH_control} according to  
    \begin{equation}\label{eq:initial_estimates:HGH_bound}
    \abs{\sum_{k,l}^{(\sm ij)} h_{ik}\Gtildeij_{kl}(z) h_{lj} } \le \frac{C(\log N)^{\xi}}{q} + C(\log N)^{2\xi}\left(\sqrt{\frac{\Im \mtilde(z)}{N\eta}}  + \frac{1}{N\eta} \right).
    \end{equation}
    on the event $\Ot(\eta)$ with $(\xi, \nu)$-high probability.

    After applying triangle inequality in Equation \eqref{ineq:Lo_expansion_2}, there will also be $2r$ terms that can be controlled by Lemma~\ref{lemma:hb_bound} as, for $\ell = i,j$,
    \begin{equation}\label{eqn:apx:main_est:hb-terms}
       \abs{\lambda_{\alpha_1}(\VMij{\alpha_1}{i})^\top\Gtildeijmatrix(z)\submatrix{\mH_{\cdot \ell}}{ij} } 
       \le  Cr(\log N)^{\xi + 3\gamma}\left(\frac{1}{q} + \sqrt{\frac{\Im \mtilde(z)}{N\eta}} + \frac{1}{N\eta} \right) 
    \end{equation}
    with $(\xi, \nu)$-high probability.
    By comparing the quantity on the right-hand side of Equation \eqref{eqn:apx:main_est:hb-terms} to the definition of $\Phi(z)$ in Equation \eqref{eq:def:Phi}, we see that on $\Ot(\eta)$ with $(\xi, \nu)$-high probability,
    \begin{equation} \label{eq:initial_estimates:hb_bounds}
    \abs{\lambda_{\alpha_1}(\VMij{\alpha_1}{i})^\top\Gtildeijmatrix(z)\submatrix{\mH_{\cdot \ell}}{ij} } 
    \le Cr(\log N)^{\xi + 3\gamma}\left(\frac{1}{q} + \sqrt{\frac{\Im \mtilde(z)}{N\eta}} + \frac{1}{N\eta} \right)
    \le C\Phi(z) .
    \end{equation}
    
    There will be $r$ terms in Equation~\eqref{ineq:Lo_expansion_2} controllable by Lemma \ref{lemma:main_quad} according to
    \begin{equation}\label{eq:initial_estimates:main_quad_bound}
       \left|\lambda^2_{\alpha_1}(\VMij{\alpha_1}{i})^\top\Gtildeijmatrix(z) \VMij{\alpha_1}{j}\right| 
        \le \left(\frac{C\lambda_{\alpha_1}\delocsq}{N} + C \frac{r^2\logN^{6\gamma}}{N\eta} \right)
    \end{equation}
    with $(\xi, \nu)$-high probability. 
    
    There will be $\binom{r}{2}$ terms in Equation~\eqref{ineq:Lo_expansion_2} controllable by Lemma~\ref{lemma:cross_term} as
    \begin{equation}\label{eq:initial_estimates:cross_term_bound}
      \left|\lambda_{\alpha_2}\lambda_{\alpha_1} (\VMij{\alpha_1}{i})^\top \Gtildeijmatrix(z) \VMij{\alpha_2}{j} \right| 
        \le C \frac{r^2 \logN^{6\gamma}}{N \eta}
    \end{equation}
    with $(\xi, \nu)$-high probability. 
    The upper bounds provided by Lemmas \ref{lemma:main_quad} and \ref{lemma:cross_term} are dominated by $\Phi(z)$ since for any $j$, looking at Equation~\eqref{eq:def:Phi},
     \begin{equation*}
       C \frac{r^2 \logN^{6\gamma}}{N \eta}
        \le C\frac{r^2(\log N)^{2\xi + 6\gamma}} {N\eta} 
        \le \frac{(\log N)^{\xi + 3\gamma}}{q} + (\log N)^{2\xi + 6\gamma} \left(\sqrt{\frac{\Im \mtilde(z)}{N\eta} } + \frac{1}{N\eta}\right) \le \frac{ C\Phi(z) }{ r^2 }
    \end{equation*}
    and %% Additional factor of r^2 account for the r choose 2 terms
    \begin{equation*}
    \frac{C\lambda_{\alpha_1}\delocsq}{N} 
    \le \frac{C\delocsq}{\sqrt{N}}
    \le \frac{r(\log N)^{\xi + 3\gamma}}{q} + r(\log N)^{2\xi + 6\gamma} \left(\sqrt{\frac{\Im \mtilde(z)}{N\eta} } + \frac{r}{N\eta}\right) \le \frac{ C\Phi(z) }{ r }.
    \end{equation*}
    Applying the above two bounds to Equations~\eqref{eq:initial_estimates:main_quad_bound} and~\eqref{eq:initial_estimates:cross_term_bound},
    \begin{equation}\label{eq:initial_estimates:main_quad_bound:done}
       r \left|\lambda^2_{\alpha_1}(\VMij{\alpha_1}{i})^\top\Gtildeijmatrix(z) \VMij{\alpha_1}{j}\right| 
       \le C\Phi(z)
    \end{equation}
    and
    \begin{equation}\label{eq:initial_estimates:cross_term_bound:done}
      \binom{r}{2}
      \left|\lambda_{\alpha_2}\lambda_{\alpha_1}(\VMij{\alpha_1}{i})^\top\Gtildeijmatrix(z) \VMij{\alpha_2}{j}\right| 
      \le C \Phi(z) .
    \end{equation}
    Applying Equations \eqref{eq:initial_estimates:HGH_bound}, \eqref{eq:initial_estimates:hb_bounds}, \eqref{eq:initial_estimates:main_quad_bound:done}, and \eqref{eq:initial_estimates:cross_term_bound:done} to \eqref{eq:initial_estimates:Gtildeij:checkpt2}, and taking the maximum over $i \ne j$, we have that on the event $\Ot(\eta)$ with $\xinu$-high-probability, 
    \begin{equation*}
        \Lo(z) := \max_{i\ne j} |\tilde{G}_{ij}(z)| \le C\Phi(z),
    \end{equation*}
    which establishes Equation~\eqref{eq:Lo:bound}.
    Note that to accommodate the union bound over $i,j \in [N]$, we increase the value of $\nu$ in the definition of $\xinu$-high-probability by a constant factor.
    
    Recalling the definition of $\Ztilde_i(z)$ from Equation \eqref{eq:def:Ztilde},  
    \begin{equation*}
        \Ztilde_{i}(z)
        = Q_i \left( (\Bi)^\top\Gtildeimatrix(z)\Bi\right)
        = (\Bi)^\top\Gtildeimatrix(z)\Bi - \E{(\Bi)^\top\Gtildeimatrix(z)\Bi \mid \submatrix{\mH}{i}} ,
    \end{equation*}
    where we have used the definition of $Q_i(\cdot )$ from Equation \eqref{eq:def:Q_i}.
    Adding and subtracting appropriate quantities and using properties of the conditional expectation,
    \begin{equation}\label{eq:Z_estimate:main}
    \begin{aligned}
         \Ztilde_{i}(z) &= \sum_{\alpha = N - r + 1} (\Hi)^\top\Gtildeimatrix(z)\lambda_{\alpha}\VM{\alpha} + \sum_{\alpha = N - r + 1} \lambda_{w}(\VM{\alpha})^\top\Gtildeimatrix(z)\Hi \\
           &~~~~~~~~~+ (\Hi)^\top\Gtildeimatrix(z)\Hi - \E{(\Hi)^\top\Gtildeimatrix(z)\Hi \mid \submatrix{\mH}{i}}
         % + (\Hi)^\top\Gtildei(z)f_{N-1}\EM{N-1} + f_N(\EM{N})^\top\Gtildei(z)\Hi 
    \end{aligned}
    \end{equation}
    The first two terms are controlled using Lemma \ref{lemma:hb_bound}:
    \begin{equation} \label{eq:Ztilde:firstTwoTerms}
    \begin{aligned}
    \Bigg|
    &\sum_{\alpha = N - r + 1} (\Hi)^\top\Gtildeimatrix(z)\lambda_{\alpha}\VM{\alpha} + \sum_{\alpha = N - r + 1} \lambda_{w}(\VM{\alpha})^\top\Gtildeimatrix(z)\Hi
    \Bigg| \\
    & ~~~~~~~~~~~~~~~~~~~~~~~~~~~~~~~~~~~~~~~~~~~~~~~~~~~~~~~~~~~~~~~~~\le 2Cr(\log N)^{\xi+ 3\gamma}\left(\frac{1}{q} + \sqrt{\frac{\Im \mtilde(z)}{N\eta}} + \frac{1}{N\eta}\right).
    \end{aligned}
    \end{equation}
    To control the latter two terms in Equation~\eqref{eq:Z_estimate:main}, we follow the argument first given in Lemma 3.13 from \cite{Erdos}.
    We have from Lemma \ref{lemma:concetration}, with $(\xi, \nu)$-high probability:
    \begin{equation}\label{eq:Z_estimate:first} 
    \begin{aligned}
       & \abs{(\Hi)^\top\Gtildeimatrix(z)\Hi - \E{(\Hi)^\top\Gtildeimatrix(z)\Hi \big| \submatrix{\mH}{i}}}  \\
       & ~~~~~~~~~~~~~~~~~~~~~~~~~~~~~~~~~~~~~~~~~~~~~~\le \abs{\sum_{k\ne l} h_{kl}\Gtildei_{kl}(z) h_{lk}} + \abs{\sum_{k}^{(\sm i)} \abs{h_{kk}}^2 \Gtildei_{kk}(z) - \frac{1}{N}\Gtildei_{kk}(z)},
    \end{aligned}
    \end{equation}
    where we have used the fact that
    \begin{equation*}
        \E{(\Hi)^\top\Gtildeimatrix(z)\Hi \mid \submatrix{\mH}{i}} 
        = \E{\sum_{k,l\ne i} h_{ki} \Gtildei_{kl}(z) h_{il} ~\Big| \submatrix{\mH}{i}} 
        = \sum_{k}^{(\sm i)}\frac{1}{N}\Gtilde_{kk}(z) .
    \end{equation*}
    The right-hand side of Equation \eqref{eq:Z_estimate:first} is estimated using Lemma \ref{lemma:concetration} for both terms, yielding
    \begin{equation*}
    \begin{aligned}
        &\abs{(\Hi)^\top\Gtildeimatrix(z)\Hi - \E{(\Hi)^\top\Gtildei(z)\Hi \mid \submatrix{\mH}{i}}} \\ 
           &~~~~~~~~~~~~~~~\le (\log N)^{2\xi} \left[\frac{\max_{k \ne l} \abs{\Gtildei_{kl}(z)}}{q} + \left(\frac{1}{N^2} \sum_{k\ne l} \abs{\Gtildei_{kl}(z)}^2 \right)^{1/2} \right] + (\log N)^{\xi}\frac{\max_{k}\abs{\Gtildei_{kk}(z)}}{q}.
    \end{aligned}
    \end{equation*}
   with $(\xi, \nu)$-high probability. 
   By Lemma \ref{lemma:bounded_resolvents}, on the event $\Ot(\eta)$, the $\Gtildei_{kk}(z)$ and $\Gtildei_{kl}(z)$ terms are controlled by constants, whence
    \begin{equation} \label{eq:Z_estimate:third}
    \begin{aligned}
       & \abs{(\Hi)^\top\Gtildeimatrix(z)\Hi - \E{(\Hi)^\top\Gtildei(z)\Hi \mid \submatrix{\mH}{i}}} \\ 
        &~~~~~~~~~~~~~~~~~~~~~~ \le (\log N)^{2\xi}\! \left[\frac{C(\log N)^{-\xi}}{q} + \left(\frac{1}{N^2} \sum_{k\ne l} \abs{\Gtildei_{kl}(z)}^2 \right)^{\!\!1/2} \right] + (\log N)^{\xi}\frac{C}{q} 
    \end{aligned}
    \end{equation}
    on the event $\Ot(\eta)$ with $(\xi, \nu)$-high probability. From Ward's Identity \citep[see Equation (3.6) in][]{Erdos_lecture_notes}, we have 
    \begin{equation} \label{eqn:lemma:Wards}
        \sum_{j = 1}^N |\Gtilde_{ij}(z)|^2 = \frac{\Im \Gtilde_{ii}(z)}{\eta} .
    \end{equation}
    So then on the event $\Ot(\eta)$, %by Equation \eqref{eqn:lemma:Wards},
    \begin{equation}\label{eq:Z_estimate:second}
    \frac{1}{N^2}\sum_{k,l}^{(\sm i)} \abs{\Gtildei_{kl}(z)}^2 
    =  \frac{1}{N^2\eta}\sum_k^{(\sm i)}\Im \Gtildei_{kk}(z) 
    \le \frac{C}{N\eta} \left( \Im \mtilde(z) + \Lo^2(z) \right) ,
    \end{equation}
    where in the inequality we have used the fact that $\Gtildei_{kk}(z) = \Gtilde_{kk}(z) + C\Lo^2(z)$ on $\Ot(\eta)$ by Lemma \ref{lemma:resolvent_identities} and the fact that $\mtilde(z) = \frac{1}{N} \Tr \Gtildematrix(z)$. 
    Equation \eqref{eq:Z_estimate:second} then implies that on $\Ot(\eta)$ with $\xinu$-high probability,
    \begin{equation}\label{eq:Z_estimate:fourth}
        \left(\frac{1}{N^2}\sum_{k,l}^{(\sm i)} \abs{\Gtildei_{kl}(z)}^2 \right)^{1/2} 
        \le C\left(\sqrt{\frac{\Im \mtilde(z)}{N\eta}} + \frac{\Lo}{\sqrt{N\eta}} \right) 
        \le C\left(\sqrt{\frac{\Im \mtilde(z)}{N\eta}} + \frac{\Phi(z)}{\sqrt{N\eta}} \right),
    \end{equation}
    where the second inequality follows from Equation \eqref{eq:Lo:bound}. Substituting Equation \eqref{eq:Z_estimate:fourth} into Equation \eqref{eq:Z_estimate:third} and collecting terms, we have 
    \begin{equation*} \begin{aligned}
        & \abs{(\Hi)^\top\Gtildeimatrix(z)\Hi - \E{(\Hi)^\top\Gtildeimatrix(z)\Hi \mid \submatrix{\mH}{i}}} \\
        &~~~~~~~~~~~~~~~~~~~~~~~~~~~~~~~~~~~~~~~~~~~~~~~~~~~ \le C\left(\frac{(\log N)^{\xi}}{q} + (\log N)^{2\xi} \left(\sqrt{\frac{\Im \mtilde(z)}{N\eta}} + \frac{\Phi(z)}{\sqrt{N\eta}} \right)\right) 
    \end{aligned} \end{equation*}
    on $\Ot(\eta)$ with $\xinu$-high probability. 
    Hence, recalling the definition of $D_L$ from Equation~\eqref{eq:def:DL} and our bounds on $L$ from Equation~\eqref{eq:L_constraints}, we have 
    \begin{equation}\label{eq:Z_estimate:fifth}
        \abs{(\Hi)^\top\Gtildeimatrix(z)\Hi - \E{(\Hi)^\top\Gtildeimatrix(z)\Hi \mid \submatrix{\mH}{i}}}
        \le C\Phi(z).
    \end{equation}
     on $\Ot(\eta)$ with $\xinu$-high probability.
     Applying the triangle inequality in Equation \eqref{eq:Z_estimate:main} and using Equations \eqref{eq:Z_estimate:fifth} and \eqref{eq:Ztilde:firstTwoTerms},
    \begin{equation}
        \abs{\Ztilde_i(z)} \le C\Phi(z) + Cr^2(\log N)^{\xi + 3\gamma}\left(\frac{1}{q} + \sqrt{\frac{\Im \mtilde(z)}{N\eta}} + \frac{1}{N\eta} \right) \le C\Phi(z)
    \end{equation}
     on $\Ot(\eta)$ with $\xinu$-high probability. 
     Hence, taking the maximum over $i \in [N]$ we have, after increasing the value of $\nu$ by a constant factor, on the event $\Ot(\eta)$ with $\xinu$-high-probability,
    \begin{equation*}
        \max_i |\Ztilde_i(z)| \le C\Phi(z) ,
    \end{equation*}
    yielding Equation~\eqref{eq:Ztilde_i}.
    
    Finally, to bound $\calAtilde_i(z)$, recall its definition from Equation~\eqref{eq:def:calAtilde},
    \begin{equation} \label{eq:initial_estimates:calAtilde:expansion}
    \begin{aligned}
    \calAtilde_i(z) &= L_{ii} - \sum_{\alpha = N - r + 1}^N \lambda_{\alpha}^2 (\VM{\alpha})^\top\Gtildei(z)\VM{\alpha} \\
    &~~~~~~~~~- \sum_{\alpha_1 \ne \alpha_2} \lambda_{\alpha_1}\lambda_{\alpha_2} (\VM{\alpha_1})^\top\Gtildei(z)\VM{\alpha_2} 
    + \frac{1}{N}\sum_j\frac{\Gtilde_{ij}(z)\Gtilde_{ji}(z)}{\Gtilde_{ii}(z)}
    \end{aligned} \end{equation}
    For the first term, Assumption \eqref{assumption:deloc} ensures that
    \begin{equation} \label{eq:Atilde:Lbound}
    L_{ii} = C\left( \frac{r\rhon\delocsq}{\sqrt{N}} \right) .
    \end{equation}
    There are $r$ terms in Equation~\eqref{eq:initial_estimates:calAtilde:expansion} that can be controlled using Lemma~\ref{lemma:main_quad}, providing the bounds of the form
    \begin{equation}\label{eq:initial_estimates:calAtilde:main_quad_term}
        \left|\lambda_{\alpha_1}^2(\VM{\alpha_1})^\top\Gtildeimatrix(z) \VM{\alpha_1}\right| 
        \le \left(\frac{C\lambda_j\delocsq}{N} + C \frac{r^2\logN^{6\gamma}}{N\eta}\right)
    \end{equation}
     with $\xinu$-high probability.
     There are $\binom{r}{2}$ cross-terms in Equation~\eqref{eq:initial_estimates:calAtilde:expansion} that can be controlled using Lemma~\ref{lemma:cross_term}, providing the bounds for $\alpha_1 \ne \alpha_2$ of the form
    \begin{equation}\label{eq:initial_estimates:calAtilde:cross_term}
      \abs{\lambda_{\alpha_2}\lambda_{\alpha_1}(\VM{\alpha_1})^\top\Gtildeimatrix(z) \VM{\alpha_2}} 
      \le C \frac{ r^2\logN^{6\gamma}}{N \eta}
    \end{equation}
    with $\xinu$-high probability.
    The last term in Equation~\eqref{eq:initial_estimates:calAtilde:expansion} is controlled using Lemma \ref{lemma:bounded_resolvents} and Ward's Identity given in Equation \eqref{eqn:lemma:Wards}. Altogether, applying Equations \eqref{eq:Atilde:Lbound}, \eqref{eq:initial_estimates:calAtilde:main_quad_term} and \eqref{eq:initial_estimates:calAtilde:cross_term}, Lemma \eqref{assumption:deloc} and Ward's Identity to Equation \eqref{eq:initial_estimates:calAtilde:expansion}, we have
    \begin{equation*}
        |\calAtilde_i(z)| \le C\frac{r\logN^{2\gamma}}{\sqrt{N}} + C\frac{r^4\logN^{6\gamma}}{N\eta}
    \end{equation*}
    on $\Ot(\eta)$ with $\xinu$-high probability. 
    Taking the maximum over $i \in [N]$, and reducing the value of $\nu$ by a constant factor yields that on the event $\Ot(\eta)$ with $\xinu$-high probability,
    \begin{equation*}
        \max_i |\calAtilde_i(z)| \le \frac{Cr\logN^{2\gamma}}{\sqrt{N}} + \frac{Cr^4\logN^{6\gamma}}{N\eta} ,
    \end{equation*}
    establishing Equation~\eqref{eq:calAtilde_i} and completing the proof.
\end{proof} 

The following is a straightforward adaptation of Lemma 7.7 from \cite{Erdos}.
Details are included for the sake of completeness.

\begin{lemma} \label{lemma:controlDiagResolvent}
    For $z = E + \iu\eta \in D_L$, we have on the event $\Ot(\eta)$ with $\xinu$-high probability,
    \begin{equation} \label{eq:Gtildemtilde:close}
        \max_i |\Gtilde_{ii}(z) - \mtilde(z)| \le C\Phi(z)
    \end{equation}
    and
    \begin{equation} \label{ineq:Ldtilde_bound}
       \Ld(z) \le \Lambdatilde(z) + C\Phi(z) .
    \end{equation}
\end{lemma}
\begin{proof}
    Recall the definition of $\Ytilde_i(z)$ in Equation \eqref{eq:def:Ytilde}. Then by Lemma \ref{lemma:Lo_Zi_Ai_estimates},
    \begin{equation*}
        \max_i|\Ytilde_i(z)| = \max_i \left| H_{ii} - \Ztilde_i(z) + \calAtilde_i(z) \right| \le \frac{C}{q} + C\Phi(z) + \frac{Cr\logN^{2\gamma}}{\sqrt{N}} + \frac{Cr^4\logN^{6\gamma}}{N\eta} 
    \end{equation*}
    on $\Ot(\eta)$ with $\xinu$-high probability. Then we see that
    \begin{equation*}
    \begin{aligned}
        |\Gtilde_{ii}(z) - \Gtilde_{jj}(z)| = \left| \frac{\Gtilde_{ii}(z)\Gtilde_{jj}(z)}{\Gtilde_{jj}(z)} - \frac{\Gtilde_{jj}(z)\Gtilde_{ii}(z)}{\Gtilde_{ii}(z)} \right| &= |\Gtilde_{jj}(z)||\Gtilde_{ii}(z)|\left| \frac{1}{\Gtilde_{jj}(z)} - \frac{1}{\Gtilde_{ii}(z)} \right| \\ 
        &= |\Gtilde_{jj}(z)||\Gtilde_{ii}(z)|\left| \Ytilde_j(z) - \Ytilde_i(z) \right|.
    \end{aligned}
    \end{equation*}
    The last equality follows from substituting the identity for $\Gtilde_{ii}(z)$, found in Lemma \ref{lemma:7.2}. 
    Then since on $\Ot(\eta)$, for any $i \in [N]$ uniformly, $ c \le |\Gtilde_{ii}(z)| \le C$, we have by triangle inequality with $\xinu$-high probability,
    \begin{equation*}
        |\Gtilde_{ii}(z) - \Gtilde_{jj}(z)| \le C\left( \frac{C}{q} + C\Phi(z) + \frac{Cr\logN^{2\gamma}}{\sqrt{N}} + \frac{Cr^4\logN^{6\gamma}}{N\eta}  \right)
    \end{equation*}
    Hence,
    \begin{equation*}
        \left|\frac{1}{N} \sum_j (\Gtilde_{ii}(z) - \Gtilde_{jj}(z)) \right| \le C\Phi(z) 
    \end{equation*}
    on $\Ot(\eta)$ with $\xinu$-high probability. Taking the maximum over $i \in [N]$ and recalling that
    \begin{equation*}
    \mtilde(z) = \frac{1}{N}\Tr\Gtildematrix(z)
    \end{equation*}
    yields Equation~\eqref{eq:Gtildemtilde:close}.
    
    Recalling the definition of $\Ld(z)$ from Equation~\eqref{eq:def:Ld-Lo} and applying the triangle inequality,
    \begin{equation*}
    \Ld(z) = \max_i | \Gtilde_{ii}(z) - \msc(z) | 
    \le | \mtilde(z) - \msc(z) | + \max_i | \Gtilde_{ii}(z) - \mtilde(z)|
    \le  \Lambdatilde(z) + \max_i | \Gtilde_{ii}(z) - \mtilde(z)| ,
    \end{equation*}
    where the second inequality follows from the definition of $\Lambdatilde(z)$ in Equation~\eqref{eq:def:stieltjes_errors}.
    Applying Equation~\eqref{eq:Gtildemtilde:close} then yields Equation~\eqref{ineq:Ldtilde_bound}, completing the proof.
    \end{proof} 
    
    \begin{lemma}[\cite{Erdos} Lemma 7.8, adapted]
        If $\eta \ge 2$, then $\Ot(\eta)$ holds with $(\xi, \nu)$-high probability, where $\Ot(\eta)$ is defined in Equation \eqref{eq:def:Ot}.
    \end{lemma}
    \begin{proof}
    Let $z = E + \iu\eta$.
    To estimate $\Lo$, we follow the approach from Lemma \ref{lemma:Lo_Zi_Ai_estimates}. Again, we apply Lemma \ref{lemma:resolvent_identity_corollary}
    to obtain
    \begin{equation*}
    \Gtilde_{ij}(z)
    = \Gtilde_{ii}(z)\Gtildei_{jj}(z)\left((\submatrix{\mB_{\cdot i}}{ij})^\top\Gtildeijmatrix(z) \submatrix{\mB_{\cdot j}}{ij} \right) - \Gtilde_{ii}(z)\Gtildei_{jj}(z)B_{ij} .
    \end{equation*}
    Using Lemma \ref{lemma:initial_elementary_resolvent_estimates}, we have with $(\xi, \nu)$-high probability,
    \begin{equation*}
    \abs{\Gtilde_{ij}(z)} 
    \le C\abs{(\submatrix{\mB_{\cdot i}}{ij})^\top\Gtildeijmatrix(z) \submatrix{\mB_{\cdot j}}{ij}} + C\abs{B_{ij}} 
    = C\abs{ \sum_{k,l}^{(\sm ij)}B_{ik} \Gtildeij_{kl}(z) B_{lj} } + C\abs{B_{ij}}.
    \end{equation*}
    Applying Lemma \ref{lemma:h_control} and Lemma \eqref{lemma:B_control},
    \begin{equation}\label{eq:initial_estimates:Lo_control:first}
    \abs{\Gtilde_{ij}(z)}
    \le C\abs{ \sum_{k,l}^{(\sm ij)}B_{ik} \Gtildeij_{kl}(z) B_{lj} } + C\frac{r\rhon\delocsq}{\sqrt{N}} + \frac{C}{q} 
    \end{equation}

    To estimate $\left| \sum_{k,l}^{(\sm ij)}B_{ik} \Gtildeij_{kl}(z) B_{lj} \right|$, we again have to estimate $ \left|\sum_{k,l}^{(\sm ij)} h_{ik}\Gtildeij_{kl}(z) h_{lj} \right|$.
    So then by Equation \eqref{lemma:concentration:quadratic_form} from Lemma \ref{lemma:concetration}, we have
    \begin{equation}\label{eq:initial_estimates:HGH_control:first_step}
    \left|\sum_{k,l}^{(\sm ij)} h_{ik}\Gtildeij_{kl}(z)h_{lj} \right|  
    \le (\log N)^{2\xi} \! \left[\frac{\max_k \abs{\Gtildeij_{kk}(z)}}{q^2} + \frac{\max_{i\ne j} \abs{\Gtildeij_{kl}(z)}}{q} + \left( \sum_{k,l}^{(\sm ij)} \abs{ \frac{ \Gtildeij_{kl}(z) }{ N }}^2 \right)^{1/2}\right]
    \end{equation}
    with $(\xi, \nu)$-high probability.
    Examining the rightmost term, we see that 
    \begin{equation}\label{eq:initial_estimates:HGH_control:second}
    \frac{1}{N^2}\sum_{k,l}^{(\sm ij)} \abs{\Gtildeij_{kl}(z)}^2 =  \frac{1}{N^2\eta}\sum_k^{(\sm ij)}\Im \Gtildeij_{kk}(z) = \frac{1}{N\eta}\Im \mtilde^{(\sm ij)}(z) \le \frac{C}{N}
    \end{equation}
    with $(\xi, \nu)$-high probability, where the first equality is by Ward's Identity from Equation \eqref{eqn:lemma:Wards}, 
    the second equality is by the definition of $\mtilde^{(\sm ij)}(z)$ given in Equation \eqref{eqn:def:mtildet},
    and the inequality follows from $\eta \ge 2$ and Lemma \ref{lemma:initial_elementary_resolvent_estimates}.
    Substituting Equation \eqref{eq:initial_estimates:HGH_control:second} into \eqref{eq:initial_estimates:HGH_control:first_step} yields
    \begin{equation*}
    \abs{\sum_{k,l}^{(\sm ij)} h_{ik}\Gtildeij_{kl}(z) h_{lj} }
    \le (\log N)^{2\xi} \left [\frac{\max_i \abs{\Gtildeij_{kk}(z)}}{q^2} + \frac{\max_{i\ne j} \abs{\Gtildeij_{kl}(z)}}{q} + \frac{C}{\sqrt{N}}\right]
    \end{equation*}
    with $(\xi, \nu)$-high probability.
    Distributing and using Lemma \ref{lemma:initial_elementary_resolvent_estimates}, we have
    \begin{equation}\label{eq:initial_estimates:HGH_control:third}
    \left|\sum_{k,l}^{(\sm ij)} h_{ik}\Gtildeij_{kl}(z) h_{lj} \right|  \le \frac{C(\log N)^{2\xi}}{q^2} + \frac{C(\log N)^{2\xi}\Lo(z)}{q} + \frac{C(\log N)^{2\xi}}{\sqrt{N}}
    \end{equation}
    % Substituting Equation \eqref{eq:initial_estimates:HGH_control:third} into Equation \eqref{eq:initial_estimates:Lo_control:first}, we obtain with $(\xi, \nu)$-high probability
    %  \begin{equation*}
    %  \abs{G_{ij}} \le  \frac{C(\log N)^{2\xi}}{q^2} + \frac{C(\log N)^{2\xi}\Lo}{q} + \frac{C(\log N)^{2\xi}}{\sqrt{N}} + \frac{C}{q}
    % \end{equation*}
    with $(\xi, \nu)$-high probability.
    For the terms of $\left| \sum_{k,l}^{(\sm ij)}B_{ik} \Gtildeij_{kl}(z) B_{lj} \right|$ controlled by Lemma \ref{lemma:hb_bound} (i.e., terms that are linear in $\Hi$, 
    we modify the argument in Lemma \ref{lemma:hb_bound} in places where the bound $R \le C\eta$ was used, such as Equation \eqref{eq:something1}, by replacing the bound with the trivial $R \le 1$.
    In doing so, we find that, as in the result of Lemma \ref{lemma:hb_bound},
    for any $\alpha_1 \in \{N-r+1, \dots, N \}$, when $\eta \ge 2$,
    \begin{equation*} 
        \abs{\lambda_{\alpha_1}(\VMij{\alpha_1}{i})^\top\Gtildeijmatrix(z)\submatrix{\mH_{\cdot i}}{ij}  } 
        \le Cr(\log N)^{\xi + 3\gamma}\left(\frac{1}{q} + \sqrt{\frac{\Im \mtilde(z)}{N\eta}} + \frac{1}{N\eta} \right)
    \end{equation*}
    with $(\xi, \nu)$-high probability.
    Then using Lemma \ref{lemma:initial_elementary_resolvent_estimates} and the assumption that $\eta \ge 2$ here, we have that with $(\xi, \nu)$-high probability,
    \begin{equation}\label{eqn:apx:large-eta-hb}
        \abs{\lambda_{\alpha_1}(\VMij{\alpha_1}{i})^\top\Gtildeijmatrix(z)\submatrix{\mH_{\cdot i}}{ij} } 
        \le  Cr(\log N)^{\xi + 3\gamma}\left(\frac{1}{q} + \frac{1}{\sqrt{N}} \right) .
    \end{equation}
    % The terms of $\abs{ \sum_{k,l}^{(\sm ij)}B_{ik} \Gtildeij_{kl}(z) B_{lj} }$ controlled by Lemmas \ref{lemma:main_quad} ($2r$ different terms) and \ref{lemma:cross_term} ($\binom{r}{2}$ different terms) are then of order $O(\delocsq/\sqrt{N})$ and $O(r^2\logN^{6\gamma}/N)$ respectively with $(\xi, \nu)$-high probability. 
    Applying the triangle inequality to $\abs{ \sum_{k,l}^{(\sm ij)}B_{ik} \Gtildeij_{kl}(z) B_{lj} }$ yields 
    $2r$ terms that, by Lemma \ref{lemma:main_quad}, are bounded by $C(\delocsq/\sqrt{N})$;
    and
    $\binom{r}{2}$ terms that, by Lemma \ref{lemma:cross_term}, are bounded by $C(r^2\logN^{6\gamma}/N)$, 
    both with $(\xi, \nu)$-high probability. That is,
    \begin{equation}\label{eqn:BGB_bound}
        \abs{ \sum_{k,l}^{(\sm ij)}B_{ik} \Gtildeij_{kl}(z) B_{lj} } \le 2r\frac{C\delocsq}{\sqrt{N}} + {C\binom{r}{2}} \frac{r^2\logN^{6\gamma}}{N}
    \end{equation}
    with $\xinu$-high probability.
    Altogether, from Lemmas \ref{lemma:main_quad}, \ref{lemma:cross_term}, and Equation \eqref{eqn:apx:large-eta-hb}, 
    we have that with $(\xi, \nu)$-high probability,
    \begin{equation}\label{eq:Gtilde:initial_estimate}
    \begin{aligned}
    |\Gtilde_{ij}(z)| &\le \frac{C(\log N)^{2\xi}}{q^2} + \frac{C(\log N)^{2\xi}\Lo(z)}{q} + \frac{C(\log N)^{2\xi}}{\sqrt{N}} \\
    &~~~~~~~~~+ Cr^2(\log N)^{\xi + 3\gamma}\left(\frac{1}{q} + \frac{1}{\sqrt{N}} \right) + \frac{Cr^4 \logN^{6\gamma}}{N} . % + \frac{Cr}{q}
    \end{aligned}
    \end{equation}
    Taking the maximum over $i\ne j$ (adjusting $\nu$ by a constant factor as needed), using the fact that $1/q \le 1/(r^2(\log N)^{3\xi + 3\gamma})$ from Equation \eqref{eq:def:q}, 
    and Assumption \eqref{assumption:rank}, we have, with $\xinu$-high-probability,
    \begin{equation}\label{eq:initial_Lo_estimate}
     \Lo(z) \le C(\log N)^{-2\xi} .
    \end{equation}
     
    To bound $|\Ztilde_i(z)|$, we apply Equation \eqref{eqn:apx:large-eta-hb} $2r$ times to obtain, from Equation \eqref{eq:Z_estimate:main}, \begin{equation}\label{eq:Ztilde_initial_bound}
    \begin{aligned}
        \abs{\Ztilde_i(z)} &\le 
        Cr^2(\log N)^{\xi + 3\gamma}\left(\frac{1}{q} + \frac{1}{\sqrt{N}} \right)
        + \Bigg| (\Hi)^\top\Gtildeimatrix(z)\Hi \\
        &~~~~~~~~~~~~~~~~~~ - \E{(\Hi)^\top\Gtildeimatrix(z)\Hi \mid \submatrix{\mH}{i}} \Bigg|
    \end{aligned} \end{equation} with $(\xi, \nu)$-high probability. 
    As in the proof of \ref{lemma:Lo_Zi_Ai_estimates}, the right two terms of Equation \eqref{eq:Ztilde_initial_bound} are bounded as 
    \begin{equation*} \begin{aligned}
        &\abs{(\Hi)^\top\Gtildeimatrix(z)\Hi - \E{(\Hi)^\top\Gtildeimatrix(z)\Hi \mid \submatrix{\mH}{i}}} \\
        &~~~~~~~~~~~~~\le (\log N)^{2\xi} \left[\frac{\max_{k \ne l} \abs{\Gtildei_{kl}(z)}}{q} + \left(\frac{1}{N^2} \sum_{k\ne l} \abs{\Gtildei_{kl}(z)}^2 \right)^{1/2} \right] + (\log N)^{\xi}\frac{\max_{k}\abs{\Gtildei_{kk}(z)}}{q}
    \end{aligned}
    \end{equation*}
    with $(\xi, \nu)$-high probability.
    Using Lemma \ref{lemma:initial_elementary_resolvent_estimates} and Equation \eqref{eq:initial_Lo_estimate}, we have 
    \begin{equation} \label{eq:Gtildei:inter}
    \begin{aligned}
       & \abs{(\Hi)^\top\Gtildeimatrix(z)\Hi - \E{(\Hi)^\top\Gtildeimatrix(z)\Hi \mid \submatrix{\mH}{i}}}  \\
        &~~~~~~~~~~~~~~~~~~~~~~~~~~~~~~~~~~~~~~~~~~~~~~~~~~~~~~ \le \frac{C}{q} + \frac{C(\log N)^{\xi}}{q} + (\log N)^{2\xi}\left(\frac{1}{N^2} \sum_{k\ne l} \abs{\Gtildei_{kl}(z)}^2 \right)^{1/2}
    \end{aligned}
    \end{equation}
    with $(\xi, \nu)$-high probability.
    Following the arguments leading up to Equation \eqref{eq:initial_estimates:HGH_control:second}, using Ward's identity, we obtain
    \begin{equation*}
    \left(\frac{1}{N^2} \sum_{k\ne l} \abs{\Gtildei_{kl}(z)}^2 \right)^{1/2} \le \frac{C}{\sqrt{N}}
    \end{equation*}
    with $(\xi, \nu)$-high probability. 
    Hence we have, applying this bound to Equation~\eqref{eq:Gtildei:inter}, 
    \begin{equation*}
    \abs{(\Hi)^\top\Gtildei(z)\Hi - \E{(\Hi)^\top\Gtildei(z)\Hi \mid \submatrix{\mH}{i}}}  
    \le \frac{C(\log N)^{\xi}}{q} + \frac{C(\log N)^{2\xi}}{\sqrt{N}} 
    \end{equation*}
    with $(\xi, \nu)$-high probability.
    Applying this bound to Equation \eqref{eq:Ztilde_initial_bound} yields
    \begin{equation*}
    \abs{\Ztilde_i(z)} \le Cr^2(\log N)^{\xi + 3\gamma}\left(\frac{1}{q} + \frac{1}{\sqrt{N}} \right) + \frac{C(\log N)^{2\xi}}{\sqrt{N}} \le C(\log N)^{-2\xi}
    \end{equation*}
    with $(\xi, \nu)$-high probability,
    where in the second inequality we applied Equation \eqref{eq:def:q} to upper bound $q$.
    Taking the maximum over $i$ and reducing the value of $\nu$ yields the bound for $\Ztilde_i(z)$ with $\xinu$-high-probability. 
    
    To estimate $\calAtilde_i(z)$, we follow the approach taken in Lemma \ref{lemma:Lo_Zi_Ai_estimates}. 
    After applying the triangle inequality,
    one application of Assumption \eqref{assumption:deloc}, $r$ applications of \ref{lemma:main_quad}, and $\binom{r}{2}$ applications of Lemma \ref{lemma:cross_term} yield that with $(\xi, \nu)$-high probability,
    \begin{equation*}
    \abs{\calAtilde_i(z)} 
    \le
    \frac{Cr\rhon\delocsq}{\sqrt{N}}
    + C\left(\frac{r\sqrt{N}\delocsq}{N} +  \frac{r^3\logN^{6\gamma}}{N\eta} \right)
    +  C\frac{r^4\logN^{6\gamma}}{N\eta} 
    + \abs{ \frac{1}{N}\sum_{j=1}^N\frac{\Gtilde_{ij}(z)\Gtilde_{ji}(z)}{\Gtilde_{ii}(z)} } .
    \end{equation*}
    Substituting our choice of $\eta \ge 2$ and applying Assumption \eqref{assumption:rank}, we obtain
    \begin{equation}\label{eq:calAtilde_initial_estimate_1}
        \abs{\calAtilde_i(z)} \le C\frac{\logN^{2\gamma + \zeta}}{\sqrt{N}} + C\frac{\logN^{6\gamma + 4\zeta}}{N} + \left| \frac{1}{N}\sum_{j=1}^N\frac{\Gtilde_{ij}(z)\Gtilde_{ji}(z)}{\Gtilde_{ii}(z)} \right|
    \end{equation}
    with $(\xi, \nu)$-high probability.
    From Lemma \ref{lemma:resolvent_identity_corollary}, we have
    \begin{equation*}
        \frac{\Gtilde_{ij}(z)}{\Gtilde_{ii}(z)}\Gtilde_{ji}(z) = -\Gtildei_{jj}(z)(B_{ij} - \sum_{k,l}^{(\sm ij)}B_{ik}\Gtilde_{kj}^{(\sm ij)}(z) B_{lj}) \Gtilde_{ji}(z)
    \end{equation*}
    Hence by Lemma \ref{lemma:initial_elementary_resolvent_estimates}, we have
    \begin{equation} \label{eq:RHSLo} \begin{aligned}
    \frac{1}{N} \sum_{j=1}^N\abs{\frac{\Gtilde_{ij}(z)}{\Gtilde_{ii}(z)}\Gtilde_{ji}(z)} 
        &\le \abs{C\Lo(z) \left(B_{ij} - \sum_{k,l}^{(\sm ij)}B_{ik}\Gtilde_{kj}^{(\sm ij)}(z) B_{lj} \right) } \\
        & \le C\Lo(z) \left(\abs{B_{ij}} + \abs{\sum_{k,l}^{(\sm ij)}B_{ik}\Gtilde_{kj}^{(\sm ij)}(z) B_{lj} } \right)
    \end{aligned} \end{equation}
    with $(\xi, \nu)$-high probability.
    To bound the absolute value of the right-hand side of the last inequality, we apply Equations \eqref{eqn:BGB_bound} and \eqref{eq:Gtilde:initial_estimate} to obtain
    \begin{equation*}
    \abs{\sum_{k,l}^{(\sm ij)}B_{ik}\Gtilde_{kj}^{(\sm ij)}(z) B_{lj} }
    \le C \frac{r^2(\log N)^{\xi+3\gamma}}{q} 
    \le C (\log N)^{-2\xi} ,
    \end{equation*}
    where we have used $q \ge r^2(\log N)^{3\xi + 3\gamma}$ from Equation \eqref{eq:def:q}.
    Applying this to Equation~\eqref{eq:RHSLo}, controlling $|B_{ij}|$ with Lemma \ref{lemma:B_control}, and bounding $\Lo \le C$ by Lemma \ref{lemma:initial_elementary_resolvent_estimates}, we have
    \begin{equation*}
        \frac{1}{N} \sum_{j=1}^N\abs{\frac{\Gtilde_{ij}(z)}{\Gtilde_{ii}(z)}\Gtilde_{ji}(z)} \le C(\log N)^{-2\xi}
    \end{equation*}
    with $(\xi, \nu)$-high probability.
    Substituting this bound into Equation \eqref{eq:calAtilde_initial_estimate_1},
    \begin{equation*}
       % \abs{\calAtilde_i} \le C\left(\frac{1}{\sqrt{N}} + \frac{r}{\sqrt{N}} + \frac{r}{N\eta}  + \frac{{{r}\choose{2}}}{N\eta} + (\log N)^{-2\xi} \right) \le (\log N)^{-2\xi}
      \abs{\calAtilde_i(z)} \le C\frac{\logN^{2\gamma + \zeta}}{\sqrt{N}} + C\frac{\logN^{6\gamma + 4\zeta}}{N}
      + C(\log N)^{-2\xi} \le C(\log N)^{-2\xi}
    \end{equation*}
    for sufficiently large $N$ with $(\xi, \nu)$-high probability. 
    Taking the maximum over $i$ (reducing the value of $\nu$ by a constant factor, if needed) yields the result with $\xinu$-high-probability.
    All of these rates are the exact same as in Lemma 7.8 from \cite{Erdos}. The rest of the proof follows from the arguments contained in Lemma 3.16 in \cite{Erdos}.
    Hence, the remainder of the proof follows exactly as given for Lemmas 7.8 and 7.9 from \cite{Erdos}, and details are omitted.
    % Moreover, following their lattice argument given in Corollary 3.19 in \cite{Erdos}, we have with $\xinu$-high-probability,
    % \begin{equation}
    %     \bigcap_{z\in D} \left\{\max_{1\le i,j\le N}\abs{\Gtilde_{ij}(z) - I(i = j)\msc(z)} \le (\log N)^{\xi} \left(\frac{1}{q} + \sqrt{\frac{\Im \msc(z)}{N\eta}} + \frac{1}{N\eta} \right) \right\}.
    % \end{equation}
    \end{proof} % KDL checked Jan 6, 2026 20:45

    The following result, which establishes the global and strong local semicircle laws for the eigenvalues of $\mB$, might be of independent interest. 
    Note that the conditions on $\xi$ and $q$ are slightly stronger than those made elsewhere in this appendix.
    
    \begin{theorem} \label{thm:semicirc}
    Let $z = E + \iu \eta \in D$ and define $\kappa_E = ||E| - 2|$.
    Suppose $\mB$ satisfies Equation \eqref{eq:def:B}, and Assumptions \eqref{assumption:deloc} -- \eqref{assumption:moments} hold, and suppose that
    \begin{equation*}
        \xi = \frac{C_0}{2}\log(\log N) \quad \text{and} \quad  r^2(\log N)^{C_3 + 3\gamma} \le q . 
    \end{equation*}
    Then there are universal constants $C_3, C_4 >0$ such that for some $\nu > 0$ that depends on $C_0, C_1,$ and $C_2$ in Equations \eqref{def:H}, \eqref{eq:def:q}, and Definition \ref{def:H},
    such that the event
    \begin{equation}\label{eqn:Global_law}
            | \mtilde(z) - \msc(z) | \le \frac{r\pi}{N\eta} + (\log N)^{C_4 \xi} \left( \min \left( \frac{1}{q^2\sqrt{\kappa_E + \eta}}, \frac{1}{q} \right) + \frac{1}{N\eta} \right ) 
        \end{equation}
    holds with $\xinu$-high probability uniformly over all $z \in D$.
        Moreover, we have with $\xinu$-high probability uniformly over all $z \in D$
        \begin{equation}\label{thm:strong-law} 
\max_{1\le i,j\le N}\abs{\Gtilde_{ij}(z) - I(i = j)\msc(z)} \le \frac{r^2(\log N)^{C_3\xi + 3\gamma}}{q} + r^2(\log N)^{C_3\xi + 6\gamma} \left(\sqrt{\frac{\Im \msc(z)}{N\eta} } + \frac{r^2}{N\eta}\right) .
        \end{equation}
    \end{theorem} 
    \begin{proof}
    To show Equation \eqref{eqn:Global_law}, we have from Lemma \ref{lemma:7.1},
    \begin{equation*}
        \abs{\Lambdatilde(z) - \Lambda} \le \frac{r\pi}{N\eta} .
    \end{equation*}
    From which, observe that 
    \begin{equation*}
        \Lambdatilde(z) \le \frac{r\pi}{N\eta} + \Lambda(z) .
    \end{equation*}
    Applying Theorem 2.8 from \cite{Erdos}, specifically Equation (2.16), and recalling the definition of $\Lambdatilde(z)$ and $\Lambda(z)$ from Equation \eqref{eq:def:stieltjes_errors} we have uniformly over $z \in D$ with $\xinu$-high probability, the event
    \begin{equation*}
    |\mtilde(z) - \msc(z)| \le \frac{r\pi}{N\eta} 
    + (\log N)^{C_2\xi} \left( \min \left( \frac{1}{q^2\sqrt{\kappa_E + \eta}}, \frac{1}{q} \right) + \frac{1}{N\eta} \right) ,
    \end{equation*}
    as desired.

    To show Equation \eqref{thm:strong-law},
    apply Lemmas \ref{lemma:Lo_Zi_Ai_estimates} and \ref{lemma:controlDiagResolvent} and Equation \ref{eqn:Global_law} to obtain the bound uniformly for $z \in D_L$.
    To extend the bound uniformly for $z \in D$, follow the arguments given in \cite{Erdos} after Equation (4.28) until the end of section 4.1.
    \end{proof} 

\subsection{Proof of Delocalization}

We are now prepared to prove our main result on the delocalization of the eigenvectors of $\mB$. We remark that only the weak local law is needed to show delocalization.
\begin{theorem}
\label{theorem:perturbed_bulk_eigenvectors_delocalized}
Let $\mB$ satisfy Equation \eqref{eq:def:B} and let $\mathcal{S} = \{N-r+1, \dots, N \}$.
Under Assumptions \eqref{assumption:deloc} -- \eqref{assumption:moments},
for some $\nu > 0$ that depends on $C_0, C_1,$ and $C_2$ in Definition \ref{def:H}, Equations \eqref{eq:def:xi} and \eqref{eq:def:q},  we have
\begin{equation}
   \max_{\alpha \in \mathcal{S}} \max_j \abs{\vhat_{j\alpha}} \le  C\frac{r^2(\log N)^{4\xi+6\gamma}}{\sqrt{N}}
\end{equation}
with $\xinu$-high probability.
\end{theorem}
\begin{proof}
Following \cite{Erdos}, set
\begin{equation}\label{eqn:deloc-L-choice}
    L = 8\xi ,
\end{equation} and set
\begin{equation}\label{eqn:deloc-eta-choice}
\eta = \frac{r^4(\log N)^{L+12\gamma}}{N} .
\end{equation}
By Equation~\eqref{ineq:Ldtilde_bound} in Lemma~\ref{lemma:controlDiagResolvent}, 
\begin{equation} \label{eq:Lambdatilde:UB}
    \Ld(z) \le \Lambdatilde(z) + C\Phi(z).
\end{equation}
We examine $\Lambdatilde(z)$ first.
By Lemma~\ref{lemma:7.1} and Theorem 3.1 from \cite{Erdos}, we have
\begin{equation} \label{ineq:Lambdatilde_bound}
\Lambdatilde(z) \le | \Lambdatilde(z) - \Lambda(z) | + \Lambda(z)
\le \frac{r\pi}{N\eta} + \frac{C(\log N)^{\xi}}{\sqrt{q}} + \frac{C(\log N)^{2\xi}}{(N\eta)^{1/3}} .
\end{equation}
Applying this to Equation~\eqref{eq:Lambdatilde:UB} and recalling the definition of the spectral control parameter $\Phi(z)$ from Equation \eqref{eq:def:Phi},
%\begin{equation*} %\label{eq:control_parameter_bound}
%    \Phi(z) := \frac{r^2(\log N)^{\xi + 3\gamma}}{q} + r^2(\log N)^{2\xi + 6\gamma} \left(\sqrt{\frac{\Im \mtilde(z)}{N\eta} } + \frac{r^2}{N\eta}\right) .
%\end{equation*}
%Applying the above two display equations to Equation~\eqref{eq:Lambdatilde:UB},
\begin{equation}\label{eq:eigenvector_delocalization:Ld_bound_1}
\Ld(z) 
\le \frac{r\pi}{N\eta} + \frac{C(\log N)^{\xi}}{\sqrt{q}} 
+ \frac{C(\log N)^{2\xi}}{(N\eta)^{1/3}} + \frac{r^2(\log N)^{\xi + 3\gamma}}{q} 
+ r^2(\log N)^{2\xi + 6\gamma} \left(\sqrt{\frac{\Im \mtilde(z)}{N\eta} } 
+ \frac{r^2}{N\eta}\right) .
\end{equation}
By the triangle inequality,
\begin{equation*}
    \Im \mtilde(z) \le |\Im \mtilde(z) - \Im \msc(z)| + \Im \msc(z) 
    \le |\mtilde(z) - \msc| + \Im \msc(z) = \Lambdatilde(z) + \Im \msc(z)
\end{equation*}
where in the equality, we have substituted the definition of $\Lambdatilde(z)$ from Equation~\eqref{eq:def:stieltjes_errors}.

Now since $C^{-1} \le |\msc (z)| \le C$ for $z \in D_L$ by Lemma 3.2 in \cite{Erdos},

and using Equation~\eqref{ineq:Lambdatilde_bound}, we have 
\begin{equation*}
    \Im \mtilde(z) \le \frac{r\pi}{N\eta} + \frac{C(\log N)^{\xi}}{\sqrt{q}} + \frac{C(\log N)^{2\xi}}{(N\eta)^{1/3}} + C
\end{equation*}
Substituting our value for $\eta$ from Equation \eqref{eqn:deloc-eta-choice},
recalling the bounds on $q$ from Equation \eqref{eq:def:q} and simplifying, 
we have
\begin{equation}\label{eq:eigenvector_delocalization:m_tilde_bound_1}
    \Im \mtilde(z) \le \frac{r\pi}{r^4(\log N)^{8\xi+12\gamma}} + \frac{C(\log N)^{\xi}}{r(\log N)^{(3\xi+ 3\gamma)/2}} + \frac{C(\log N)^{2\xi}}{(r^4(\log N)^{8\xi+12\gamma})^{1/3}} + C \le C
\end{equation}
Hence, we find that by substituting Equation \eqref{eq:eigenvector_delocalization:m_tilde_bound_1} into Equation \eqref{eq:eigenvector_delocalization:Ld_bound_1},
\begin{equation*}
\begin{aligned}
    \Ld(z) \le \frac{r\pi}{r^4(\log N)^{8\xi+12\gamma}} &+ \frac{C(\log N)^{\xi}}{r(\log N)^{(3\xi+ 3\gamma)/2}} + \frac{C(\log N)^{2\xi}}{(r^4(\log N)^{8\xi+12\gamma})^{1/3}} + \frac{r^2(\log N)^{\xi + 3\gamma}}{r^2(\log N)^{3\xi + 3\gamma}}  \\
    &~~~~~~~~~~~~~~~~~~+ r^2(\log N)^{2\xi + 6\gamma}\left(\sqrt{\frac{C}{r^4(\log N)^{8\xi+12\gamma}}} + \frac{r^2}{r^4(\log N)^{8\xi+12\gamma}}\right) .
\end{aligned}
\end{equation*}
Bounding decaying terms of the form $(\log n)^{-c}$ by a constant,
\begin{equation*}
    \Ld(z) \le C .
\end{equation*}
Recalling the definition of $\Ld(z)$ from \eqref{eq:def:Ld-Lo}, we have for any $j \in [N]$ 
\begin{equation*}
    |\Gtilde_{jj}(z) - \msc(z) | \le C .
\end{equation*}
Recalling again that $|\msc (z)|$ is bounded by constants from Lemma 3.2 in \cite{Erdos}, 
we have that
\begin{equation*}
     |\Gtilde_{jj}(z)| \le C,
\end{equation*}
which then implies that for $\alpha \notin \{N-r+1, N-r +2, \dots, N \}$,
\begin{equation*}
\begin{aligned}
    C \ge \Im \Gtilde_{jj}(\lambdahat_{\alpha} + \iu\eta) = \sum_{\beta = 1}^N
    \frac{\eta|\vhat_{j\beta}|^2}{(\lambdahat_{\beta} - \lambdahat_{\alpha})^2 + \eta^2} 
    = \frac{\eta\abs{\vhat_{j\alpha}}^2}{\eta^2} + &\sum_{\beta \ne \alpha}
    \frac{\eta|\vhat_{j\beta}|^2}{(\lambdahat_{\beta} - \lambdahat_{\alpha})^2 + \eta^2}
    \ge \frac{\abs{\vhat_{j\alpha}}^2}{\eta} .
\end{aligned}
\end{equation*} 
Rearranging and recalling our choice of $\eta$ from Equation \eqref{eqn:deloc-eta-choice} and $L$ from Equation \eqref{eqn:deloc-L-choice}, 
\begin{equation*}
    \abs{\vhat_{j\alpha}} \le C\frac{r^2(\log N)^{4\xi+6\gamma}}{\sqrt{N}} ,
\end{equation*}
as we set out to show.
\end{proof} 

\subsection{Auxiliary Estimates}

Here we collect a handful of useful matrix results, mostly related to resolvents.
%\begin{lemma}\label{lemma:trace-Stieltjes-relation}
%For a matrix $\mB$, denote its resolvent as $\Gtildematrix(z)$, and let the $\mtilde(z)$ be the Stieltjes transform of the empirical eigenvalue density of $\mB$.
%We have $\mtilde(z) = \frac{1}{N}\Tr \Gtildematrix(z)$ 
%\end{lemma}
%\begin{proof}
%    Recall that $\mtilde(z)$ is the Stieltjes transform of $\mutilde(x) = N^{-1} \sum_{\alpha} \delta(x - \lambdahat_\alpha)$. Thus,
%    \begin{equation*}
%        \mtilde(z) = \int_{\mathbb{R}}\frac{\mutilde(x)}{x-z}dx =  \frac{1}{N}\int_{\mathbb{R}}\frac{\sum_{\alpha} \delta(x - \lambdahat_\alpha)}{x-z}dx = \frac{1}{N} \sum_{\alpha}\int_{\mathbb{R}}\frac{ \delta(x - \lambdahat_\alpha)}{x-z}dx =\frac{1}{N} \sum_{\alpha} \frac{1}{\lambdahat_{\alpha}-z} = \frac{1}{N}\Tr \Gtildematrix(z),
%    \end{equation*}
%    as we set out to show.
%\end{proof} 

\begin{lemma}[Resolvent Identities; \cite{Erdos_lecture_notes} Lemma 3.5]
\label{lemma:resolvent_identities}
For any Hermitian matrix $\mB$ and $\T \subset [N]$, we have for any $i,j,\notin \T$ and $i \ne j$
\begin{equation}\label{eq:resolvent:a}
    G_{ij}^{(\sm\T)}(z) = -G_{ii}^{(\sm\T)}(z) \sum_{k}^{(\sm\T i)}B_{ik}G_{kj}^{(\sm \T i)}(z) = -G_{jj}^{(\sm\T)}(z)\sum_k ^{(\sm\T j)}G_{ik}^{(\sm\T j)}(z)B_{kj}
\end{equation} 
Moreover, for any $i,j,$ and $k \notin \T$ and $i,j \ne k$
    \begin{equation}\label{eq:resolvent:b}
          G_{ij}^{(\sm\T)}(z) =  G_{ij}^{(\sm\T k)}(z) + \frac{ G_{ik}^{(\sm\T)}(z) G_{kj}^{(\sm\T)}(z)}{ G_{kk}^{(\sm\T)}(z)}
    \end{equation}
\end{lemma}
\begin{proof}
    A proof is given in Appendix A of \cite{Erdos_lecture_notes}.
\end{proof}

\begin{lemma}\label{lemma:resolvent_identity_corollary} 
For a resolvent $\mathbf G = (\mF - z\mI)^{-1}$, where $\mF$ is Hermitian, 
we have for any $i\ne j \in [N]$
\begin{equation*}
     G_{ij}(z) = G_{ii}(z)G_{jj}^{(\sm i)}(z) \sum_{k,l}^{(\sm ij)}F_{ik}G_{kj}^{(\sm ij)}(z)F_{lj} - G_{ii}(z)G_{jj}^{(\sm i)}(z)H_{ij} .
\end{equation*}
\end{lemma}
\begin{proof}
This follows from Equation \eqref{eq:resolvent:a} in Lemma~\ref{lemma:resolvent_identities}. In particular, setting $\T = \emptyset$ in that lemma, we have
\begin{equation}\label{eq:appendix:resolvent_1}
    G_{ij}(z) = -G_{ii}(z) \sum_k^{(\sm i)}F_{ik}G_{kj}^{(\sm i)}(z) = -G_{ii}(z)\sum_k^{(\sm ij)}F_{ik}G_{kj}^{(\sm i)}(z) -G_{ii}(z)G_{jj}^{(\sm i)}(z)F_{ij} .
\end{equation}
Using Equation \eqref{eq:resolvent:a} again with $\T = \{i\}$, we have
\begin{equation*} %\label{eq:appendix:resolvent_2}
G_{kj}^{(\sm i)}(z) = -G_{jj}^{(\sm i)}(z)\sum_{l}^{(\sm ij)}G_{kl}^{(\sm ij)}(z)F_{lj}.
\end{equation*}
Substituting this into \eqref{eq:appendix:resolvent_1} completes the proof. %, we have the claim
%\begin{equation}
%    G_{ij}(z) = G_{ii}(z)G_{jj}^{(\sm i)}(z) \sum_{k,l}^{(\sm ij)}B_{ik}G_{kj}^{(\sm ij)}(z)B_{lj} - G_{ii}(z)G_{jj}^{(\sm i)}(z)B_{ij}.
%\end{equation}
\end{proof} 

\begin{lemma}\label{lemma:minor_resolvents}
Let $\tau < N$ be constant with respect to $N$. For $\abs{\T} \le \tau$,
where $\T \subset [N]$ we have that on the event $\Ot(\eta)$ defined in Equation \eqref{eq:def:Ot},
\begin{equation*}
   \max_{i\ne j} \abs{\Gtildet_{ij}(z)} \le C\max_{i\ne j} \abs{\Gtilde_{ij}(z)} \le (\log N)^{-\xi}
\end{equation*}
\end{lemma}
\begin{proof}
    This is an application of Lemma \ref{lemma:resolvent_identities}, from which, for $k \in \T$
    \begin{equation*}
        \Gtilde_{ij}(z) = \Gtilde_{ij}^{(\sm k)}(z) + \frac{\Gtilde_{ik}(z)\Gtilde_{kj}(z)}{\Gtilde_{kk}(z)}.
    \end{equation*}
    Applying the reverse triangle inequality, 
    \begin{equation*}
        \abs{ \abs{\Gtilde_{ij}^{(\sm k)}(z) } - \abs{\frac{\Gtilde_{ik}(z)\Gtilde_{kj}(z)} {\Gtilde_{kk}(z)}} } \le \abs{\Gtilde_{ij}(z) } ,
    \end{equation*}
    from which the triangle inequality yields
    \begin{equation}\label{eq:appendix:minor_resolvents_1}
        \abs{\Gtilde_{ij}^{(\sm k)}(z)} 
        \le  \abs{\Gtilde_{ij}(z)} + \frac{\abs{\Gtilde_{ik}(z)}\abs{\Gtilde_{kj}}}{\abs{\Gtilde_{kk}(z)}} .
    \end{equation}
    By Lemma \ref{lemma:bounded_resolvents},
    we have that on $\Ot(\eta)$, $c \le \abs{\Gtilde_{ii}(z)} \le C$, which then implies
    \begin{equation*} %\label{eq:appendix:minor_resolvents_2}
         \frac{\abs{\Gtilde_{ik}(z)}\abs{\Gtilde_{kj}(z)}}{\abs{\Gtilde_{kk}(z)}} \le C\max_{i\ne j}\abs{\Gtilde_{ij}(z)}^2.
    \end{equation*}
    Applying this to Equation \eqref{eq:appendix:minor_resolvents_1} and taking the maximum over $i\ne j$ yields
    \begin{equation*} %\label{eq:appendix:minor_resolvents_3}
        \max_{i\ne j} \abs{\Gtilde_{ij}^{(\sm k)}(z)} \le  \max_{i\ne j}\abs{\Gtilde_{ij}(z)} +  C\max_{i\ne j}\abs{\Gtilde_{ij}(z)}^2 \le  \max_{i\ne j}\abs{\Gtilde_{ij}(z)} (1 +  C\max_{i\ne j}\abs{\Gtilde_{ij}(z)}) \le \max_{i\ne j}\abs{\Gtilde_{ij}(z)}.
    \end{equation*}
    where we used in the last inequality that on $\Ot(\eta)$, we have $\max_{i \ne j} \abs{\Gtilde_{ij}(z)} \le (\log N)^{-\xi} \le C$ by construction of $\Ot(\eta)$. Recalling the definition of $\Lo(z)$ from Equation \eqref{eq:def:Ld-Lo}
    yields that on $\Ot(\eta)$,
    \begin{equation}
       \max_{i\ne j}\abs{\Gtilde_{ij}^{(\sm k)}(z)} \le  C\Lo(z) \le C(\log N)^{-\xi}.
    \end{equation}
    where the second inequality follows from the definition of $\Ot(\eta)$.
    To see that $ \max_{i\ne j}\abs{\Gtildet_{ij}(z)} \le C\Lo(z)$, we repeat the above for each $k \in \T$. 
    In particular, in the next iteration, set $\Gtildematrix := \Gtildematrix^{(\sm k)}$ and repeat the above for $k' \in \T\sm \{k\}$. 
\end{proof} 

\begin{lemma} \label{lemma:h_control}
Let $\mH$ be matrix whose entries, $H_{ij}$ are independent up to symmetry and satisfy Equation \eqref{eq:def:h_moments}. For $C > 0$ sufficiently large, we have with $(\xi, \nu)$-high probability,
\begin{equation}
    \max_{i,j} \abs{H_{ij}} \le \frac{C}{q} .
\end{equation}
\end{lemma}
\begin{proof}
    Set $p = \nu (\log N)^{\xi}$ in Equation \eqref{eq:def:h_moments}. Then by Markov's inequality,
    \begin{equation*}
        \Prob{\abs{H_{ij}} \ge \frac{C}{q}} = \Prob{\abs{H_{ij}}^p \ge \frac{C^p}{q^p}} \le \frac{\E{\abs{H_{ij}}^p}q^p}{C^p} 
        \le  \frac{C_1^pq^p}{Nq^{p-2}C^p} 
        = \left(\frac{C_1}{C}\right)^p \frac{q^2}{N} \le \left(\frac{C_1}{C}\right)^p,
    \end{equation*}
    where the second inequality follows from Equation \eqref{eq:def:h_moments} and the last inequality follows from Equation \eqref{eq:def:q}. Taking a union bound over all $i,j \in [N]$ yields
    \begin{equation*}
        \Prob{\max_{i,j}\abs{H_{ij}} \ge \frac{C}{q}} \le N^2\left(\frac{C_1}{C}\right)^p = e^{2(\log N)} \left(\frac{C_1}{C}\right)^{\nu (\log N)^{\xi}}.
    \end{equation*}
    So then for sufficiently large $C > 0$, and using the fact that $1 < \xi$ by Definition~\eqref{eq:def:xi},
    \begin{equation*}
        \Prob{\max_{i,j}\abs{H_{ij}} \ge  \frac{C}{q}} 
        \le e^{- \nu (\log N)^{\xi}},
    \end{equation*}
    completing the proof.
\end{proof}

\begin{lemma}\label{lemma:HGH_control}
Let $\mH$ be as in Equation \eqref{eq:def:h_moments} and $z\in D_L$, then on $\Ot(\eta)$ defined in Equation \eqref{eq:def:Ot},
we have for all $i \ne j \in [N]$ that
\begin{equation}
    \left|\sum_{k,l}^{(\sm ij)} h_{ik}\Gtildeij_{kl}(z) h_{lj} \right| \le \frac{C (\log N)^{\xi}}{q} + C(\log N)^{2\xi}\left(\sqrt{\frac{\Im \mtilde(z)}{N\eta}}  + \frac{1}{N\eta} \right)
\end{equation}
with $\xinu$-high probability.
\end{lemma}
\begin{proof}
This is an application of Lemma \ref{lemma:concetration}.
By Equation \eqref{lemma:concentration:quadratic_form} in that lemma, we have that with $\xinu$-high probability
\begin{equation}\label{eq:HGH_control:first_step}
    \left|\sum_{k,l}^{(\sm ij)} h_{ik}\Gtildeij_{kl}(z)h_{lj} \right|  
    \le (\log N)^{2\xi} \! \left [\frac{\max_k \abs{\Gtildeij_{kk}(z)}}{q^2} + \frac{\max_{i\ne j} \abs{\Gtildeij_{kl}(z)}}{q} + \!\left(\! \frac{1}{N^2} \sum_{k,l}^{(\sm ij)} \abs{\Gtildeij_{kl}}^2 \!\right)^{\! \!1/2}\right] .
\end{equation}
Note that on $\Ot(\eta)$, we have $ \max_{i\ne j} \abs{G_{kl}^{
(\sm ij)}} \le C\max_{i\ne j} \abs{G_{kl}}$ by Lemma \ref{lemma:minor_resolvents}.
By the definition of $q$ in Equation \eqref{eq:def:q}, we have $q \ge (\log N)^{3\xi}$.
Hence, on $\Ot(\eta)$, substituting this bound for $q$ into Equation \eqref{eq:HGH_control:first_step}, using the fact that $\max_i \abs{G_{ii}} \le C$ on $\Ot(\eta)$, and distributing,
\begin{equation}\label{eq:HGH_control:second}
    \left|\sum_{k,l}^{(\sm ij)} h_{ik}\Gtildeij_{kl}(z)h_{lj} \right|  \le  \frac{C(\log N)^{2\xi}}{q(\log N)^{3\xi}} + \frac{C(\log N)^{2\xi}\max_{i\ne j} \abs{G_{ij}}}{q} + (\log N)^{2\xi}\left(\frac{1}{N^2} \sum_{i \ne j} \abs{\Gtilde_{ij}(z)}^2 \right)^{1/2}
\end{equation}
with $\xinu$-high probability.
Notice by Ward's Lemma (Equation \eqref{eqn:lemma:Wards}), the summation in Equation \eqref{eq:HGH_control:second} can be expressed as
\begin{equation}\label{eq:HGH_control:third}
\sum_{k,l}^{(\sm ij)} \abs{\Gtildeij_{kl}(z)}^2 
= \frac{1}{\eta}\sum_k^{(\sm ij)}\Im \Gtildeij_{kk}(z).
\end{equation}
Now notice that by using Equation \eqref{eq:resolvent:b} Lemma \ref{lemma:resolvent_identities} twice and using Lemma \ref{lemma:bounded_resolvents},
we have
\begin{equation}\label{eq:HGH_control:fourth}
    \Gtildeij_{kk}(z) 
    = \Gtilde_{kk}(z) + \frac{\abs{\Gtilde_{ik}(z)}^2}{\Gtilde_{ii}(z)} +  \frac{\abs{\Gtilde_{jk}(z)}^2}{\Gtilde_{jj}(z)} 
    = \Gtilde_{kk}(z) + O(\Lo^2(z)).
\end{equation} 
Substituting Equation \eqref{eq:HGH_control:fourth} into \eqref{eq:HGH_control:third} and dividing by $N^2$ yields
\begin{equation}\label{eq:HGH_control:fifth}
     \frac{1}{N^2}\sum_{k,l}^{(\sm ij)} \abs{\Gtildeij_{kl}(z)}^2 
     =  \frac{1}{N^2\eta}\sum_k^{(\sm ij)}\Im \Gtildeij_{kk}(z) 
     \le \frac{1}{N\eta} \left( \Im \mtilde(z) + C\Lo^2(z)\right),
\end{equation}
where the inequality follows from the fact that
\begin{equation*}
    \mtilde(z) = \frac{1}{N} \Tr \Gtildematrix(z).
\end{equation*}
Substituting Equation \eqref{eq:HGH_control:fifth} into Equation \eqref{eq:HGH_control:second} yields that with $\xinu$-high probability
\begin{equation*} %\label{eq:HGH_control:sixth}
    \left|\sum_{k,l}^{(\sm ij)} h_{ik}\Gtildeij_{kl}(z)h_{lj} \right|  \le \frac{C}{q(\log N)^{\xi}} + \frac{(\log N)^{2\xi}\max_{i\ne j} \abs{\Gtilde_{ij}}}{q} + (\log N)^{2\xi}\left(\frac{1}{N\eta} \left( \Im \mtilde(z) + C\Lo^2(z)\right)\right)^{1/2}.
\end{equation*}
Since $\sqrt{a+ b} \le \sqrt{a} + \sqrt{b}$, it follows that
\begin{equation}\label{eq:HGH_control:seventh}
    \left|\sum_{k,l}^{(\sm ij)} h_{ik}\Gtildeij_{kl}(z)h_{lj} \right|  \le  \frac{C}{q(\log N)^{\xi}} + \frac{C(\log N)^{2\xi}\Lo(z)}{q} + (\log N)^{2\xi}\left(\sqrt{\frac{\Im \mtilde(z)}{N\eta}} + \frac{C\Lo(z)}{\sqrt{N\eta}} \right)
\end{equation} 
with $\xinu$-high probability.
Using that $(\log N)^{4\xi} \le \sqrt{N\eta}$ on $D_L$,
we have following Equation \eqref{eq:HGH_control:seventh}
\begin{equation*} 
    \left|\sum_{k,l}^{(\sm ij)} h_{ik}\Gtildeij_{kl}(z)h_{lj} \right|  \le  \frac{C}{q(\log N)^{\xi}} + \frac{C(\log N)^{2\xi}\Lo(z)}{q}  + \sqrt{\frac{\Im \mtilde(z)}{N\eta}} + \frac{C\Lo(z)}{(\log N)^{4\xi}}.
\end{equation*}
Hence, by increasing the value of $\nu$ by a constant factor, %
we have that with $\xinu$-high probability, it holds uniformly over all $i\ne j\in [N]$ that
\begin{equation*}
    \left|\sum_{k,l}^{(\sm ij)} h_{ik}\Gtildeij_{kl}(z)h_{lj} \right|  \le  \frac{C}{q(\log N)^{\xi}} + \frac{C(\log N)^{2\xi}\Lo(z)}{q}  + \sqrt{\frac{\Im \mtilde(z)}{N\eta}},
\end{equation*}
as we set out to show.
\end{proof} %KDL checked Jan 6 2026, 23:20

\begin{lemma}\label{lemma:bounded_resolvents} % This holds for any i \in [N] and z \in D_L, given the event $\Ot(\eta)$ happens. All the bounds used in this proof are non-probabilistic.
Recall the definitions of $\Lo(z)$ and $\Ld(z)$ from Equation~\eqref{eq:def:Ld-Lo},
\begin{equation*}
\Lo(z) := \max_{i\ne j} |\Gtilde_{ij}(z)|, \quad \Ld(z) := \max_{i} |\Gtilde_{ii}(z) - \msc(z)|.
\end{equation*}
On the event $\Ot(\eta)$ given in Equation \eqref{eq:def:Ot}, 
we have for any $i\in [N]$ and $z \in D_L$
\begin{equation} \label{eq:bounded_resolvents:goal1}
    c \le \abs{\Gtilde_{ii}(z)} \le C .
\end{equation}
Further, for any $\T \subset [N]$ with $\abs{\T} \le \tau < N$, where $\tau$ is constant with respect to $N$,
 \begin{equation} \label{eq:bounded_resolvents:goal2}
      c \le \abs{\Gtildet_{ii}(z)} \le C. 
 \end{equation}
\end{lemma}
\begin{proof}
Notice that on $\Ot(\eta)$, since $\Ld(z) \ge 0$, we have
\begin{equation*}
  \Lo(z) \le \Ld(z) + \Lo(z) \le (\log N)^{-\xi} ,
\end{equation*}
so that for any $i \in [N]$,
\begin{equation*}
  -(\log N)^{-\xi} \le \Gtilde_{ii}(z) - \msc(z) \le (\log N)^{-\xi}.
\end{equation*}
Adding $\msc(z)$ to both sides, we obtain
\begin{equation}\label{eq:appendix:bounded_resolvents_3}
\msc(z) -(\log N)^{-\xi} \le \Gtilde_{ii}(z) \le (\log N)^{-\xi} + \msc(z).
\end{equation}
For $z \in D_L$, Lemma 3.2 in \cite{Erdos} asserts that $C^{-1} \le |\msc (z)| \le C$. Using this fact in Equation \eqref{eq:appendix:bounded_resolvents_3} yields
\begin{equation}\label{eq:appendix:bounded_resolvents_4}
c \le \abs{\Gtilde_{ii}(z)} \le C,
\end{equation}
proving the claim in Equation~\eqref{eq:bounded_resolvents:goal1}. 

For the claim in Equation~\eqref{eq:bounded_resolvents:goal2}, suppose we have $k \in \T$.
We apply Equation \eqref{eq:resolvent:b}, from which we obtain
\begin{equation}\label{eq:appendix:bounded_resolvents_intermediate}
    \Gtilde_{ii}(z) - \frac{\abs{\Gtilde_{ik}(z)}^2}{\Gtilde_{kk}(z)} = \Gtilde_{ii}^{(k)}(z).
\end{equation}
Repeatedly applying Equation \eqref{eq:resolvent:b} to \eqref{eq:appendix:bounded_resolvents_intermediate} for the remaining indices in $\T$ yields
\begin{equation}\label{eq:appendix:bounded_resolvents_5}
    \Gtilde_{ii}(z) - \sum_{k \in \T} \frac{\abs{\Gtilde_{ik}(z)}^2}{\Gtilde_{kk}(z)} = \Gtildet_{ii}(z).
\end{equation}
Applying Equation \eqref{eq:appendix:bounded_resolvents_5} to Equation \eqref{eq:appendix:bounded_resolvents_3} and recalling $\msc(z)$ is lower and upper bounded by constants from Lemma 3.2 in \cite{Erdos}, we see 
\begin{equation}\label{eq:appendix:bounded_resolvents_6} 
\msc(z) -(\log N)^{-\xi} - \sum_{k \in \T} \frac{\abs{\Gtilde_{ik}(z)}^2}{\Gtilde_{kk}(z)} \le \Gtilde_{ii}^{(k)}(z) \le (\log N)^{-\xi} + \msc(z) - \sum_{k \in \T} \frac{\abs{\Gtilde_{ik}(z)}^2}{\Gtilde_{kk}(z)}.
\end{equation}
Notice that on $\Ot(\eta)$ and $z \in D_L$, we have from Equation \eqref{eq:appendix:bounded_resolvents_4}
that $c \le \abs{\Gtilde_{kk}(z)} \le C$, so then on $\Ot(\eta)$ ,
\begin{equation}
\frac{\abs{\Gtilde_{ik}(z)}^2}{\Gtilde_{kk}(z)} = O((\log N)^{-2\xi}).
\end{equation}
It then follows from applying this bound to Equation \eqref{eq:appendix:bounded_resolvents_6} that 
\begin{equation}
c \le \abs{\Gtildet_{ii}(z) }\le C,
\end{equation}
proving Equation~\eqref{eq:bounded_resolvents:goal2}.
\end{proof} 

The following lemma provides initial bounds for when $z$ is far from the real line ($3 \ge \eta \ge 2$) and the resolvent entries are easily controlled. In particular, they will be used to show that the event $\Ot(\eta)$ occurs with high probability and hence our main resolvent estimates are bounded by the function $\Phi(z)$. Hence, using that $\Phi(z), \Gtilde_{ij}(z)$, and $\mtilde(z)$ are all Lipschitz-continuous, the resolvent entries can then be bounded for arbitrarily small $\eta$.
\begin{lemma} 
\label{lemma:initial_elementary_resolvent_estimates}
    Let $\tau < N$ be constant with respect to $N$. For $z = E + \iu \eta \in D_L$ and $ 2 \le \eta \le 3$, 
    suppose that $\T \subset [N]$ satisfies $\abs{\T} \le \tau$.
    Then with $(\xi, \nu)$-high probability,
    \begin{equation*} 
        c \le \abs{\Gtildet_{kk}(z)} \le \frac{1}{\eta}, \qquad \abs{\mtilde^{(\sm \T)}(z)} \le \frac{1}{\eta}, \qquad \abs{\Gtildet_{ij}(z)} \le \frac{1}{\eta}, \qquad \text{ and } \qquad
        \max_{i\ne j}\abs{\Gtildet_{ij}(z)} \le C \Lo(z) .
    \end{equation*}
\end{lemma}
\begin{proof}
We first lower bound $\abs{\Gtildet_{kk}(z)}$.
Notice that by taking the imaginary part of the spectral decomposition,
\begin{equation}
    \Im \Gtildet_{kk}(z) = \eta \sum_{\alpha = 1}^{N - \abs{\T}} \frac{\abs{\vhatt_{\alpha}(k)}^2}{(\lambdahat_{\alpha} - E)^2 + \eta^2} \ge \eta\sum_{\alpha \notin \mathcal{S}} \frac{\abs{\vhatt_{\alpha}(k)}^2}{(\lambdahat_{\alpha} - E)^2 + \eta^2}
\end{equation}
where $\calS = \{N-r+1 - \abs{\T}, \dots,N - \abs{\T}\}$. 
Now, by Lemma 4.3 from \cite{Erdos} and Lemma \ref{lemma:weyl}, 
with $\xinu$-high probability, we have $\lambdahat_{\alpha} \le \norm{\mH} \le C$ for all $\alpha \in \calS$. 
Moreover, $E \le 3$ by definition of $D_L$, so that $(\lambdahat_{\alpha} - E)^2 \le C$ for all $\alpha \notin \calS$ with $(\xi, \nu)$-high probability. Then since $\eta$ is bounded above and below by constants, we have for some $c > 0$, 
\begin{equation}\label{eq:elementary_estimates:ImG_lower_bound}
    \Im \Gtildet_{kk}(z)
    \ge \eta\sum_{\alpha \notin \calS}\frac{\abs{\vhatt_{\alpha}(k)}^2}{(\lambdahat_{\alpha} - E)^2 + \eta^2}
    \ge c \sum_{\alpha \notin \calS} \abs{\vhatt_{\alpha}(k)}^2
\end{equation} 
with $\xinu$-high probability.
To lower bound $\sum_{\alpha \notin \calS}\abs{\vhatt_{\alpha}(k)}^2$, notice that since $\submatrix{\mB}{\T}$ is a real, symmetric matrix, the eigenvectors of $\submatrix{\mB}{\T}$ form an orthonormal basis.
Thus, $\sum_{\alpha =1}^{N - \abs{\T}} \abs{\vhatt_{\alpha}(k)}^2 = 1$. We note that the eigenvectors of $\submatrix{\mB}{\T}$ can be used to obtain $N - \abs{\T}$ eigenvectors of $\mB^{[\T^c]}$ by filling in zeros at the indices given by $\T$.
It follows that for each $\alpha \in \calS$, corresponding to the leading eigenvalues of $\mB$,
\begin{equation}\label{eq:initial_estimates:subspace_difference}
\begin{aligned}
     \norm{\bvhatt_{\alpha - \abs{\T}}(\bvhat^{(\sm \T)}_{\alpha - \abs{\T}})^\top - \submatrix{\bv_{\alpha}}{\T}(\submatrix{\bv_{\alpha}}{\T})^\top}_{F}^2 
     &\le \norm{\bvhatt_{\alpha - \abs{\T}}(\bvhatt_{\alpha - \abs{\T}})^\top - \submatrix{\bv_{\alpha}}{\T}(\submatrix{\bv_{\alpha}}{\T})^\top}_{F}^2 + \norm{ \bv^{[\T]}_{\alpha}(\bv^{[\T]}_{\alpha})^\top}_F^2 \\
     &= \norm{\bvhat^{[\T^c]}_{\alpha}(\bvhat^{[\T^c]}_{\alpha})^\top- \bv_{\alpha}\bv_{\alpha}^\top}_{F}^2
\end{aligned}
\end{equation}
By Lemma \ref{lemma:subspace_bound}, with $\xinu$-high probability,
\begin{equation*}
    \norm{\bvhat^{[\T^c]}_{\alpha}(\bvhat^{[\T^c]}_{\alpha})^\top- \bv_{\alpha}\bv_{\alpha}^\top}_{F}^2 = 2  \norm{\left( \mI - \bvhat^{[\T^c]}_\alpha (\bvhat^{[\T^c]}_\alpha)^\top \right)^\top\bv_{\alpha}}_F^2
    \le \frac{C}{N} .
\end{equation*}
Applying this to Equation \eqref{eq:initial_estimates:subspace_difference},
\begin{equation} \label{eq:initial_estimates:refThisOne}
     \norm{\bvhatt_{\alpha - \abs{\T}}(\bvhatt_{\alpha - \abs{\T}})^\top - \submatrix{\bv_{\alpha}}{\T}(\submatrix{\bv_{\alpha}}{\T})^\top}_{F}^2 \le \frac{C}{N}
\end{equation}
with $\xinu$-high probability.
In particular, the $(i,j) \in \left([N] \sm \T\right)^2$ entries of $\bvhatt_{\alpha}(\bvhatt_{\alpha})^\top -  \submatrix{\bv_{\alpha}}{\T}(\submatrix{\bv_{\alpha}}{\T})^\top$ are given by $\vhat_{i\alpha}\vhat_{j\alpha} - v_{i\alpha}v_{j\alpha}$, and thus
\begin{equation*}
    \sum_{i,j \in [N] \sm \T} (\vhat_{i\alpha}\vhat_{j\alpha} - v_{i\alpha}v_{j\alpha})^2 = \norm{\bvhatt_{\alpha}(\bvhatt_{\alpha})^\top -  \submatrix{\bv_{\alpha}}{\T}(\submatrix{\bv_{\alpha}}{\T})^\top}_{F}^2 .
\end{equation*}
Using the trivial upper bound
\begin{equation*}
    \sum_{j \in [N] \sm \T} (\vhat_{j\alpha}\vhat_{j\alpha} - v_{j\alpha}v_{j\alpha})^2 
    \le  \sum_{i,j \in [N] \sm \T} (\vhat_{i\alpha}\vhat_{j\alpha} - v_{i\alpha}v_{j\alpha})^2,
\end{equation*}
we have for $k \in [N] \sm \T$
\begin{equation*}
 (\vhat_{k\alpha}^2 - v_{k\alpha}^2)^2 \le \sum_{i,j \in [N] \sm \T} (\vhat_{i\alpha}\vhat_{j\alpha} - v_{i\alpha}v_{j\alpha})^2.
\end{equation*}
Hence, applying this bound to Equation \eqref{eq:initial_estimates:refThisOne} we have with $\xinu$-high probability
\begin{equation*}
    (\vhat_{k\alpha}^2 - v_{k\alpha}^2)^2 \le \frac{C}{N} .
\end{equation*}
Taking square roots, we have that with $\xinu$-high probability,
\begin{equation}
    \abs{\vhat_{k\alpha}^2 - v_{k\alpha}^2} \le \frac{C}{\sqrt{N}} .
\end{equation} 
By Assumption \eqref{assumption:deloc},
\begin{equation*}
    \vhat_{k\alpha}^2 \le \frac{C}{\sqrt{N}} + v_{k\alpha}^2 \le \frac{C}{\sqrt{N}} + \frac{C\delocsq}{N},
\end{equation*}
from which
\begin{equation*}
    \vhat_{k\alpha}^2 \le \frac{C}{\sqrt{N}} + v_{k\alpha}^2 \le \frac{C}{\sqrt{N}} + \frac{C\delocsq}{N} \le \frac{C}{\sqrt{N}},
\end{equation*}
and after taking square roots, with $\xinu$-high probability,
\begin{equation}
    \abs{\vhat_{k\alpha}} \le \frac{C}{N^{1/4}} .
\end{equation}
So since the eigenvectors of $\Gtildetmatrix(z)$ form an orthogonal matrix, 
\begin{equation}
    \sum_{\alpha = 1}^{N - \abs{\T}} \abs{\vhat_{k\alpha}}^2 = 1,
\end{equation}
we have with $\xinu$-high probability,
\begin{equation}
    \sum_{\alpha \notin \mathcal{S}} \abs{\vhat_{k\alpha}}^2 = 1 - \sum_{\alpha \in \mathcal{S}}\abs{\vhat_{k\alpha}}^2 \ge 1 - \frac{C r}{\sqrt{N}} \ge c
\end{equation}
for some $c > 0$, sufficiently large $N$, and Assumption~\eqref{assumption:rank}.
It then follows from Equation \eqref{eq:elementary_estimates:ImG_lower_bound} 
that
\begin{equation*}
   c \le \Im \Gtildet_{kk}(z) \le \abs{\Gtildet_{kk}(z)}.
\end{equation*}
with $(\xi, \nu)$-high probability.

To upper bound $\abs{\Gtildet_{kk}(z)}$, note that $\submatrix{\mB - z\mI}{\T}$ is normal and its singular values are given by
\begin{equation*}
    \left\{ \sqrt{(\lambdahatt_{\alpha} - E)^2 + \eta^2} : \alpha=1,2,\dots,N - \abs{\T} \right\}.
\end{equation*}
Hence, the singular values of $\Gtildetmatrix = (\submatrix{\mB - z\mI}{\T})^{-1}$ can be bounded according to
\begin{equation*}
    \frac{1}{\sqrt{(\lambdahatt_{\alpha} - E)^2 + \eta^2}} \le \frac{1}{\eta}.
\end{equation*}
It then follows that 
\begin{equation}\label{eq:elementary_estimates:Gtilde_entries_bound}
    \max_{i,j}\abs{\Gtildet_{ij}(z)} \le \norm{\Gtildetmatrix(z)} \le \frac{1}{\eta}. 
\end{equation}

For the upper bound on $\mtildet(z)$, notice that
\begin{equation*}
    \abs{\mtildet(z)} = \abs{\frac{1}{N} \sum_{i} \Gtildet_{ii}(z)} \le \frac{1}{N} \sum_{i}\abs{ \Gtildet_{ii}(z)} \le \frac{1}{\eta} 
\end{equation*}
where the inequality follows from Equation \eqref{eq:elementary_estimates:Gtilde_entries_bound}. 

To show the last claim that $\max_{i\ne j}\abs{\Gtildet_{ij}(z)} \le C \Lo(z)$, observe that from Equation \eqref{eq:appendix:minor_resolvents_1} in the proof of Lemma \ref{lemma:minor_resolvents} and Equation \eqref{eq:elementary_estimates:Gtilde_entries_bound} and recalling $1/\eta \le 1/2$, 
we have
\begin{equation}\label{eq:elementary_estimates:minor_resolvents_adaptation}
\begin{aligned}
   \max_{i\ne j} \abs{\Gtilde_{ij}^{(\sm k)}}(z)
   &= \max_{i\ne j} \abs{-\Gtilde_{ij}(z) + \frac{\Gtilde_{ik}(z)\Gtilde_{kj}(z)}{\Gtilde_{kk}(z)}} \le  \max_{i\ne j} \abs{\Gtilde_{ij}(z)} +\max_{i\ne j} \frac{\abs{\Gtilde_{ik}(z)}\abs{\Gtilde_{kj}(z)}}{\abs{\Gtilde_{kk}(z)}} \\
   &\le \max_{i\ne j} \abs{\Gtilde_{ij}(z)} +  C\max_{i\ne j}\abs{\Gtilde_{ij}(z)}^2 .
\end{aligned} \end{equation}
Then take $\T = \emptyset$ in Equation \eqref{eq:elementary_estimates:Gtilde_entries_bound}, so we have
\begin{equation}
\max_{i\ne j} \abs{\Gtilde_{ij}(z)} +  C\max_{i\ne j}\abs{\Gtilde_{ij}(z)}^2 = \max_{i\ne j} \abs{\Gtilde_{ij}(z)}(1 + C\max_{i\ne j} \abs{\Gtilde_{ij}(z)}) \le \max_{i\ne j} \abs{\Gtilde_{ij}(z)} \left(1 + \frac{C}{\eta} \right),
\end{equation}
it follows that
\begin{equation}
   \max_{i\ne j} \abs{\Gtilde_{ij}^{(\sm k)}(z)} \le C\max_{i\ne j}\abs{\Gtilde_{ij}(z)}.
\end{equation}
The claim then follows by the same argument as given in the proof of Lemma~\ref{lemma:minor_resolvents} for the remaining $k'\in \T \sm \{k\}$ by replacing $\Gtildematrix$ with $\Gtildematrix^{(\sm k)}$ in Equation \eqref{eq:elementary_estimates:minor_resolvents_adaptation} and repeating the above. 
\end{proof}

\begin{lemma}\label{lemma:B_control}
   Let $\mB$ be as in Equation \eqref{eq:def:B}. We have with $\xinu$-high-probability,
\begin{equation*}
   \max_{i,j} \abs{B_{ij}} \le  C\frac{r\rhon\delocsq}{\sqrt{N}} + \frac{C}{q} .
\end{equation*}
\end{lemma}
\begin{proof}
Note we can decompose any entry $B_{ij}$ as 
\begin{equation}\label{eq:A_decomp}
    B_{ij} = L_{ij} + H_{ij}.
\end{equation}
We can further decompose $L_{ij}$ by recalling Equation \eqref{eqn:L-entry-bound}
\begin{equation*}
    \mL = \sum_{\alpha = N - r + 1}^N \lambda_{\alpha}\bv_{\alpha}\bv_{\alpha}^\top .
\end{equation*}
By Assumption \eqref{assumption:eigen_bound}, $c\sqrt{N} \le \lambda_{\alpha} \le C\sqrt{N}$ for all $\alpha \in \{N-r+1, N-r+2, \dots, N \}$.
So then for each $i,j \in [N]$, we have
\begin{equation*}
L_{ij}
= \left(\sum_{\alpha = N - r + 1}^N \lambda_{\alpha}\bv_{\alpha}\bv_{\alpha}^\top \right)_{ij}
\le C\sqrt{N}\left(\sum_{\alpha = N - r + 1}^N\bv_{\alpha}\bv_{\alpha}^\top \right)_{ij}.
\end{equation*}
By Assumption \eqref{assumption:deloc}, $\bv_i\bv_i^\top$ has entries of order $1/N$, whence 
\begin{equation*}
L_{ij} = O\left( \frac{r\rhon\delocsq}{\sqrt{N}} \right) .
\end{equation*}
Using this bound for $L_{ij}$ and controlling the entries of $\mH$ with Lemma \ref{lemma:h_control}, applying the triangle inequality to Equation \eqref{eq:A_decomp} yields
\begin{equation*}
\max_{i,j} \abs{B_{ij}} \le 
\max_{i,j} \abs{L_{ij}} + \max_{i,j} \abs{H_{ij}} \le C \frac{r\rhon\delocsq}{\sqrt{N}} + \frac{C}{q}
\end{equation*}
with $\xinu$-high-probability, as we set out to show.
\end{proof} % KDL checked Jan 7, 2026, 00:05

\begin{lemma}[Weyl's Inequality; Theorem 4.3.1 in \cite{Horn_Johnson}]
\label{lemma:weyl}
    Let $\mA, \mB$ be Hermitian whose eigenvalues are ordered as $\lambdamin = \lambda_1 \le \lambda_2 \le \dots \le \lambda_N = \lambdamax$.
    Then
    \begin{equation} \label{eq:weyl:a}
        \lambda_i(\mA + \mB) \le \lambda_{i+j}(\mA) + \lambda_{N-j}(\mB)
    \end{equation}
    where $j = 0, \dots N-i$ for each $i\in [N]$. 
    Similarly,
    \begin{equation} \label{eq:weyl:b}
        \lambda_{i - j+1}(\mA) + \lambda_{j}(\mB) \le \lambda_i(\mA+\mB) 
    \end{equation}
    where $j = 1,2, \dots, i$ for each $i \in [N]$.

    Further, if 
    $\mA, \mB$ are Hermitian matrices, and $\mB$ is rank $r$.
    Then
    \begin{equation*}
    \lambda_i(\mA) \le \lambda_i(\mA + \mB) \le \lambda_{i+r}(\mA) \quad
    \text{ for } 1 \le i \le N-r .
    \end{equation*}
\end{lemma}
%As a corollary to Weyl's inequality,
%\begin{corollary} \label{cor:weyl_corollary}
%Suppose $\mA, \mB$ are Hermitian matrices, and $\mB$ is rank $r$.
%Then
%\begin{equation*}
%\lambda_i(\mA) \le \lambda_i(\mA + \mB) \le \lambda_{i+r}(\mA) \quad
%\text{ for } 1 \le i \le N-r .
%\end{equation*}
%\end{corollary}
%\begin{proof}
%Apply Weyl's inequality, and set $j = r + 1$ in Equation~\eqref{eq:weyl:a} and $j=1$ in Equation~\eqref{eq:weyl:b} for all $i \le n-r$.
%By assumption, $\lambda_i(\mB) = 0$ for all $i \le n-r$, and the claim follows.
%\end{proof}

\end{document}